\documentclass[final,leqno,onefignum,onetabnum]{siamltex}
\usepackage[letterpaper,top=1.5in, bottom=1.5in, left=1in, right=1in]{geometry}

\usepackage{epsfig,epsf,fancybox}
\usepackage{amsmath}
\usepackage{mathrsfs}
\usepackage{amssymb}
\usepackage{amsfonts}
\usepackage{graphicx}
\usepackage{color}
\usepackage[linktocpage,pagebackref,colorlinks,linkcolor=blue,anchorcolor=blue,citecolor=blue,urlcolor=blue,hypertexnames=false]{hyperref}
\usepackage{boxedminipage}
\usepackage{stmaryrd}
\usepackage{multirow}
\usepackage{subfig}
\usepackage{booktabs}
\usepackage{algorithm}
\usepackage{algpseudocode}
\usepackage{cite}
\usepackage{float}
\usepackage{tikz}
\usepackage{xcolor}
\usepackage{tcolorbox}
\usepackage{braket,mathdots}
\usepackage{mathtools}
\usepackage[thinlines]{easytable}
\usepackage{array}

\allowdisplaybreaks[4]
\usepackage{mathtools}

\newcommand{\vnu}{{\boldsymbol{\nu}}}
\newcommand{\vgamma}{{\boldsymbol{\gamma}}}

\newcommand{\vxi}{{\boldsymbol{\xi}}}

\usepackage{mathtools}

\usepackage[normalem]{ulem} 

\newcommand{\vb}{{\mathbf{b}}}

\newcommand{\vd}{{\mathbf{d}}}

\newcommand{\vh}{{\mathbf{h}}}

\newcommand{\vu}{{\mathbf{u}}}
\newcommand{\vv}{{\mathbf{v}}}

\newcommand{\vx}{{\mathbf{x}}}
\newcommand{\vy}{{\mathbf{y}}}
\newcommand{\vz}{{\mathbf{z}}}

\newcommand{\vA}{{\mathbf{A}}}

\newcommand{\vC}{{\mathbf{C}}}

\newcommand{\vG}{{\mathbf{G}}}
\newcommand{\vH}{{\mathbf{H}}}
\newcommand{\vI}{{\mathbf{I}}}
\newcommand{\vJ}{{\mathbf{J}}}

\newcommand{\vM}{{\mathbf{M}}}

\newcommand{\cD}{{\mathcal{D}}}

\newcommand{\cL}{{\mathcal{L}}}

\newcommand{\cP}{{\mathcal{P}}}

\newcommand{\cS}{{\mathcal{S}}}

\newcommand{\cX}{{\mathcal{X}}}

\newcommand{\vareps}{\varepsilon}

\newcommand{\RR}{\mathbb{R}} 
\newcommand{\zz}{^{\top}} 
\newcommand{\vzero}{\mathbf{0}} 

\newcommand{\dist}{\mathrm{dist}}    
\newcommand{\prox}{{\mathbf{prox}}} 
\newcommand{\dom}{{\mathrm{dom}}} 

\newcommand{\st}{\mbox{ s.t. }}

\DeclareMathOperator*{\argmin}{arg\,min} 
\DeclareMathOperator*{\Argmin}{Arg\,min} 

\newcommand{\bc}{\begin{center}}
\newcommand{\ec}{\end{center}}

\newcommand{\bdm}{\begin{displaymath}}
\newcommand{\edm}{\end{displaymath}}

\newcommand{\beq}{\begin{equation}}
\newcommand{\eeq}{\end{equation}}

\newcommand{\bfl}{\begin{flushleft}}
\newcommand{\efl}{\end{flushleft}}

\newcommand{\bt}{\begin{tabbing}}
\newcommand{\et}{\end{tabbing}}

\newcommand{\beqn}{\begin{eqnarray}}
\newcommand{\eeqn}{\end{eqnarray}}

\newcommand{\beqs}{\begin{align*}} 
\newcommand{\eeqs}{\end{align*}}  

\newtheorem{remark}{Remark}[section]
\newtheorem{assumption}{Assumption}

\numberwithin{equation}{section}
\numberwithin{theorem}{section}

\begin{document}
	
	\title{A
		Near-optimal Method for Linearly Constrained Composite Non-convex Non-smooth Problems}
	
	\author{Wei Liu\thanks{\{liuw16, xuy21\}@rpi.edu, Department of Mathematical Sciences, Rensselaer Polytechnic Institute, Troy, NY}, \and  Qihang Lin\thanks{qihang-lin@uiowa.edu, Department of Business Analytics, University of Iowa, Iowa City, IA}, \and Yangyang Xu$^*$}

	\date{\today}
		\maketitle
	
	

	\begin{abstract}
	We study first-order methods (FOMs) for solving \emph{composite nonconvex nonsmooth} optimization with linear constraints. Recently, the lower complexity bounds  of FOMs on finding an ($\varepsilon,\varepsilon$)-KKT point of the considered problem is established in \cite{liu2025lowercomplexityboundsfirstorder}. However, optimization algorithms that achieve this lower bound had not been developed.
	In this paper, we propose an inexact proximal gradient method, where subproblems are solved using a recovering primal-dual procedure.
	Without making the bounded domain assumption, we establish that the oracle complexity of the proposed method, for finding an ($\varepsilon,\varepsilon$)-KKT point of the considered problem, matches the lower bounds up to a logarithmic factor. Consequently, in terms of the complexity, our algorithm outperforms all existing methods.  We demonstrate the advantages of our proposed algorithm over the (linearized) alternating direction method of multipliers and the (proximal) augmented Lagrangian method in the numerical experiments.
\end{abstract}

\vspace{0.3cm}


	\section{Introduction}\label{sec:intro}
First-order methods (FOMs) have attracted increasing attention due to their efficiency in solving large-scale problems.
In this paper,  we explore this topic for problems with a composite nonconvex nonsmooth objective function and linear constraints, formulated as 
\begin{equation}\label{eq:model}
	\begin{aligned}
		\min_{ \vx\in\RR^d} \,\,& F_0( \vx):= f_0( \vx) + g( \overline{\vA} \vx +\overline{\vb}), \\ 
		\st & \vA \vx +  \vb=\mathbf{0},
	\end{aligned}
	\tag{P}
\end{equation}	
where  $\vA\in\RR^{n\times d}$, $\vb\in\RR^{n}$, $\overline{\vA}\in \RR^{\overline{n}\times d}$, $\overline{\vb}\in\RR^{\overline{n}}$, $f_0:\mathbb{R}^d \rightarrow \mathbb{R}$ is smooth, and $g$ is a closed convex function. This model is widely used in problems requiring a specific sparsity structure in the desired solution,  such as  lasso and group lasso problems~\cite{tibshirani1996regression,yuan2006model}.
We make the following  structural assumptions throughout this paper.
\begin{assumption}The following statements hold.
	\label{assume:problemsetup}
	\begin{itemize}
		\item[\textnormal{(a)}]  $\nabla f_0$ is $L_f$-Lipschitz continuous,  i.e., 
		$
		\left\|\nabla f_0\left(\vx\right)- \nabla f_0\left(\vx'\right)\right\| \leq L_f\left\|\vx-\vx'\right\|,\quad \forall\, \vx,\,\vx'\in \mathbb{R}^d.
		$ 
		\item[\textnormal{(b)}] $\inf_{\vx} F_0(\vx) > -\infty$.
		\item[\textnormal{(c)}]	
		${g}: \mathbb{R}^{\overline n} \rightarrow \mathbb{R} \cup\{+\infty\}$ is a proper lower semicontinuous convex function but potentially nonsmooth.  
		\item[\textnormal{(d)}]	 $\vA$ has  a full-row rank. 
		There exists a feasible point $\vx$ of~\eqref{eq:model} such that $\overline{\vA} \vx +\overline{\vb}$ is in the relative interior of $\dom( g)$. 
		\item[\textnormal{(e)}] $g$ is $l_g$-Lipschitz continuous over $dom(g)$ for $l_g>0$ .
	\end{itemize}
\end{assumption}
The full-row rankness of $\vA$ is to ensure the existence of KKT point and the necessity of the KKT conditions for local optimality.
We assume the following first-order oracle accessible at any given input:  
\begin{align}
	&\mathrm{ORACLE}(\vx, \vy, \vz, \eta)\mbox{ returns } \big(\nabla f_0(\vx), \bar\vA\vx, \vA\vx, \bar\vA\zz \vy, \vA\zz \vz, \vb, \overline{\vb}, \prox_{\eta g}(\vy), \prox_{\eta g}(\bar\vA\vx + \bar\vb)\big),\label{oracle}\\
	&\hspace{8cm} \forall\, \vx\in\RR^d, \vy\in\RR^{\bar n}, \vz\in\RR^n, \eta >0\nonumber
\end{align}
with $
\prox_{\eta g}(\vx):=\arg\min_{ \vx'} \left\{ \textstyle g(\vx')+\frac{1}{2\eta}\|\vx'-\vx\|^2 \right\}$ for any $\vx\in \RR^d$  and $\eta > 0.
$
The complexity of first-order methods (FOMs) is frequently evaluated through their oracle complexity, which quantifies the number of oracle calls needed to obtain an approximate stationary point. This metric is critical for understanding the practical efficiency of algorithms, especially for solving constrained nonconvex nonsmooth problems. Numerous studies have explored the oracle complexity of FOMs within this context, providing theoretical guarantees across diverse formulations \cite{bian2015linearly, haeser2019optimality, kong2019complexity, jiang2019structured, melo2020iteration1, zhang2020proximal, o2021log, liu2022linearly, zhang2022global}.

By the oracle in \eqref{oracle}, \cite{liu2025lowercomplexityboundsfirstorder} shows that the lower bound to achieve an ($\vareps,\vareps$)-KKT point (see Definition~\ref{def:eps-pt-P}) of problem~\eqref{eq:model} is  at least
$\Omega({\kappa([\overline{\vA};\vA]) L_f \Delta_{F_0}} \vareps^{-2}),$ where
$$
[\overline{\vA};\vA]:= \left[
\begin{array}{c}
	\overline{\vA}\\
	\vA
\end{array}
\right],\quad
\kappa(\vA):= \sqrt{\frac{\lambda_{\max}(\vA\vA\zz)}{\lambda^+_{\min}(\vA\vA\zz)}},\text{ and } \Delta_{F_0}=F_0(\vx^{(0)})-\inf_\vx F_0(\vx)
$$
with $\lambda^+_{\min}(\vA\vA\zz)$  and $\lambda_{\max}(\vA\vA\zz)$ being the smallest positive and the largest eigenvalues of $\vA\vA\zz$, respectively. 
However,  it remains unclear whether an algorithm can find an ($\vareps,\vareps$)-KKT point of problem~\eqref{eq:model} by $\mathcal{O}({\kappa([\overline{\vA};\vA]) L_f \Delta_{F_0}} \vareps^{-2})$ first-order oracles as defined in \eqref{oracle}.  This paper aims to close this gap.

Bridging the gap between upper and lower complexity bounds is essential for two reasons. First, achieving an optimal algorithm for composite nonconvex nonsmooth problems with linear constraints would unify existing results, extending optimality beyond specific subclasses like proximal gradient methods for unconstrained problems~\cite{nesterov2012make} and projected gradient methods for linearly constrained smooth problems~\cite{sun2019distributed, carmon2021lower}. The key challenge lies in simultaneously handling nonconvex regularization and linear constraints. Second, an optimal algorithm would ensure robustness to problem parameters such as the Lipschitz constant $L_f$ and the condition number, addressing the performance degradation that limits existing methods.

\subsection{Algorithm framework}\label{sec:alg}
The FOM that we will design is based on the framework of an inexact Proximal Gradient method, with its subproblems solved by a Recovering Primal-Dual  step (\mbox{PG-RPD}). 

First, we apply a variable splitting technique, by introducing a new variable $\vy\in\mathbb{R}^{\overline n}$, to reformulate problem~\eqref{eq:model} to its Splitting Problem
\begin{equation}
	\label{eq:model-spli}
	\begin{aligned}		
		\min_{ \vx\in\RR^d, \vy\in\RR^{\overline n}} &~F(\vx,\vy):=f_0(\vx) +  g( \vy),\\
		\st \,\,\,\,\,&\vA \vx + \textbf{b} =0,\,\, \vy = \overline{\vA} \vx + \overline{\textbf{b}}.
	\end{aligned}		\tag{SP}
\end{equation}
Next, we propose 
generating a sequence $\{(\vx^{(k)},\vy^{(k)})\}$ by 
\begin{equation}
	\label{eq:subpga-2}
	\begin{aligned}
		(\vx^{(k+1)},\vy^{(k+1)})\approx\argmin_{\vx, \vy} \,\,&\underbrace{\left\langle \nabla f_0(\vx^{(k)}), \vx-\vx^{(k)}\right\rangle+\frac{\tau}{2}\left\|\vx-{\vx}^{(k)}\right\|^2}_{:=\overline{f}(\vx)}+ g(\vy)
		\\
		\st \,\,&\vy=
		\overline{\vA}\vx +
		\overline{\vb}, \,\, {\vA}\vx +
		{\vb}=\mathbf{0},
	\end{aligned}
\end{equation}		
for each $k\ge0$, where $\tau \ge L_f$. Solving the minimization problem in~\eqref{eq:subpga-2} exactly can be difficult due to the coexistence of affine constraints and the regularization term. 
%
Hence, we pursue an inexact solution $(\vx^{(k+1)},\vy^{(k+1)})$ to a desired accuracy. We consider the following Lagrangian function of the problem in~\eqref{eq:subpga-2} 
\begin{equation}\label{eq:lag-func}
	\mathcal{L}_k(\vx, \vy, \vz)=\overline{f}(\vx)+ g(\vy)-\vz_1^{\top}\left(\vy-\left(
	\overline{\vA} \vx+
	\overline{\vb} \right)\right)+ \vz_2^{\top}\left(
	{\vA} \vx+
	{\vb} \right),
\end{equation}
where $\vz_1\in \RR^{\overline{n}}$, $\vz_2\in\RR^n$ are dual variables, and $\vz=(\vz_1\zz,\vz_2\zz)\zz$. Let $\mathcal{D}_k$ be the negative Lagrangian dual function
\footnote{For ease of discussion, we formulate the Lagrangian dual problem into a minimization one by negating the dual function.}, i.e., 
$$
\mathcal{D}_k(\vz):= - \min_{\vx, \vy}\cL_k(\vx, \vy, \vz) = {\frac{1}{2\tau}\|\overline{\vA}\zz\vz_1 +\vA\zz\vz_2 +\nabla f_0(\vx^{(k)})-\tau \vx^{(k)}\|^2} +  g^{\star} (\vz_1)- \vz_1^{\top}\overline{\vb}- \vz_2^{\top}\vb, \forall\,\vz,
$$
where $  g^{\star}$ is the convex conjugate function of $ g$, i.e., $ g^{\star}(\vz_1)=\max_{\vy}\{\vy^\top\vz_1- g(\vy)\}$. 
We then define
\begin{equation}
	\label{eq:sublb-2}
	\Omega^{(k+1)}:=\Argmin_{\vz} \mathcal{D}_k(\vz)\quad \text{ and }\quad 	\mathcal{D}_k^*:=\min_{\vz} \mathcal{D}_k(\vz).
\end{equation}

Note that $\mathbf{prox}_{\eta g^{\star}}(\vz_1)= \vz_1-\mathbf{prox}_{\eta g(\cdot/\eta)}(\vz_1)= \vz_1-\eta\mathbf{prox}_{\eta^{-1} g(\cdot)}(\vz_1/\eta)$~\cite{rockafellar1970convex}. Thus 
the proximal operator $\mathbf{prox}_{\eta g^{\star}}(\vz)$ can be calculated easily if $\mathbf{prox}_{\eta g(\cdot/\eta)}(\vz_1)$ can be.  
We find a near-optimal point $\vz^{(k+1)}=((\vz_1^{(k+1)})\zz, (\vz_2^{(k+1)})\zz)\zz$ of the convex problem $\min_{\vz} \mathcal{D}_k(\vz)$ such that  
\begin{equation}
	\label{eq:boundz-2}
	\vz_1^{(k+1)}\in \dom( g^{\star})\quad\text{ and }\quad
	\dist(\vz^{(k+1)}, \Omega^{(k+1)})\leq \mu_k, 
\end{equation} 
where $\mu_k>0$ is a small number (to be specified). We will show that under certain conditions (i.e., either Assumption~\ref{ass:dual2} or Assumption~\ref{ass:polyhedralg}), an \emph{accelerated proximal gradient} (APG) method, e.g., the one in~\cite{nesterov2013gradient}, or a restarted APG method, can approach $\Omega^{(k+1)}$ in~\eqref{eq:sublb-2} at a linear rate. Though the second condition in \eqref{eq:boundz-2} cannot be \emph{directly} verified, it can be guaranteed by using the (restarted) APG with \emph{computable} inputs. Thus, a point $\vz^{(k+1)}$ satisfying \eqref{eq:boundz-2} can be found efficiently; see Theorems~\ref{thm:inner-iter-dual2}--\ref{thm:complesub}, along with the subsequent remarks. 

Finally, given $\sigma>0$, we obtain a primal solution from $\vz^{(k+1)}$ by the recovering steps:
\begin{eqnarray}
	\label{eq:xupdate-2}
	{\vx}^{(k+1)}&=& \vx^{(k)}-\frac{1}{\tau} \left(\overline{\vA}\zz\vz_1^{(k+1)} +\vA\zz\vz_2^{(k+1)} +\nabla f_0(\vx^{(k)})\right),\\
	\label{eq:yupdate-2}
	{\vy}^{(k+1)}&=& \prox_{\sigma^{-1}  g}\left(\sigma^{-1} \vz_1^{(k+1)} +\overline{\vA}{\vx}^{(k+1)} +
	\overline{\vb}  \right).
\end{eqnarray}
This procedure is presented in Algorithm~\ref{alg:ipganormal}. 

\begin{algorithm}[H]
	\caption{An inexact proximal gradient method for problem~\eqref{eq:model-spli}}\label{alg:ipganormal}
	\begin{algorithmic}[1]
		\State \textbf{Input:} a feasible initial point $({\vx}^{(0)}, \vy^{(0)})$, $\sigma>0$, $\vareps>0$, and $\tau> L_f$. 
		\State Choose $\{\mu_k\}_{k\ge 0}\subset \RR_{++}$ and let $k\leftarrow 0$.
		\While {an $\vareps$-KKT  point of problem~\eqref{eq:model-spli} is not obtained}
		
		\State Calculate  a near-optimal point $\vz^{(k+1)}$ of problem~$\min_{\vz} \mathcal{D}_k(\vz)$ that satisfies the conditions in~\eqref{eq:boundz-2}. 
		
		\State Set $\vx^{(k+1)}$ by~\eqref{eq:xupdate-2} and $\vy^{(k+1)}$ by~\eqref{eq:yupdate-2}.
		
		\State Let $k\leftarrow k+1$.
		
		\EndWhile
		\State \textbf{Output:} $(\vx^{(k)},\vy^{(k)})$.
	\end{algorithmic}
\end{algorithm}

\subsection{Contributions}
We make significant contributions toward 
achieving (near) optimal complexity for solving nonsmooth, nonconvex optimization problems with linear constraints, in the form of  \eqref{eq:model}. To this end, we propose an inexact proximal gradient algorithm, PG-RPD, featuring a recovering primal-dual  subroutine specifically designed for such structured optimization tasks. 
A key strength of PG-RPD is that every step of the algorithm is explicitly computable, making it fully implementable in practical applications.
In addition,
the convergence of PG-RPD does not rely on restrictive assumptions such as boundedness of the domain, thus broadening its applicability to a wide range of practical scenarios.

We   establish that PG-RPD achieves an oracle complexity of
$
\mathcal{O}({\kappa([\overline{\mathbf{A}}; \mathbf{A}]) L_f \Delta_{F_0}} \vareps^{-2} \log(\delta^{-1}\varepsilon^{-1})\big)
$
for obtaining a 
\( (\delta, \vareps) \)-KKT point of~\eqref{eq:model}.
  Consequently,  {PG-RPD} can produce a solution that is $\delta$-close to an \( \vareps \)-KKT point 
  of problem~\eqref{eq:model} within   the complexity of
$
\mathcal{O}({\kappa([\overline{\mathbf{A}}; \mathbf{A}]) L_f \Delta_{F_0}} \vareps^{-2} \log(\vareps^{-1})\big)
$, if $\delta$ is a polynomial of $\vareps$. 
This result matches a known lower complexity bound \cite{liu2025lowercomplexityboundsfirstorder} up to a logarithmic factor and a constant multiplier, making the PG-RPD algorithm nearly optimal.
Notably, our complexity analysis demonstrates an improvement over existing methods by significantly reducing the dependence on the condition number $\kappa([\overline{\vA};\vA])$, particularly evident in the worst-case scenarios used to establish lower complexity bounds in \cite{liu2025lowercomplexityboundsfirstorder}. 

Furthermore, numerical experiments illustrate the clear advantages of our proposed PG-RPD algorithm over the (linearized) ADMM from \cite{melo2017iteration3} and the (proximal) ALM from \cite{rockafellar1976augmented}. The performance improvement is particularly pronounced in problems with large condition numbers, underscoring the robustness and efficiency of our method in ill-conditioned scenarios.
\subsection{Notations and definitions} 
For any $a\in\RR$, $\lceil a\rceil$ refers to the smallest integer that is no less than~$a$.  
We use 
$\mathbf{0}$ to represent an all-zero vector when its dimension is clear from the context. 
A vector $\vx$ is said to be $\delta$-close to another vector $\widehat\vx$ if $\|\vx -\widehat\vx\| \le \delta$. For any set $\cX$, we denote $\iota_\cX$ as its indicator function, i.e., $\iota_\cX(\vx) = 0$ if $\vx\in\cX$ and $+\infty$ otherwise.	 
$\partial g$ denotes the subdifferental of a closed convex function $g$. 
\begin{definition}\label{def:eps-pt-P}
	Given $\vareps \ge0$, a point $\vx^*$ is called an $\vareps$-KKT point of~\eqref{eq:model} if for some $\vgamma\in\RR^n$, it holds 
	\begin{equation}
		\label{eq:epsta}
		\begin{aligned}
			&\max\left\{\dist\left(\mathbf{0},\nabla f_0(\vx^*)+\vA\zz \vgamma 
			+\overline\vA\zz \partial g(\overline\vA\vx^*+\overline\vb)  \right), \left\| \vA\vx^*+\vb\right\|
			\right\}\leq \vareps,
		\end{aligned}
	\end{equation}
	and a point $(\vx^*, \vy^*)$ is called an $\vareps$-KKT point of~\eqref{eq:model-spli} if for some $\vgamma_1\in \RR^{\overline{n}}$ and $\vgamma_2\in\RR^n$, it holds 
	\begin{equation}
		\label{eq:kktviofgesub}
		\begin{aligned}
			&	\max\bigg\{ \dist (\vzero, \partial  g(\vy^*) - \vgamma_1), \left\|\nabla f_0(\vx^*) + \overline{\vA}\zz \vgamma_1 + \vA\zz \vgamma_2  \right\|,   \|\vy^*-\overline{\vA}\vx^*-\overline{\vb}\|,  \|\vA\vx^*+\vb\| \bigg\} 
			\leq\vareps.
		\end{aligned}
	\end{equation}	
	When $\vareps = 0$, we simply call $\vx^*$ and $(\vx^*, \vy^*)$ KKT points of problems~\eqref{eq:model} and~\eqref{eq:model-spli}, respectively.	We say that $\overline{\vx}$ is a  $(\delta, \vareps)$-KKT point of~\eqref{eq:model} if it is $\delta$-close to an $\vareps$-KKT point $\vx^*$ of~\eqref{eq:model}.
\end{definition}	


\subsection{Organizations}

The rest of this paper is organized as follows. Section~\ref{sec:relatedwork} provides a review of the related literature. In Section~\ref{sec:outer}, we analyze the worst-case oracle complexity of Algorithm \ref{alg:ipganormal} for solving problems~\eqref{eq:model} and~\eqref{eq:model-spli}. 
Section~\ref{sec:numer} presents the experimental results. Finally, concluding remarks are provided in Section~\ref{sec:conclusion}.

\section{Related work}\label{sec:relatedwork}	
In this section, we review several algorithms for solving problem~\eqref{eq:model} and discuss why they do not match the known lower complexity bound, indicating potential room for improvement.

Penalty-based methods represent an important class of algorithms for problem~\eqref{eq:model} and its variants. For instance, Kong et al.\cite{kong2019complexity} develop a quadratic-penalty accelerated inexact proximal point method with an oracle complexity of $\mathcal{O}(\varepsilon^{-3})$. Lin et al.\cite{lin2022complexity} refine this approach, achieving a complexity improvement to $\mathcal{O}(\varepsilon^{-5/2})$. These results do not match the best-known lower-bound complexity $\Omega({\kappa([\overline{\vA};\vA]) L_f \Delta_{F_0}} \vareps^{-2})$.

The augmented Lagrangian method (ALM) and the alternating direction method of multipliers (ADMM) have also proven effective for these problems. The oracle complexity of ALM for problem~\eqref{eq:model} has been rigorously studied by~\cite{hong2016decomposing, hajinezhad2019perturbed, melo2020iteration1, zhang2022global}. For example, the inexact proximal accelerated ALM proposed by Melo et al.\cite{melo2020iteration1} achieves an oracle complexity of $\mathcal{O}(\varepsilon^{-5/2})$. Notably, Zhang et al.\cite{zhang2022global} improve the complexity to $\mathcal{O}(\varepsilon^{-2})$ under the special case where $g(\mathbf{x})$ is the indicator function of a polyhedral set.
ADMM is particularly effective for problem~\eqref{eq:model-spli} due to its separable structure. Studies on ADMM and its variants for constrained nonconvex optimization include \cite{melo2017iteration, melo2017iteration3, goncalves2017convergence, jiang2019structured, zhang2020proximal, yashtini2022convergence, zhang2022global, hong2016convergence,dahal2023damped}. Among them, \cite{goncalves2017convergence, yashtini2022convergence, jiang2019structured} demonstrate that ADMM can achieve an $\varepsilon$-KKT point with oracle complexity $\mathcal{O}(\varepsilon^{-2})$.

Despite significant progress,  ALM and ADMM do not achieve the known lower-bound complexity $\Omega(\kappa([\overline{\mathbf{A}}; \mathbf{A}]) L_f \Delta_{F_0} \varepsilon^{-2})$ for the problem class~\eqref{eq:model}. 
These two methods often rely on restrictive assumptions or exhibit sensitivity to problem parameters. For example, the oracle complexity of ADMM in~\cite{yashtini2022convergence} depends on the Kurdyka-Łojasiewicz (KŁ) coefficient and full row-rank of $[\overline{\vA}; \vA]$, making it difficult to compare directly with the lower bound. Similarly, the complexity bounds derived in~\cite{melo2017iteration, goncalves2017convergence, jiang2019structured} are specific to problems with separable structures or particular constraints (e.g., $\overline{\mathbf{A}} = \mathbf{I}_d$ or $g \equiv 0$).
Additionally, Zhang and Luo~\cite{zhang2022global} analyze the complexity of ALM under the assumption that $g$ is the indicator function of a polyhedral set, achieving a bound of $\mathcal{O}(\widehat{\kappa}^2 L_f^3 \Delta_{F_0} \varepsilon^{-2})$. When $g \equiv 0$, the complexity reduces to $\mathcal{O}(\kappa^2(\mathbf{A}) L_f^3 \Delta_{F_0} \varepsilon^{-2})$. We will present two detailed comparisons in Appendix \ref{sec:matching}.

The proximal point method has been a fundamental approach in optimization since 1976~\cite{rockafellar1976augmented}, and its oracle complexity for solving problem~\eqref{eq:model} has been analyzed in several recent works~\cite{li2021rate, lin2022complexity, kong2022iteration, zhu2023optimal}. Specifically, the methods in~\cite{kong2022iteration, li2021rate} achieve an $\varepsilon$-KKT point with a complexity of $\mathcal{O}(\varepsilon^{-3})$, while the approach in~\cite{lin2022complexity} follows a similar framework but improves the complexity to $\mathcal{O}(\varepsilon^{-5/2})$. The work in~\cite{zhu2023optimal} solves the dual problem of the proximal point subproblems at each iteration and attains an $(\varepsilon, \varepsilon)$-KKT point in  
$
\mathcal{O}(\kappa([\overline{\mathbf{A}}; \mathbf{A}]) L_f \Delta_{F_0} \varepsilon^{-2})
$
oracles. Their result relies on the assumption that the strongly convex problem  
\[
\min _{\mathbf{x}}\left\{\frac{1}{2}\|\mathbf{x}-\hat{\mathbf{x}}\|^2+\alpha \cdot g(\overline{\vA}\mathbf{x} + \overline\vb) \mid \mathbf{A x}+\mathbf{b}=\mathbf{0}\right\}
\]
can be efficiently solved, requiring only $\mathcal{O}(1)$ oracle calls as defined in \eqref{oracle}, for any given $\hat{\mathbf{x}}$ and $\alpha > 0$. However, the validity of this assumption is uncertain, which means that their complexity result does not fully resolve the problem, i.e., it does not match the lower bound.
For instance, if a projected gradient method is used, the number of oracles needed to reach a high-accuracy solution could be substantial, potentially resulting in numerous (or even infinite) calls to ${\mathbf{A}}$ and $\mathbf{b}$.

\section{Convergence and complexity analysis of Algorithm \ref{alg:ipganormal}} \label{sec:outer}

In this section, we establish the oracle complexity of Algorithm~\ref{alg:ipganormal}.
We assume either Assumption \ref{ass:dual2} or  Assumption \ref{ass:polyhedralg}.
\begin{assumption}
	\label{ass:dual2}
	$[\overline{\vA}; \vA]$ has a full-row rank.
\end{assumption}
\begin{assumption}
	\label{ass:polyhedralg}
	$g(\vy)=\max\{\vu^{\top}\vy : \vC \vu\leq\vd, \vu\in\RR^{\overline{n}}\}$ for some $\vC$ and $\vd$. 
\end{assumption}

We note that Assumption~\ref{ass:dual2} is a natural assumption when considering problem~\eqref{eq:model-spli}, as in ADMM methods. This ensures that the constraint set of problem~\eqref{eq:model-spli} satisfies the existence and necessity conditions for KKT points. Additionally, the
$l_1$ and $l_{\infty}$ regularizers satisfy Assumption~\ref{ass:polyhedralg}. Moreover, if \( g \) satisfies 
Assumption~\ref{ass:polyhedralg}, then its dual function \( g^\star \) satisfies
 Assumption~\ref{ass:polyhedral2} (given later), as  
$
g^\star(\vz) = \iota_{\vC \vz \leq \vd}(\vz).
$
The latter assumption is utilized in~\cite{zhang2022global, zhang2020proximal} to  establish the $O(\varepsilon^{-2})$ complexity results.

\subsection{Number of outer iterations for finding a stationary point}
To characterize the convergence property of Algorithm~\ref{alg:ipganormal}, we need the following two lemmas. 
\begin{lemma}
	\label{lem:barxyz-2}
	Suppose that Assumption~\ref{assume:problemsetup} holds.	For any $\sigma>0$ and any $\overline{\vz}^{(k+1)}\in\Omega^{(k+1)}$, let  $(\overline{\vx}^{(k+1)}, \overline{\vy}^{(k+1)})$ be the optimal solution of the strongly convex problem 
	\begin{equation}
		\label{eq:subxy-2}
		\min_{\vx, \vy }\left\{\mathcal{L}_k(\vx, \vy, \overline{\vz}^{(k+1)})+\frac{\sigma}{2}\|{\vy}-
		(\overline{\vA}{\vx} +
		\overline{\vb})\|^2+ \frac{\sigma}{2}\|
		{\vA}{\vx} +
		{\vb}\|^2\right\},
	\end{equation}
	where $\cL_k$ and $\Omega^{(k+1)}$ are defined in~\eqref{eq:lag-func} and~\eqref{eq:sublb-2}, respectively.			Then it holds that for any $k\ge0$,
	\begin{equation}
		\label{eq:feasiblexy-2}
		\overline{\vy}^{(k+1)}-
		(\overline{\vA}\overline{\vx}^{(k+1)} +
		\overline{\vb})=\mathbf{0},\quad {\vA}\overline{\vx}^{(k+1)} +
		{\vb}=\mathbf{0},
	\end{equation}
	\begin{equation}
		\label{eq:barx-2}
		\begin{aligned}
			\overline{\vx}^{(k+1)}= \vx^{(k)}-\frac{1}{\tau} \left(\overline{\vA}\zz\overline{\vz}_1^{(k+1)}+\vA\zz\overline{\vz}_2^{(k+1)}+\nabla f_0(\vx^{(k)})\right),
		\end{aligned}
	\end{equation}
	and
	\begin{equation}
		\label{eq:bary-2}
		\,\,\,\overline{\vy}^{(k+1)}= \prox_{\sigma^{-1} g}\left(\sigma^{-1} \overline{\vz}_1^{(k+1)} +\overline{\vA}\overline{\vx}^{(k+1)} +
		\overline{\vb}  \right).
	\end{equation}
\end{lemma}
\begin{proof}
	Let $(\widehat{\vx}^{(k+1)}, \widehat{\vy}^{(k+1)})$ be an optimal solution of the problem in~\eqref{eq:subpga-2}. 
	Under Assumption~\ref{assume:problemsetup}(d), the strong duality holds~\cite[Section 5.2.3]{boyd2004convex}.  Then the optimal objective value of the minimization problem in~\eqref{eq:subpga-2} is 
	$\max_{\vz} \mathcal{L}_k(\widehat{\vx}^{(k+1)}, \widehat{\vy}^{(k+1)}, {\vz})$. By the definition of $\overline{\vz}^{(k+1)}$ and the strong duality,  it holds that
	\begin{equation*}
		\mathcal{L}_k(\widehat{\vx}^{(k+1)}, \widehat{\vy}^{(k+1)}, \overline{\vz}^{(k+1)})\geq \min_{\vx, \vy} \mathcal{L}_k(\vx, \vy, \overline{\vz}^{(k+1)})=\max_{\vz} \mathcal{L}_k(\widehat{\vx}^{(k+1)}, \widehat{\vy}^{(k+1)}, {\vz})\geq  \mathcal{L}_k(\widehat{\vx}^{(k+1)}, \widehat{\vy}^{(k+1)}, \overline{\vz}^{(k+1)}).
	\end{equation*}
	Hence, the inequalities above must hold with equalities. 
	Thus
	$
	(\widehat{\vx}^{(k+1)}, \widehat{\vy}^{(k+1)})\in\Argmin_{\vx, \vy} \mathcal{L}_k(\vx, \vy, \overline{\vz}^{(k+1)}).
	$
	By this fact and also that $(\widehat{\vx}^{(k+1)}, \widehat{\vy}^{(k+1)})$ solves the problem in~\eqref{eq:subpga-2}, we obtain
	\begin{eqnarray*}
		\mathcal{L}_k(\widehat{\vx}^{(k+1)}, \widehat{\vy}^{(k+1)}, \overline{\vz}^{(k+1)}) 
		&=&\min_{\vx, \vy} \mathcal{L}_k(\vx, \vy, \overline{\vz}^{(k+1)})
		\leq  \mathcal{L}_k(\overline{\vx}^{(k+1)}, \overline{\vy}^{(k+1)}, \overline{\vz}^{(k+1)})\\
		&\leq&\mathcal{L}_k(\overline{\vx}^{(k+1)}, \overline{\vy}^{(k+1)}, \overline{\vz}^{(k+1)}) +\frac{\sigma}{2}\|\overline{\vy}^{(k+1)}-(\overline{\vA}{\overline{\vx}^{(k+1)}} +\overline{\vb})\|^2+ \frac{\sigma}{2}\|{\vA}{\overline{\vx}^{(k+1)}} +{\vb}\|^2\\
		&=& \min_{\vx, \vy} \left\{\mathcal{L}_k(\vx, \vy, \overline{\vz}^{(k+1)})+\frac{\sigma}{2}\|{\vy}-
		(\overline{\vA}{\vx} +\overline{\vb})\|^2+ \frac{\sigma}{2}\|
		{\vA}{\vx} +{\vb}\|^2\right\}\\		
		&\leq  &\mathcal{L}_k(\widehat{\vx}^{(k+1)}, \widehat{\vy}^{(k+1)}, \overline{\vz}^{(k+1)}) +\frac{\sigma}{2}\|\widehat{\vy}^{(k+1)}-(\overline{\vA}{\widehat{\vx}^{(k+1)}} +\overline{\vb})\|^2+ \frac{\sigma}{2}\|{\vA}{\widehat{\vx}^{(k+1)}} +{\vb}\|^2\\
		&=&\mathcal{L}_k(\widehat{\vx}^{(k+1)}, \widehat{\vy}^{(k+1)}, \overline{\vz}^{(k+1)}),
	\end{eqnarray*}
	where the second equality is by the definition of $(\overline{\vx}^{(k+1)}, \overline{\vy}^{(k+1)})$. Hence all the inequalities above must hold with   equalities. Thus~\eqref{eq:feasiblexy-2} follows and
	$
	(\overline{\vx}^{(k+1)}, \overline{\vy}^{(k+1)})\in\Argmin_{\vx, \vy} \mathcal{L}_k(\vx, \vy, \overline{\vz}^{(k+1)}),
	$
	the optimality condition of which gives 
	\eqref{eq:barx-2} and~\eqref{eq:bary-2}. This completes the proof.
	%
\end{proof}

\begin{remark}
	\label{rem:ipga2}
	The proof of Lemma~\ref{lem:barxyz-2} also implies that  $(\overline{\vx}^{(k+1)}, \overline{\vy}^{(k+1)})$ is an optimal solution of the problem in~\eqref{eq:subpga-2}. The lemma below bounds the inexactness of $(\vx^{(k+1)},\vy^{(k+1)})$ generated in Algorithm~\ref{alg:ipganormal}.
\end{remark}

\begin{lemma}
	\label{lem:boundxyz-2}
	Suppose that Assumption~\ref{assume:problemsetup} holds. Let $\{(\vx^{(k+1)},\vy^{(k+1)})\}_{k\ge 0}$ be generated from Algorithm~\ref{alg:ipganormal}. Denote the vector used to produce $(\vx^{(k+1)},\vy^{(k+1)})$ in~\eqref{eq:boundz-2}-\eqref{eq:yupdate-2} by $\vz^{(k+1)}=(({\vz}_1^{(k+1)})\zz,({\vz}_2^{(k+1)})\zz)\zz$. Let  $(\overline{\vx}^{(k+1)}, \overline{\vy}^{(k+1)})$ be defined as in Lemma~\ref{lem:barxyz-2} with $\overline{\vz}^{(k+1)} = \mathbf{proj}_{\Omega^{(k+1)}}(\vz^{(k+1)})$, and $l_f^k = \|\nabla f_0(\vx^{(k)})\|$. Then  the following inequalities hold for all $k\ge0$:
	\begin{align}
		&\label{eq:lemmaineq1}
		\|\overline{\vx}^{(k+1)}-{\vx}^{(k+1)}\|\leq \frac{1}{\tau}\left\|[\overline{\vA}; \vA]\right\|\mu_k,\quad  \|\overline{\vy}^{(k+1)}-{\vy}^{(k+1)}\|\leq \frac{1}{\tau}\left\|\overline{\vA}\right\|\left\|[\overline{\vA}; \vA]\right\|\mu_k +\sigma^{-1} \mu_k, \\
		&\label{eq:lemmaineq2}
		\|{\vy}^{(k+1)}-
		(\overline{\vA}{\vx}^{(k+1)} +	
		\overline{\vb})\|+ \|{\vA}{\vx}^{(k+1)} +
		{\vb}\| 
		\leq B_1 \mu_k,	\\
		&\label{eq:lemmaineq3}
		\|\vz_1^{(k+1)}\|\leq l_g, \quad \|\vz_2^{(k+1)}\|\leq B_2\mu_k + B_3^k,	\\	
		&\label{eq:lemmaineq4}
		\|\vx^{(k+1)}-\vx^{(k)}\| \leq B^k_4 +  \frac{1}{\tau} \left\|[\overline{\vA}; \vA]\right\|B_2 \mu_k,					
	\end{align}
	where $B_1,B_2,B_3^k,B_4^k$ are  defined by
	\begin{align*}
		& \textstyle B_1:=\frac{1}{\tau}\|\overline{\vA}\|\left\|[\overline{\vA}; \vA]\right\| +\sigma^{-1}  + 	\frac{1}{\tau}(\|\overline{\vA}\|+\|{\vA}\|)\left\|[\overline{\vA}; \vA]\right\|, \ B_2:= (1+\|( \vA \vA\zz)^{-1}\vA\| \left\|[\overline{\vA}; \vA]\right\|), \\
		& \textstyle B_3^k:= \|( \vA \vA\zz)^{-1}\vA\| (l_f^k+ \|\overline{\vA}\|l_g),  \ B_4^k:={\frac{1}{\tau} \left(l_f^k+\left\|[\overline{\vA}; \vA]\right\| \left( l_g+ B_3^k\right)\right)}.
	\end{align*}
\end{lemma}

\begin{proof}
	By~\eqref{eq:boundz-2}, we have $\|\vz^{(k+1)}-\overline{\vz}^{(k+1)}\|\leq\mu_k$. 
	Then we obtain from~\eqref{eq:xupdate-2} and~\eqref{eq:barx-2} that
	$$
	\|\overline{\vx}^{(k+1)}-{\vx}^{(k+1)}\|\leq\frac{1}{\tau}\left\|[\overline{\vA}; \vA]\right\|\|\vz^{(k+1)}-\overline{\vz}^{(k+1)}\|\leq  \frac{1}{\tau}\left\|[\overline{\vA}; \vA]\right\|\mu_k,
	$$
	and from~\eqref{eq:boundz-2},~\eqref{eq:yupdate-2} and~\eqref{eq:bary-2} that
	$$
	\|\overline{\vy}^{(k+1)}-{\vy}^{(k+1)}\|\leq \|\overline{\vA}\|\|\overline{\vx}^{(k+1)}-{\vx}^{(k+1)}\|+\sigma^{-1}\|\vz^{(k+1)}-\overline{\vz}^{(k+1)}\|\leq \frac{1}{\tau}\left\|\overline{\vA}\right\|\left\|[\overline{\vA}; \vA]\right\|\mu_k +\sigma^{-1} \mu_k.
	$$
	Hence, the two inequalities in~\eqref{eq:lemmaineq1} hold. 
	
	In addition, by~\eqref{eq:feasiblexy-2}, we have 
	\begin{equation*}	
		\begin{aligned}
			&\|{\vy}^{(k+1)}-
			(\overline{\vA}{\vx}^{(k+1)} +	
			\overline{\vb})\|+ \|{\vA}{\vx}^{(k+1)} +
			{\vb}\|  \\
			= & \left\|{\vy}^{(k+1)}-
			(\overline{\vA}{\vx}^{(k+1)} +
			\overline{\vb})-\overline{\vy}^{(k+1)}+
			(\overline{\vA}\overline{\vx}^{(k+1)} +
			\overline{\vb})\right\| +  \left\|({\vA}\overline{\vx}^{(k+1)} +
			{\vb})-({\vA}{\vx}^{(k+1)} +{\vb})\right\|\\
			\leq  &
			\left\|{\vy}^{(k+1)}-\overline{\vy}^{(k+1)}\right\| + \left\| \overline{\vA}{\vx}^{(k+1)}-   \overline{\vA}\overline{\vx}^{(k+1)} \right\| +\left\| {\vA}{\vx}^{(k+1)}-   {\vA}\overline{\vx}^{(k+1)} \right\| \\
			\leq  &
			\left\|{\vy}^{(k+1)}-\overline{\vy}^{(k+1)}\right\| + \left(\|\overline{\vA}\|+\|\vA\|\right)\left\| {\vx}^{(k+1)} -  \overline{\vx}^{(k+1)} \right\|
			\leq
			\mu_k {B_1},
		\end{aligned}
	\end{equation*}
	where the last inequality is from~\eqref{eq:lemmaineq1} and the definition of $B_1$. Hence, the claim in~\eqref{eq:lemmaineq2} holds.
	
	Moreover, recall $\vz_1^{(k+1)}\in \dom(g^{\star})$ in~\eqref{eq:boundz-2}. Since $g$ is $l_g$-Lipschitz continuous by Assumption~\ref{assume:problemsetup}, we must have $\|\vz_1^{(k+1)}\|\leq l_g$~\cite{rockafellar1970convex}. 
	Write $\overline{\vz}^{(k+1)}=((\overline{\vz}_1^{(k+1)})\zz,(\overline{\vz}_2^{(k+1)})\zz)\zz$. Then $\overline\vz_1^{(k+1)}\in \dom(g^{\star})$. Thus for the same reason, 
	we have $\|\overline{\vz}_1^{(k+1)}\|\leq l_g$. By  Assumption~\ref{assume:problemsetup}, $\vA \vA\zz$ is non-singular, so the optimality condition of~\eqref{eq:sublb-2}  implies 
	\begin{equation}\label{eq:formula-z-2-k+1}
		\overline{\vz}_2^{(k+1)}=-\left( \vA \vA\zz\right)^{-1} \left(\vA\overline{\vA}\zz\overline{\vz}_1^{(k+1)}+\vA\nabla f_0(\vx^{(k)})-\tau\vA\vx^{(k)} -\tau \vb\right).
	\end{equation}
	For $k\geq1$, we obtain from the second inequality in~\eqref{eq:feasiblexy-2} and the first inequality in~\eqref{eq:lemmaineq1} that 
	\begin{align*}
		\|\left( \vA \vA\zz\right)^{-1} (\vA\vx^{(k)} +\vb)\|=&\|\left( \vA \vA\zz\right)^{-1}(\vA\vx^{(k)} + \vb-\vA\overline\vx^{(k)} -\vb)\|\\
		\leq& \|( \vA \vA\zz)^{-1}\vA\|\|\vx^{(k)}-\overline\vx^{(k)}\|\leq \frac{1}{\tau}\|( \vA \vA\zz)^{-1}\vA\|\left\|[\overline{\vA}; \vA]\right\|\mu_k.
	\end{align*}
	Since $\vA\vx^{(0)} +\vb=\mathbf{0}$, the inequality above also holds for $k=0$. 
	Hence, 
	\begin{align*}
		\|\vz_2^{(k+1)}\|\leq&\|\vz_2^{(k+1)}-\overline{\vz}_2^{(k+1)}\|+\|\overline{\vz}_2^{(k+1)}\| \leq\mu_k+\left\|\left( \vA \vA\zz\right)^{-1} \left(\vA\overline{\vA}\zz\overline{\vz}_1^{(k+1)}+\vA\nabla f_0(\vx^{(k)})-\tau\vA\vx^{(k)} -\tau \vb\right)\right\|\\
		\leq &\mu_k+\|( \vA \vA\zz)^{-1}\vA\| (l_f^k+ \|\overline{\vA}\|l_g)+\|( \vA \vA\zz)^{-1}\vA\| \left\|[\overline{\vA}; \vA]\right\|\mu_k\leq B^k_3 +B_2\mu_k.
	\end{align*} 
	Thus both inequalities in~\eqref{eq:lemmaineq3} hold.
	
	Finally, it follows from the updating rule in~\eqref{eq:xupdate-2} and the inequalities in~\eqref{eq:lemmaineq3} that 
	\begin{align*}&\|\vx^{(k+1)}-\vx^{(k)}\| =\frac{1}{\tau}\left\|  [\overline{\vA}; \vA]\zz \vz^{(k+1)}+\nabla f_0(\vx^{(k)})\right\| \\
		\leq & \frac{1}{\tau} \left(\left\|[\overline{\vA}; \vA]\right\|\left\| \vz_1^{(k+1)}\right\|+\left\|[\overline{\vA}; \vA]\right\|\left\|\vz_2^{(k+1)}\right\|+\left\|\nabla f_0(\vx^{(k)})\right\|\right)\\
		\leq &\frac{1}{\tau}l_f^k+ \frac{1}{\tau} \left\|[\overline{\vA}; \vA]\right\|l_g+ \frac{1}{\tau} \left\|[\overline{\vA}; \vA]\right\| (B^k_3 +B_2\mu_k).  
	\end{align*}
	Hence by the definition of $B_4^k$,~\eqref{eq:lemmaineq4} is obtained, and we complete the proof.			
\end{proof}	

With Lemmas~\ref{lem:barxyz-2} and~\ref{lem:boundxyz-2}, we can characterize the number of outer iterations that Algorithm~\ref{alg:ipganormal} needs to find an $\varepsilon$-KKT point of problem~\eqref{eq:model-spli}. We will show that $\{\mu_k\}_{k\ge 0}$ is well defined in the subsequent remark. 

\begin{theorem}
	\label{thm:allcomple}
	Suppose that Assumption~\ref{assume:problemsetup} holds. Given $0<\varepsilon <1$,  let $\tau=2L_f$ and $ \{\mu_k\}_{k\ge 0} $ satisfy
	\begin{align}
		\label{eq:delta}
		\mu_k \leq \widetilde{\mu}_k:= \min\left\{\frac{\varepsilon}{B_1\sigma},\ \frac{\varepsilon}{B_1},\ \frac{\varepsilon^2}{12L_fB_1\left(B^k_3 +\sigma \|\overline{\vA}\|B^k_4+l_g\right)},\ \sqrt{\frac{\varepsilon^2}{12L_fB_1B_2\left(1+\frac{1}{\tau} \sigma \|\overline{\vA}\|\left\|[\overline{\vA}; \vA]\right\|\right)}} \right\},
	\end{align}
	where   $B_1,B_2,B_3^k$ and $B^k_4$ are given in Lemma~\ref{lem:boundxyz-2}.  Define 
	\begin{equation}\label{eq:set-K-eps}
		K_\varepsilon:= \left \lceil    12L_f \Delta_F\varepsilon^{-2}   \right\rceil, 
	\end{equation}	
	where $ \Delta_{F}=F(\vx^{(0)}, \vy^{(0)})-\inf_{\vx,\vy} F(\vx,\vy)$. Let  $\{\vx^{(k)},\vy^{(k)}\}_{k\ge 0}$ be generated by  Algorithm~\ref{alg:ipganormal}, and $$k'=\argmin_{k=0,1,\dots,K_\varepsilon-1}\left\|\vx^{(k+1)}-{\vx}^{(k)}\right\|.$$ Then $(\vx^{(k'+1)},\vy^{(k'+1)})$ is an  $\varepsilon$-KKT point  of problem~\eqref{eq:model-spli}.
\end{theorem}

\begin{proof} 
	We obtain from inequality~\eqref{eq:lemmaineq2} that, for any $k\ge 0$,  
	\begin{equation}
		\label{eq:feasibility}
		\|{\vy}^{(k+1)}-
		(\overline{\vA}{\vx}^{(k+1)} +	
		\overline{\vb})\|+ \|{\vA}{\vx}^{(k+1)} +
		{\vb}\| 
		\leq\mu_k B_1 \leq    \varepsilon.
	\end{equation}
	In addition, by 
	\eqref{eq:yupdate-2}, 
	it holds for any $k\ge 0$ that
	\begin{equation}
		\label{eq:kktvy-2}
		\mathbf{0}\in \partial g(\vy^{(k+1)}) - \vz_1^{(k+1)} +\sigma \left({\vy}^{(k+1)}-
		(\overline{\vA}{{\vx}^{(k+1)}} +
		\overline{\vb})\right).
	\end{equation}	
	Hence, for all $k\ge 0$, there exists $\vxi^{(k+1)}\in\partial g(\vy^{(k+1)})$ such that 
	\begin{equation}
		\label{eq:kktvy-2vio}
		\| \vxi^{(k+1)} - \vz_1^{(k+1)}\| = \sigma \|{\vy}^{(k+1)}-
		(\overline{\vA}{{\vx}^{(k+1)}} +
		\overline{\vb})\| \overset{\eqref{eq:lemmaineq2}} \leq  \mu_k B_1\sigma  \overset{\eqref{eq:delta}}\leq  \varepsilon.
	\end{equation} 
	Moreover, it holds by~\eqref{eq:xupdate-2} that for any $k\ge 0$, 
	\begin{equation}
		\label{eq:kktvxvio-2}
		\mathbf{0}= \nabla f_0(\vx^{(k)}) + [
		\overline{\vA} ;
		\mathbf{A}]\zz \vz^{(k+1)} +\tau (\vx^{(k+1)}-\vx^{(k)}).
	\end{equation}
	Hence, by the $L_f$-Lipschitz continuity of $\nabla f_0$ and the fact that $\tau=2L_f$, we have for any $k\ge0$,
	\begin{align}\label{eq:xerror-2}
		\left\|\nabla f_0(\vx^{(k+1)}) + [
		\overline{\vA} ;
		\mathbf{A}]\zz \vz^{(k+1)}\right\|
		\leq &\left\|\nabla f_0(\vx^{(k)}) + [
		\overline{\vA} ;
		\mathbf{A}]\zz \vz^{(k+1)}\right\|+L_f\|\vx^{(k+1)}-\vx^{(k)} \| \nonumber\\
		=& (\tau+L_f)\|\vx^{(k+1)}-\vx^{(k)} \|
		=3L_f\|\vx^{(k+1)}-\vx^{(k)} \|.
	\end{align}
	Below we bound the average of the square of the right-hand side of~\eqref{eq:xerror-2} over $K$ terms.
	First, by~\eqref{eq:lemmaineq2} and the feasibility of $({\vx}^{(0)}, {\vy}^{(0)})$, i.e, ${\vA}{\vx}^{(0)} +{\vb} = \vzero$ and ${\vy}^{(0)} = \overline{\vA}{\vx}^{(0)} +\overline{\vb}$, we have for any $k\ge0$,
	\begin{align}
		\label{eq:lemmaineq21}
		\left\|{\vy}^{(k+1)}-{\vy}^{(k)}-
		\overline{\vA}({\vx}^{(k+1)}-{\vx}^{(k)})\right\|
		=\left\|{\vy}^{(k+1)}-
		(\overline{\vA}{\vx}^{(k+1)} +
		\overline{\vb})-{\vy}^{(k)}+
		(\overline{\vA}{\vx}^{(k)} +
		\overline{\vb})\right\|\leq &2\mu_k B_1,\\	\label{eq:lemmaineq22}
		\left\|\vA({\vx}^{(k+1)}-{\vx}^{(k)})\right\| = \left\|({\vA}{\vx}^{(k)} +
		{\vb})-({\vA}{\vx}^{(k+1)} +{\vb})\right\|\leq &2\mu_k B_1	.		
	\end{align}
	Second, it follows from the $L_f$-Lipschitz continuity of $\nabla f_0$ that 
	\begin{align}
		\nonumber
		&f_0(\vx^{(k+1)})-f_0(\vx^{(k)})
		\leq \left\langle \nabla f_0(\vx^{(k)}), \vx^{(k+1)}-\vx^{(k)}\right\rangle+\frac{L_f}{2}\left\|\vx^{(k+1)}-{\vx}^{(k)}\right\|^2
		\\	\nonumber
		\overset{\eqref{eq:kktvxvio-2}}= &\left\langle -{\tau}\left(\vx^{(k+1)}-{\vx}^{(k)}\right) -\overline{\vA}\zz \vz_1^{(k+1)}-\vA\zz \vz_2^{(k+1)}, \vx^{(k+1)}-\vx^{(k)}\right\rangle +\frac{L_f}{2}\left\|\vx^{(k+1)}-{\vx}^{(k)}\right\|^2
		\\	\nonumber
		= &\frac{L_f-2\tau}{2}\left\|\vx^{(k+1)}-{\vx}^{(k)}\right\|^2-\left\langle \overline{\vA}\zz \vxi^{(k+1)}, \vx^{(k+1)}-\vx^{(k)}\right\rangle-\left\langle \vz_2^{(k+1)}, \vA(\vx^{(k+1)}-\vx^{(k)})		
		\right\rangle\\	\nonumber
		&~-\left\langle \overline{\vA}\zz (\vz_1^{(k+1)}-\vxi^{(k+1)}), \vx^{(k+1)}-\vx^{(k)}\right\rangle\\	\nonumber
		\leq &-\frac{3L_f}{2}\left\|\vx^{(k+1)}-{\vx}^{(k)}\right\|^2-\left\langle \overline{\vA}\zz \vxi^{(k+1)}, \vx^{(k+1)}-\vx^{(k)}\right\rangle+2\mu_k B_1\| \vz_2^{(k+1)}\|+2\mu_k B_1\sigma \|\overline{\vA}\|\|\vx^{(k+1)}-{\vx}^{(k)}\|\\	\nonumber
		\leq &-\frac{3L_f}{2}\left\|\vx^{(k+1)}-{\vx}^{(k)}\right\|^2-\left\langle \overline{\vA}\zz \vxi^{(k+1)}, \vx^{(k+1)}-\vx^{(k)}\right\rangle+2\mu_k B_1\left(B^k_3 +B_2\mu_k\right)\\\label{eq:proofstationarity1}
		&~+2\mu_k B_1\sigma \|\overline{\vA}\|\left(B^k_4 +  \frac{1}{\tau} \left\|[\overline{\vA}; \vA]\right\|B_2\mu_k\right),
	\end{align}
	where 
	the second inequality holds from $\tau=2L_f$,~\eqref{eq:kktvy-2vio} and~\eqref{eq:lemmaineq22}, and the last inequality follows from~\eqref{eq:lemmaineq3} and~\eqref{eq:lemmaineq4}.
	Third, by the convexity of $g$ and the fact that $\vxi^{(k+1)}\in\partial g(\vy^{(k+1)})$, we have 
	\begin{align}
		\nonumber
		& g(\vy^{(k+1)})- g(\vy^{(k)})
		\leq \left\langle \vxi^{(k+1)}, \vy^{(k+1)}-\vy^{(k)} \right\rangle 
		\\	\nonumber
		= &	 \left\langle \vxi^{(k+1)}, 	\overline{\vA}({\vx}^{(k+1)}-{\vx}^{(k)})  \right\rangle + \left\langle \vxi^{(k+1)}, {\vy}^{(k+1)}-{\vy}^{(k)}-
		\overline{\vA}({\vx}^{(k+1)}-{\vx}^{(k)}) \right\rangle 
		\\	\nonumber
		\leq & \left\langle \vxi^{(k+1)}, 	\overline{\vA}({\vx}^{(k+1)}-{\vx}^{(k)})  \right\rangle +2\mu_k B_1\|\vxi^{(k+1)}\|\\\label{eq:proofstationarity2}
		\leq &\left\langle \vxi^{(k+1)}, 	\overline{\vA}({\vx}^{(k+1)}-{\vx}^{(k)})  \right\rangle +2\mu_k B_1l_g,
	\end{align}
	where the second inequality is from~\eqref{eq:lemmaineq21}, and the last one holds by $\|\vxi^{(k+1)}\|\le l_g$ from the $l_g$-Lipschitz continuity of $g$. 
	Adding~\eqref{eq:proofstationarity1} and~\eqref{eq:proofstationarity2} and combining terms give
	\begin{align}\label{eq:dec-obj}
		\nonumber
		&f_0(\vx^{(k+1)}) +g(\vy^{(k+1)})-f_0(\vx^{(k)})-g(\vy^{(k)})\\
		\leq &-\frac{3L_f}{2}\left\|\vx^{(k+1)}-{\vx}^{(k)}\right\|^2+2\mu_k B_1\left(B^k_3 +\sigma \|\overline{\vA}\|B^k_4+l_g\right)
		+2\mu_k^2 B_1B_2\left(1+\frac{1}{\tau} \sigma \|\overline{\vA}\|\left\|[\overline{\vA}; \vA]\right\|\right).
	\end{align}
	Multiplying $6L_f$ to the inequality above and summing it over $k=0,1,\ldots,K-1$, we  obtain that
	\begin{equation}\label{eq:stationarybound}
		\begin{aligned}
			&\frac{(3L_f)^2}{K} \sum_{k=0}^{K-1} \left\|\vx^{(k+1)}-{\vx}^{(k)}\right\|^2 \leq   \frac{6L_f\left(F(\vx^{(0)}, \vy^{(0)})-F(\vx^{(K)}, \vy^{(K)})\right)}{K }\\
			& +  \frac{1}{K}\sum_{k=0}^{K-1}12L_f\mu_k B_1\left(B^k_3 +\sigma \|\overline{\vA}\|B^k_4+l_g\right)	
			+ \frac{1}{K}\sum_{k=0}^{K-1} 12L_f\mu_k^2 B_1B_2\left(1+\frac{1}{\tau} \sigma \|\overline{\vA}\|\left\|[\overline{\vA}; \vA]\right\|\right).
		\end{aligned}	
	\end{equation}
	Let $K=K_\varepsilon$ in~\eqref{eq:stationarybound}. Since  $K_\varepsilon\geq 12L_f\left(F(\vx^{(0)}, \vy^{(0)})-\inf_{\vx,\vy}F(\vx, \vy)\right)\varepsilon^{-2}$, it holds that 
	\begin{equation}
		\label{eq:Kep}
		\frac{6L_f\left(F(\vx^{(0)}, \vy^{(0)})-F(\vx^{(K)}, \vy^{(K)})\right)}{K_\varepsilon}\leq \frac{1}{2}\varepsilon^2.
	\end{equation}
	Moreover, the choice of $\mu_k $ in~\eqref{eq:delta} ensures			
	\begin{equation}
		\label{eq:deltaep}
		\begin{aligned}
		 & \frac{1}{K}\sum_{k=0}^{K-1} 12L_f\mu_k B_1\left(B^k_3 +\sigma \|\overline{\vA}\|B^k_4+l_g\right)\leq  
		    \frac{1}{4}\varepsilon^2, \text{ and }
		 \\
		 & \frac{1}{K}\sum_{k=0}^{K-1} 12L_f\mu_k^2 B_1B_2\left(1+\frac{1}{\tau} \sigma \|\overline{\vA}\|\left\|[\overline{\vA}; \vA]\right\|\right) 
		 \leq  \frac{1}{4}\varepsilon^2.
		 \end{aligned}
	\end{equation}			
	Applying the bounds in~\eqref{eq:Kep}  to~\eqref{eq:stationarybound} with $K=K_\varepsilon$ gives 
	\begin{align}
		\label{eq:Kepsilon1}
		\frac{(3L_f)^2}{K_\varepsilon}\sum_{k=0}^{K_\varepsilon-1} \left\|\vx^{(k+1)}-{\vx}^{(k)}\right\|^2 \leq \varepsilon^2.
	\end{align}
	Hence, it follows that ${3L_f}\min_{k=0,\dots,K_\varepsilon-1}\left\|\vx^{(k+1)}-{\vx}^{(k)}\right\| \leq \varepsilon$. Then it results from the definition of $k'$ and~\eqref{eq:xerror-2} that 
	\begin{equation}
		\label{eq:kktviosq-1}
		\begin{aligned}
			\left\|\nabla f_0(\vx^{(k'+1)}) + [
			\overline{\vA} ;
			\mathbf{A}]\zz \vz^{(k'+1)}\right\| 
			\leq  {3L_f}\left\|\vx^{(k'+1)}-{\vx}^{(k')}\right\| = {3L_f}\min_{k=0,\dots,K_\varepsilon-1}\left\|\vx^{(k+1)}-{\vx}^{(k)}\right\| \leq   
			\varepsilon,
		\end{aligned}
	\end{equation}
	which together with~\eqref{eq:feasibility} and~\eqref{eq:kktvy-2vio} 
	shows that $(\vx^{(k'+1)},\vy^{(k'+1)})$ is an  $\varepsilon$-KKT point  of problem~\eqref{eq:model-spli} by Definition~\ref{def:eps-pt-P}. This completes the proof.
\end{proof}

\begin{remark}[Boundedness of the sequence]
	\label{rem:1}
	We observe that, for any $0 \le k\leq K_\varepsilon $, it holds that
	\begin{align}
		\notag
		\|\vx^{(k)}\| \le &\|\vx^{(0)}\| + \sum_{s=0}^{k-1}\|\vx^{(s)}-\vx^{(s+1)}\|\leq \|\vx^{(0)}\| + \sum_{s=0}^{K_\varepsilon-1}\|\vx^{(s)}-\vx^{(s+1)}\|\\\label{eq:xkx01}
		\leq &\|\vx^{(0)}\| + \sqrt{K_\varepsilon}\cdot \sqrt{\sum_{s=0}^{K_\varepsilon-1}\|\vx^{(s)}-\vx^{(s+1)}\|^2} \leq \|\vx^{(0)}\| + \sqrt{K_\varepsilon}\cdot \sqrt{\frac{K_\varepsilon\cdot \varepsilon^2}{9L_f^2}} = \|\vx^{(0)}\| +  {\mathcal O}(\varepsilon^{-1}), 
	\end{align}
	where the third inequality is by the Cauchy–Schwarz inequality, the fourth one is by~\eqref{eq:Kepsilon1}, and the last equality follows from~\eqref{eq:set-K-eps}.
	Consequently, we establish an upper bound for the gradient norm:
	\begin{align}
		\notag
		 l_f^k=\|\nabla f_0(\vx^{(k)})\| \leq &\|\nabla f_0(\vx^{(0)})\| + \|\nabla f_0(\vx^{(0)}) - \nabla f_0(\vx^{(k)})\|\leq \|\nabla f_0(\vx^{(0)})\| + L_f \|\vx^{(0)}- \vx^{(k)}\|  \\ \le&  \|\nabla f_0(\vx^{(0)})\| + L_f \sqrt{K_\varepsilon}\cdot \sqrt{\frac{K_\varepsilon\cdot \varepsilon^2}{9L_f^2}} =\|\nabla f_0(\vx^{(0)})\| + {\mathcal O}(\varepsilon^{-1}) .
		\label{eq:xkx02} 
	\end{align}
	Thus, the quantities $B_3^k$ and $B_4^k$ defined in Lemma \ref{lem:boundxyz-2} are at most of order $ {\mathcal O}(\varepsilon^{-1})$, and are well-defined.  $\{\mu_k\}_{k\ge 0}$ given in~\eqref{eq:delta} can be well controlled. 
\end{remark}

Theorem~\ref{thm:allcomple}   only shows that  Algorithm~\ref{alg:ipganormal} can produce an $\varepsilon$-KKT point  of problem~\eqref{eq:model-spli}. In fact, with a~$\mu_k$ that is slightly smaller that in Algorithm~\ref{alg:ipganormal} at each iteration, we can also show that Algorithm~\ref{alg:ipganormal} can produce a $(\delta, \varepsilon)$-KKT point  of~\eqref{eq:model} with an outer-iteration number similar to that in Theorem~\ref{thm:allcomple}. This result is presented in the following theorem.

\begin{theorem}
	\label{thm:allcompleforp}
	Suppose that {Assumption~\ref{assume:problemsetup} holds}.  Given $0<\varepsilon <1$, in Algorithm~\ref{alg:ipganormal}, let $\tau=2L_f$ and $\{\mu_k\}_{k\ge 0}$ satisfy
\begin{equation}
	\label{eq:delta2}
	\begin{aligned}
		\mu_k\leq \overline{\mu}_k:=\min\Bigg\{&\frac{\vareps^2}{480B_1 l_g L_f}, \frac{\tau\delta}{\|[\overline{\vA}; \vA]\|}, \frac{\varepsilon}{6\|[\overline{\vA}; \vA]\|}, \\ &\frac{\varepsilon}{B_1\sigma},\ \frac{\varepsilon}{B_1},\ \frac{\varepsilon^2}{12L_fB_1\left(B^k_3 +\sigma \|\overline{\vA}\|B^k_4+l_g\right)},\ \sqrt{\frac{\varepsilon^2}{12L_fB_1B_2\left(1+\frac{1}{\tau} \sigma \|\overline{\vA}\|\left\|[\overline{\vA}; \vA]\right\|\right)}} \Bigg\}, 
		\end{aligned}
	\end{equation}where $\delta>0$, and $B_1,B_2,B_3^k$ and $B^k_4$ are given in Lemma~\ref{lem:boundxyz-2}.  Let 
	\begin{equation}\label{eq:set-K-eps2}
		\overline K_\varepsilon:= \left \lceil  120 L_f \Delta_{F_0}\varepsilon^{-2} \right\rceil 
	\end{equation}	
	with $ \Delta_{F_0}=F_0(\vx^{(0)})-\inf_{\vx} F_0(\vx)$,  $\{\vx^{(k)},\vy^{(k)}, \vz^{(k)}\}_{k\ge 0}$ be generated from  Algorithm~\ref{alg:ipganormal},  $(\overline{\vx}^{(k+1)}, \overline{\vy}^{(k+1)})$ be defined as in Lemma~\ref{lem:barxyz-2} with $\overline{\vz}^{(k+1)} = \mathbf{proj}_{\Omega^{(k+1)}}(\vz^{(k+1)})$. Let $k'=\underset{{k=0,\dots,\overline{K}_\varepsilon-1}}\argmin\left\|\vx^{(k+1)}-{\vx}^{(k)}\right\|$. Then $(\vx^{(k'+1)},\vy^{(k'+1)})$ is an  $\varepsilon$-KKT point  of problem~\eqref{eq:model-spli},  $\overline{\vx}^{(k'+1)}$ is an $\varepsilon$-KKT point  of problem~\eqref{eq:model}, and $\vx^{(k'+1)}$ is an $(\delta, \varepsilon)$-KKT point  of problem~\eqref{eq:model}.  
\end{theorem}
\begin{proof}
	Notice $\mathbf{0}= \nabla f_0(\vx^{(k)}) + [
	\overline{\vA} ;
	\mathbf{A}]\zz \overline\vz^{(k+1)} +\tau (\overline\vx^{(k+1)}-\vx^{(k)})$ from~\eqref{eq:barx-2}. Hence through the same arguments to obtain~\eqref{eq:xerror-2}, we arrive at  
	\begin{equation}
		\label{eq:fobar}
		\left\|\nabla f_0(\overline{\vx}^{(k+1)}) + [
		\overline{\vA} ;
		\mathbf{A}]\zz \overline{\vz}^{(k+1)}\right\| \leq {3L_f}\left\|\overline{\vx}^{(k+1)}-{\vx}^{(k)}\right\|, \forall\, k\ge0.
	\end{equation}
	By the triangle inequality, it holds $\left\|\overline{\vx}^{(k+1)}-{\vx}^{(k)}\right\| \le \left\|\overline{\vx}^{(k+1)}-{\vx}^{(k+1)}\right\|+ \left\|{\vx}^{(k+1)}-{\vx}^{(k)}\right\|$. Thus from~\eqref{eq:lemmaineq1}, we have	
	$\left\|\overline{\vx}^{(k+1)}-{\vx}^{(k)}\right\| \le \left\|{\vx}^{(k+1)}-{\vx}^{(k)}\right\| + \frac{1}{\tau}\left\|[\overline{\vA}; \vA]\right\|{\mu_k},$ which together with~\eqref{eq:fobar} and $\tau=2L_f$ gives
	\begin{equation}
		\label{eq:fobar-2-1}
		\textstyle \left\|\nabla f_0(\overline{\vx}^{(k+1)}) + [
		\overline{\vA} ;
		\mathbf{A}]\zz \overline{\vz}^{(k+1)}\right\| \leq {3L_f}\left\|{\vx}^{(k+1)}-{\vx}^{(k)}\right\| + \frac{3}{2} \left\|[\overline{\vA}; \vA]\right\|{\mu_k}, \forall\,\ k\ge0.
	\end{equation}
	By using the same method to obtain argument~\eqref{eq:stationarybound} and the definition of $ \mu_k$, we have
	\begin{equation}\label{eq:stationarybound-15}
		\begin{aligned}
			\frac{(3L_f)^2}{K}&\sum_{k=0}^{K-1} \left\|\vx^{(k+1)}-{\vx}^{(k)}\right\|^2 \leq   \frac{6L_f\left(F(\vx^{(0)}, \vy^{(0)})-[f_0(\vx^{(K)}) +g(\vy^{(K)})]\right)}{K }+\frac{\varepsilon^2}{2}, \forall\, K\ge1.
		\end{aligned}	
	\end{equation}
	Since $({\vx}^{(0)}, \vy^{(0)})$ is a feasible point of problem \eqref{eq:model}, we have
	\begin{align*}
		&F(\vx^{(0)}, \vy^{(0)})-\left[f_0(\vx^{(K)}) +g(\vy^{(K)})\right]\\=& \left(f_0(\vx^{(0)})+g(\overline{\vA}\vx^{(0)}+\overline\vb)-\left[f_0(\vx^{(K)}) +g(\overline{\vA}{\vx}^{(K)} +	
		\overline{\vb})\right]\right) + \left[g(\overline{\vA}{\vx}^{(K)} +	
		\overline{\vb})-g(\vy^{(K)})\right] 
		\\
		\leq&  \Delta_{F_0} + l_g\|\overline{\vA}{\vx}^{(K)} +	
		\overline{\vb}-\vy^{(K)}\|  \\
		\leq&  \Delta_{F_0} + \frac{\epsilon^2}{480 L_f}   
		, \quad\forall\, K\ge1,
	\end{align*}
	where the first inequality is because  $g$ is $l_g$-Lipschitz continuous by Assumption~\ref{assume:problemsetup} and the second inequality is because 
	$ \|{\vy}^{(K)}-
	(\overline{\vA}{\vx}^{(K)} +	
	\overline{\vb})\|
	\leq {\mu_K} B_1\leq \frac{\epsilon^2}{480l_g L_f}$ according to~\eqref{eq:feasibility} and the definition of $\mu_k$.
	Now substituting the above inequality into~\eqref{eq:stationarybound-15} with $K=\overline K_\varepsilon\ge\lceil 120 L_f\Delta_{F_0}\varepsilon^{-2} \rceil$,  
	we obtain 	
	$$\frac{(3L_f)^2}{K}\sum_{k=0}^{K-1} \left\|\vx^{(k+1)}-{\vx}^{(k)}\right\|^2 \le \frac{6 L_f \Delta_{F_0} + \frac{\epsilon^2}{80}   }{120 L_f\Delta_{F_0}\varepsilon^{-2}} +\frac{\varepsilon^2}{2} \leq  \frac{\varepsilon^2}{16}+\frac{\varepsilon^2}{2}=\frac{9\varepsilon^2}{16}.$$  
	Since $k'=\argmin_{k=0,\dots,\overline{K}_\varepsilon-1}\left\|\vx^{(k+1)}-{\vx}^{(k)}\right\|$, we have
	${3L_f}\left\|\vx^{(k'+1)}-{\vx}^{(k')}\right\| \leq \frac{3\varepsilon}{4}$. Then, it follows from the same arguments in the proof of Theorem~\ref{thm:allcomple} that $(\vx^{(k'+1)},\vy^{(k'+1)})$ is an  $\varepsilon$-KKT point  of problem~\eqref{eq:model-spli}. 
	
	Also,
	we obtain from~\eqref{eq:fobar-2-1} and the choice of $\mu_k$ that 
	\begin{equation}\label{eq:fobar-2-2}
		\left\|\nabla f_0(\overline{\vx}^{(k'+1)}) + [
		\overline{\vA} ;
		\mathbf{A}]\zz \overline{\vz}^{(k'+1)}\right\| \leq \varepsilon.
	\end{equation}		
	On the other hand, we have from~\eqref{eq:feasiblexy-2} and~\eqref{eq:bary-2} that	
	\begin{equation}
		\label{eq:stabarx}		
		\overline{\vy}^{(k'+1)}= 
		(\overline{\vA}\overline{\vx}^{(k'+1)} +	
		\overline{\vb}),\quad {\vA}\overline{\vx}^{(k'+1)} +
		{\vb} =\vzero ,\quad
		\overline{\vz}_1^{(k'+1)}\in \partial g(\overline{\vy}^{(k'+1)}).	
	\end{equation}
	Recall $g(\vx) = g(\overline\vA\vx + \overline\vb)$. 	By~\cite{clarke1990optimization}, 
	it holds that
	$
	\partial g(\overline{\vx}^{(k'+1)})=\overline{\mathbf{co}}\left(\left\{\overline{\vA}\zz \vxi: \vxi \in \partial g\big(\overline{\vA}\overline{\vx}^{(k'+1)} +	
	\overline{\vb}\big)\right\} \right),
	$
	and thus from~\eqref{eq:stabarx}, it follows	 $\overline{\vA}\zz\overline{\vz}_1^{(k'+1)} \in \partial g(\overline{\vx}^{(k'+1)}).$ Then by~\eqref{eq:fobar-2-2}, we have 
	$$
	\textstyle \dist\left(\vzero, \nabla {f}_0\big(\overline{\vx}^{(k'+1)}\big)+\partial g\big(\overline{\vx}^{(k'+1)}\big)+{\vA}\zz \overline{\vz}^{(k'+1)}_2\right) \le \left\|\nabla {f}_0\big(\overline{\vx}^{(k'+1)}\big)+\overline{\vA}\zz\overline{\vz}^{(k'+1)}_1+{\vA}\zz \overline{\vz}^{(k'+1)}_2\right\| \le \varepsilon,$$
	which, together with 	${\vA}\overline{\vx}^{(k'+1)} + {\vb} =\vzero$ from~\eqref{eq:stabarx}, indicates that
	$\overline{\vx}^{(k'+1)}$ is an $\varepsilon$-KKT point of~\eqref{eq:model}. 
	
	By~\eqref{eq:lemmaineq1}, 	$\|\overline{\vx}^{(k'+1)}-{\vx}^{(k'+1)}\|\leq \frac{1}{\tau}\left\|[\overline{\vA}; \vA]\right\|\mu_k\leq \delta$, where the second inequality is from the definition of~$\mu_k$ which ensures $ \mu_k \leq  \frac{\tau\delta}{\|[\overline{\vA}; \vA]\|}$. The proof is then completed.
\end{proof}
 
\begin{remark}
	\label{rem:2}
	With a similar technique to obtain arguments in Remark \ref{rem:1}, we can obtain that for any $0 \le k\leq \overline K_\varepsilon $, the inequalities \eqref{eq:xkx01}--\eqref{eq:xkx02} hold with $K_{\varepsilon}$ replaced by $\overline K_{\varepsilon}$. Thus, the quantities $B_3^k$ and $B_4^k$ defined in Lemma~\ref{lem:boundxyz-2} are at most of order $ {\mathcal O}(\varepsilon^{-1})$, and are well-defined.  $\{\mu_k\}_{k\ge 0}$ given in~\eqref{eq:delta2}  can be well controlled. 
\end{remark}

\subsection{Number of inner iterations for finding $\vz^{(k+1)}$ satisfying~\eqref{eq:boundz-2}}\label{sec:z-prob}

To obtain the total number of inner iterations of Algorithm~\ref{alg:ipganormal} to find an $\varepsilon$-KKT point of~\eqref{eq:model-spli}, we still need to evaluate the  number of iterations for computing $\vz^{(k+1)}$ that satisfies the criterion in~\eqref{eq:boundz-2} for each $k\ge0$. 
The problem in~\eqref{eq:sublb-2} has the convex composite structure. Additionally, the gradient of the smooth function $\mathcal{D}_k(\vz)-g^{\star} (\vz_1)$ in the objective function $\mathcal{D}_k$ in~\eqref{eq:sublb-2} is $L_\mathcal{D}$-Lipschitz continuous with 
\begin{equation}\label{def:L_D}
L_\mathcal{D}:=\lambda_{\max}([\overline{\vA}; \vA][\overline{\vA}; \vA]\zz)/\tau.
\end{equation}
Hence, we apply the APG method in~\cite{beck2009fast} with a standard restarting technique
to~\eqref{eq:sublb-2} to find $\vz^{(k+1)}$.  The restarted APG algorithm instantiated on~\eqref{eq:sublb-2} is presented in Algorithm~\ref{alg:restartedAPG}. 
While it is written in a double-loop format, this is merely for notational brevity—the algorithm is inherently a restarted APG method, which can be reformulated as a single-loop algorithm.
At each iteration, the algorithm maintains both a main iterate  $\vz^{(j,k)}$ and an auxiliary iterate $\widehat\vz^{(j,k)}$. In the algorithm, we use $\widehat{\mathcal{G}}_1^j$ and $\widehat{\mathcal{G}}_2^j$ to represent the gradients of  $\mathcal{D}_k(\vz)-g^{\star} (\vz_1)$ with respect to $\vz_1$ and $\vz_2$, respectively, at $\vz=\widehat\vz^{(j,k)}$. 
To accelerate convergence, the algorithm restarts every
  $j_k$ APG steps. We call each APG step, i.e., Line~5-9 in Algorithm~\ref{alg:restartedAPG}, as one inner iteration of Algorithm~\ref{alg:ipganormal}. 

\begin{algorithm}[H]
	\caption{Restarted accelerated proximal gradient method for~\eqref{eq:sublb-2}}\label{alg:restartedAPG}
	\begin{algorithmic}[1]
		\State \textbf{Input:} an initial point $\vz^{\text{ini}}\in\dom(g^*)$ where $(\vx^{(0)},\vy^{(0)})$ is the same as in Algorithm~\ref{alg:ipganormal}.
		\State $\vz^{(0,k)}\leftarrow \vz^{\text{ini}}$, $\widehat\vz^{(0,k)}\leftarrow \vz^{(0,k)}$, $\alpha_0\leftarrow 1$.
		\For  {$i=0,\dots,i_k-1$}
		\For {$j=0,\dots,j_k-1$}
		\State $\widehat{\mathcal{G}}_1^j\leftarrow\frac{1}{\tau}\overline{\vA}\left(\overline{\vA}\zz\widehat\vz_1^{(j,k)} +\vA\zz\widehat\vz_2^{(j,k)} +\nabla f_0(\vx^{(k)})-\tau \vx^{(k)}\right)-\overline\vb.$
		
		\State
		$\widehat{\mathcal{G}}_2^j\leftarrow\frac{1}{\tau}\vA\left(\overline{\vA}\zz\widehat\vz_1^{(j,k)} +\vA\zz\widehat\vz_2^{(j,k)} +\nabla f_0(\vx^{(k)})-\tau \vx^{(k)}\right)-\vb$.
		
		\State $\vz_1^{(j+1,k)}\leftarrow \mathbf{prox}_{L_\mathcal{D}^{-1}g^*}\left(\widehat\vz_1^{(j,k)}-\frac{1}{L_\mathcal{D}}\widehat{\mathcal{G}}_1^j\right)$,~
		$\vz_2^{(j+1,k)}\leftarrow \widehat\vz_2^{(j,k)}-\frac{1}{L_\mathcal{D}}\widehat{\mathcal{G}}_2^j$.
		
		\State $\alpha_{j+1}\leftarrow \frac{1+\sqrt{1+4\alpha_j^2}}{2}$.
		
		\State $\widehat\vz^{(j+1,k)}\leftarrow \vz^{(j+1,k)}+\left(\frac{\alpha_{j}-1}{\alpha_{j+1}}\right)\left(\vz^{(j+1,k)}-\vz^{(j,k)}\right)$.
		
		\EndFor
		\State $\vz^{(0,k)}\leftarrow \vz^{(j_k,k)}$, $\widehat\vz^{(0,k)}\leftarrow \vz^{(0,k)}$, $\alpha_0\leftarrow 1$.
		\EndFor
		\State \textbf{Output:} $\vz^{(k+1)}=\vz^{(0,k)}$.
	\end{algorithmic}
\end{algorithm}

The number of APG steps performed in Algorithm~\ref{alg:restartedAPG} to find a point satisfying~\eqref{eq:boundz-2} is known when Assumption~\ref{assume:problemsetup} and either one of Assumptions~\ref{ass:dual2} and~\ref{ass:polyhedralg} hold. We first present the result when Assumptions	\ref{assume:problemsetup} and~\ref{ass:dual2} hold.	
\begin{theorem}\label{thm:inner-iter-dual2}
	Suppose Assumptions~\ref{assume:problemsetup} and~\ref{ass:dual2} hold.
	Given any $\mu_k>0$, if $j_k=\lceil 2\sqrt{2}\kappa ([\overline{\vA}; \vA])\rceil$ and {$i_k \ge  \left\lceil \log_2(\frac{2(\mathcal{D}_k(\vz^{\text{ini}})-\mathcal{D}_k^*)}{\kappa_\mathcal{D}\mu_k^2})\right\rceil$} 
	in Algorithm~\ref{alg:restartedAPG}, the output $\vz^{(k+1)}$ satisfies the criteria in~\eqref{eq:boundz-2}. Here, 
	$\kappa_\mathcal{D}=\lambda^+_{\min}([\overline{\vA}; \vA][\overline{\vA}; \vA]\zz)/\tau$.
\end{theorem}
 
\begin{proof}
	With Assumption~\ref{ass:dual2}, the problem in~\eqref{eq:sublb-2} is $\kappa_\mathcal{D}$-strongly convex. Consider the end of the first inner loop (i.e., $i=0$) of Algorithm~\ref{alg:restartedAPG}.  According to~\cite[Theorem 4.4.]{beck2009fast}, after $j_k$ APG steps, the APG method generates an output $\vz^{(j_k,k)}$ satisfying 
	$$
	\mathcal{D}_k(\vz^{(j_k,k)})-\mathcal{D}_k^*\leq\frac{2L_\mathcal{D}\dist^2(\vz^{\text{ini}}, \Omega^{(k+1)})}{j_k^2}\leq\frac{4L_\mathcal{D} [\mathcal{D}_k(\vz^{\text{ini}})-\mathcal{D}_k^*]}{\kappa_\mathcal{D}j_k^2}=\frac{4 \kappa^2([\overline{\vA}; \vA]) [\mathcal{D}_k(\vz^{\text{ini}})-\mathcal{D}_k^*]}{j_k^2},
	$$
	where the second inequality is by the $\kappa_\mathcal{D}$-strong convexity of $\mathcal{D}_k$. Since $j_k=\lceil 2\sqrt{2}\kappa ([\overline{\vA}; \vA])\rceil$, we have $\mathcal{D}_k(\vz^{(j_k,k)})-\mathcal{D}_k^* \le \frac{1}{2}\left(\mathcal{D}_k(\vz^{\text{ini}})-\mathcal{D}_k^*\right)$, meaning that we can reduce the objective gap by a half with each inner loop of  Algorithm~\ref{alg:restartedAPG}. Since the next inner loop is started at $\vz^{(j_k,k)}$ according to Line 11 of Algorithm~\ref{alg:restartedAPG}, we repeat this argument $i_k {\ge} \left\lceil \log_2(\frac{2(\mathcal{D}_k(\vz^{\text{ini}})-\mathcal{D}_k^*)}{\kappa_\mathcal{D}\mu_k^2})\right\rceil$ times and show that the final output $\vz^{(k+1)}$ satisfies $\mathcal{D}_k(\vz^{(k+1)})-\mathcal{D}_k^*\leq \frac{\kappa_\mathcal{D}\mu_k^2}{2}$ and thus
	$$
	\dist(\vz^{(k+1)}, \Omega^{(k+1)})\leq  \sqrt{\frac{2(\mathcal{D}_k(\vz^{(k+1)})-\mathcal{D}_k^*)}{\kappa_\mathcal{D}}} \leq \mu_k,
	$$
	which means $\vz^{(k+1)}$ satisfies ~\eqref{eq:boundz-2}. 
\end{proof}

\begin{remark}
		In Theorem~\ref{thm:inner-iter-dual2}, 
		to satisfy the condition $i_k \ge  \left\lceil \log_2(\frac{2(\mathcal{D}_k(\vz^{\text{ini}})-\mathcal{D}_k^*)}{\kappa_\mathcal{D}\mu_k^2})\right\rceil$, one can set $i_k=\left\lceil \log_2(\frac{ \|\vxi_k\|^2}{\kappa^2_\mathcal{D}\mu_k^2})\right\rceil$ for any $\vxi_k \in \partial \cD_k(\vz^{\text{ini}})$. 
	This is a valid and computable $i_k$, because by the $\kappa_\mathcal{D}$-strong convexity of $\cD_k$, it holds that $\mathcal{D}_k(\vz^{\text{ini}})-\mathcal{D}_k^* \le \frac{\|\vxi_k\|^2}{2\kappa_\cD}$ for any $\vxi_k \in \partial \cD_k(\vz^{\text{ini}})$. 
\end{remark}

Below we derive the number of APG steps performed in Algorithm~\ref{alg:restartedAPG}  to find a point satisfying~\eqref{eq:boundz-2} under Assumptions~\ref{assume:problemsetup} and~\ref{ass:polyhedralg}. In this case, 
$g^{\star}(\vz_1)=\iota_{\vC \vz_1\leq\vd}(\vz_1)$ because the conjugate of $\iota_{\vC \vz_1\leq\vd}(\vz_1)$ is $\max\{\vu\zz\vz_1: \vC \vu\leq\vd\}=g(\vz_1)$ and the conjugate of the conjugate of a closed convex function is just the function itself. In this case, the minimization problem in~\eqref{eq:sublb-2} 
may not be strongly convex, but it is a quadratic program for which a restarted APG method can still have a linear convergence rate as shown in~\cite{necoara2019linear}.  Recall that $\Omega^{(k+1)}$ is the solution set in~\eqref{eq:sublb-2}. Since $\|\cdot\|^2$ is strongly convex, by the same proof of Eqn.~(40) in~\cite{necoara2019linear}, there exists $\vnu^{(k)}\in\mathbb{R}^d$ such that 
\begin{eqnarray}
	\label{eq:tk}
	[\overline{\vA}; \vA]\zz\overline\vz^{(k+1)} +\nabla f_0(\vx^{(k)})-\tau \vx^{(k)}=\vnu^{(k)},~\forall\,  \overline\vz^{(k+1)} \in \Omega^{(k+1)}.
\end{eqnarray}
As a result, we have 
\begin{eqnarray}
	\label{eq:sk}
	- (\overline\vz^{(k+1)}_1)^{\top}\overline{\vb}- (\overline\vz^{(k+1)}_2)^{\top}\vb =	\mathcal{D}_k^*-\frac{1}{2\tau}\|\vnu^{(k)}\|^2,~\forall\,  \overline\vz^{(k+1)} \in \Omega^{(k+1)}.
\end{eqnarray}
Therefore, $\Omega^{(k+1)}$ can be characterized  as the solution set of the linear system as follows
\begin{align}
	\label{eq:linearsystem}
	\Omega^{(k+1)}=\left\{\vz=(\vz_1\zz,\vz_2\zz)\zz\,\Bigg| \,
	\begin{array}{c}
		[\overline{\vA}; \vA]\zz\vz=\vnu^{(k)}-\nabla f_0(\vx^{(k)})+\tau \vx^{(k)},\\
		-\vz_1^{\top}\overline{\vb}- \vz_2^{\top}\vb=\mathcal{D}_k^*-\frac{1}{2\tau}\|\vnu^{(k)}\|^2,\\
		\vC \vz_1\leq\vd
	\end{array}  
	\right\}.
\end{align}
The Hoffman constant is a constant $\theta(\overline{\vA},\vA,\overline\vb,\vb,\vC)$ such that 
\begin{align}
	\label{eq:defHoffman}
	\dist(\vz, \Omega^{(k+1)})\leq \theta(\overline{\vA},\vA,\overline\vb,\vb,\vC)\left\|\left[
	\begin{array}{c}
		[\overline{\vA}; \vA]\zz\vz+\nabla f_0(\vx^{(k)})-\tau \vx^{(k)}-\vnu^{(k)}\\
		-\vz_1^{\top}\overline{\vb}- \vz_2^{\top}\vb-\mathcal{D}_k^*+\frac{1}{2\tau}\|\vnu^{(k)}\|^2\\
		\left(\vC \vz_1-\vd \right)_+
	\end{array}
	\right]
	\right\|, \forall\, \vz.
\end{align}
Note that $\theta(\overline{\vA},\vA,\overline\vb,\vb,\vC)$ does not depend on the right-hand sides of the linear system in~\eqref{eq:linearsystem}. As shown in~\cite{necoara2019linear}, the linear convergence rate of the APG depends on $\theta(\overline{\vA},\vA,\overline\vb,\vb,\vC)$. In a special case where $\overline{\vb}=\mathbf{0}$ and $\vb=\mathbf{0}$, we drop the equality $-\vz_1^{\top}\overline{\vb}- \vz_2^{\top}\vb=\mathcal{D}_k^*-\frac{1}{2\tau}\|\vnu^{(k)}\|^2$ from~\eqref{eq:linearsystem} and denote the corresponding Hoffman constant by $\theta(\overline{\vA},\vA,\vC)$ which satisfies
\begin{align}
	\label{eq:defHoffman_small}
	\dist(\vz, \Omega^{(k+1)})\leq \theta(\overline{\vA},\vA,\vC)\left\|\left[
	\begin{array}{c}
		[\overline{\vA}; \vA]\zz\vz+\nabla f_0(\vx^{(k)})-\tau \vx^{(k)}-\vnu^{(k)}\\
		\left(\vC \vz_1-\vd \right)_+
	\end{array}
	\right]
	\right\| , \forall\, \vz.
\end{align}
With these notations, the number of APG steps to obtain $\vz^{k+1}$ satisfying~\eqref{eq:boundz-2} is characterized as follows. 


\begin{theorem}
	\label{thm:complesub} 
	Suppose Assumptions~\ref{assume:problemsetup} and~\ref{ass:polyhedralg} hold.
	Given any $\mu_k>0$, if  $j_k\ge \left\lceil 2\sqrt{2 L_\mathcal{D}/\rho_k}\right\rceil$ and $i_k\ge \left\lceil \log_2(\frac{2(\mathcal{D}_k(\vz^{\text{ini}})-\mathcal{D}_k^*)}{\rho_k\mu_k^2})\right\rceil$   in Algorithm~\ref{alg:restartedAPG},  the output $\vz^{(k+1)}$ satisfies  the criteria in~\eqref{eq:boundz-2}. Here, 
	\begin{align}
		\label{eq:muk}
		\rho_k= \left\{
		\begin{array}{ll}
			\left[\theta^2(\overline{\vA},\vA,\overline\vb,\vb,\vC)(\tau+ \mathcal{D}_k(\vz^{\text{ini}})-\mathcal{D}_k^*+2\tau \|\vnu^{(k)}\|^2)\right]^{-1}&\text{ if }\overline{\vb}\neq\vzero\text{ or } \vb\neq\vzero,\\
			\left[\theta^2(\overline{\vA},\vA,\vC)\tau\right]^{-1} &\text{ if }\overline{\vb}=\vzero\text{ and } \vb=\vzero,
		\end{array}
		\right.
	\end{align}
	with $\vnu^{(k)}$ defined in~\eqref{eq:tk} for $k\geq0$.
\end{theorem}

\begin{proof}
	According to~\cite[Theorem~10]{necoara2019linear}, $\mathcal{D}_k(\vz)$  has a quadratic growth on the level set 
	$$
	\mathcal{Z}_k:=\left\{\vz\big |\vC \vz_1\leq\vd,~\mathcal{D}_k(\vz)\leq \mathcal{D}_k(\vz^{\text{ini}}) \right\}.
	$$
	More specifically, it holds that
	\begin{align}
		\label{eq:quadraticgrowth}
		\frac{\rho_k}{2}\dist^2(\vz, \Omega^{(k+1)})\leq \mathcal{D}_k(\vz)-\mathcal{D}_k^*,\quad\forall\, \vz\in \mathcal{Z}_k,
	\end{align}
	where $\rho_k$ is defined in~\eqref{eq:muk}. Consider the end of the first inner loop (i.e., $i=0$) of Algorithm~\ref{alg:restartedAPG}.  According to~\cite[Theorem 4.4]{beck2009fast}, after $j_k$ APG steps, the APG method generates an output $\vz^{(j_k,k)}$ satisfying 
	$$
	\mathcal{D}_k(\vz^{(j_k,k)})-\mathcal{D}_k^*\leq\frac{2L_\mathcal{D}\dist^2(\vz^{\text{ini}}, \Omega^{(k+1)})}{j_k^2} \overset{\eqref{eq:quadraticgrowth}}\leq\frac{4 L_\mathcal{D} [\mathcal{D}_k(\vz^{\text{ini}})-\mathcal{D}_k^*]}{\rho_k j_k^2}.
	$$
	Since $j_k\ge \left\lceil 2\sqrt{2 L_\mathcal{D}/\rho_k}\right\rceil$, we have $\mathcal{D}_k(\vz^{(j_k,k)})-\mathcal{D}_k^* \le \frac{1}{2}\left(\mathcal{D}_k(\vz^{\text{ini}})-\mathcal{D}_k^*\right)$, 
	meaning that we can reduce the objective gap by a half with each inner loop of  Algorithm~\ref{alg:restartedAPG}. Since the next inner loop is started at $\vz^{(j_k,k)}$ according to Line 11 of Algorithm~\ref{alg:restartedAPG}, we repeat this argument $i_k\ge \left\lceil\log_2(\frac{2(\mathcal{D}_k(\vz^{\text{ini}})-\mathcal{D}_k^*)}{\rho_k\mu_k^2})\right\rceil$ times and show that the final output $\vz^{(k+1)}$ satisfies $\mathcal{D}_k(\vz^{(k+1)})-\mathcal{D}_k^*\leq \frac{\rho_k\mu_k^2}{2}$ and thus
	$$
	\dist(\vz^{(k+1)}, \Omega^{(k+1)})\leq  \sqrt{\frac{2(\mathcal{D}_k(\vz^{(k+1)})-\mathcal{D}_k^*)}{\rho_k}} 
	\leq  {\mu_k}{ },
	$$
	which means $\vz^{(k+1)}$ satisfies~\eqref{eq:boundz-2}.  The proof is then completed.  
\end{proof}

%

\begin{remark}
	\label{rem:4} 
	In Theorem~\ref{thm:complesub}, the lower bounds of $i_k$ and $j_k$ may not be computable. In Section~\ref{sec:com}, we will show  computable values of $i_k$ and $j_k$ that satisfy the required condition, provided that the Hoffman constant $\theta(\overline{\vA},\vA,\overline\vb,\vb,\vC)$ is given, and also we will provide a strategy to estimate the Hoffman constant, with an increase of the total oracle complexity by just a constant factor. 
\end{remark}

\subsection{Total number of inner iterations (oracles)}

Combining Theorems~\ref{thm:allcomple},~\ref{thm:inner-iter-dual2} and~\ref{thm:complesub}, we derive the total number of inner iterations for computing an $\varepsilon$-KKT point of~\eqref{eq:model-spli} as follows. 

\begin{corollary}
	\label{cor:totalcom}
	Suppose Assumption~\ref{assume:problemsetup} holds and Algorithm~\ref{alg:ipganormal} uses Algorithm~\ref{alg:restartedAPG} at Step~4. Let $\tau=2L_f$. For any $0<\varepsilon <\min\{\sqrt{3 L_f},1\}$, the following statements hold.
	\begin{enumerate}	
		\item[\textnormal{(a)}] If Assumption~\ref{ass:dual2} holds,  let $\mu_k= \widetilde{\mu}_k$  for all $k\ge 0$ with $\widetilde{\mu}_k$ given in \eqref{eq:delta}. Then, Algorithm~\ref{alg:ipganormal} finds an $\varepsilon$-KKT point of problem~\eqref{eq:model-spli} 
		with total number 
		\begin{eqnarray}
			\label{eq:complexity_stronglyconvex}
			\mathcal O\left({\kappa([\overline{\vA}; \vA])} \log\left({\textstyle  \frac{1}{\varepsilon}}\right)\frac{L_f\Delta_F}{\varepsilon^2}\right)  
		\end{eqnarray}	
		of inner iterations (i.e, APG steps),
		where $\Delta_F=F(\vx^{(0)}, \vy^{(0)})-\inf_{\vx,\vy}F(\vx, \vy)$.	
		\item[\textnormal{(b)}] If Assumption~\ref{ass:polyhedralg} holds and, in addition, $\overline{\vb}=\mathbf{0}, \vb=\mathbf{0}$,  let $\mu_k= \widetilde{\mu}_k$  for all $k\ge 0$ with $\widetilde{\mu}_k$ given in \eqref{eq:delta}. Then,
		the same conclusion in statement \textnormal{(a)} holds except that the total number of inner iterations becomes
		\begin{eqnarray} 
			\mathcal O\left(\sqrt{\lambda_{\max}([\overline{\vA}; \vA] [\overline{\vA}; \vA]\zz)\theta^2(\overline{\vA},\vA,\vC)}\log\left({\textstyle\frac{1}{\varepsilon}}\right)\frac{L_f\Delta_F}{\varepsilon^2}\right),
			\label{eq:complexity_polyhedral_nob}
		\end{eqnarray}	
		where $\theta(\overline{\vA},\vA,\vC)$ is the Hoffman constant given in~\eqref{eq:defHoffman_small}.		
		\item[\textnormal{(c)}]
		If Assumption~\ref{ass:polyhedralg} holds and, in addition, {the level set $\{(\vx, \vy): F(\vx,\vy) \le F(\vx^{(0)}, \vy^{(0)}) +1\}$ is compact}, 
		  let $\mu_k= \frac{1}{(k+1)^2}\widetilde{\mu}_k  $  for all $k\ge 0$ with $\widetilde{\mu}_k$ given in \eqref{eq:delta}. 
		  There exists $B_5$ and $B_6$, which are constants independent of $\varepsilon$, such that 
		  \begin{align}
		  	\label{eq:boundgAx}
		  	\|\vnu^{(k)}\|^2\leq B_5,\quad \mathcal{D}_k(\vz^{\text{ini}})-\mathcal{D}_k^*\leq B_6,~~~\forall\, 0\le k< K_{\epsilon}, 
		  \end{align}
		  with $\vnu^{(k)}$ defined in~\eqref{eq:tk}.
		  Then,
		the same conclusion in statement \textnormal{(a)} holds except that the total number of inner iterations becomes
		\begin{eqnarray} 
		\mathcal O\bigg(\sqrt{\lambda_{\max}([\overline{\vA}; \vA] [\overline{\vA}; \vA]\zz)\theta^2(\overline{\vA},\vA,\overline\vb,\vb,\vC)(1+B_6/\tau+2 B_5)}\log\left({\textstyle \frac{ B_6}{\varepsilon}}\right)\frac{L_f\Delta_F}{\varepsilon^2} \bigg),
			\label{eq:complexity_polyhedral}
		\end{eqnarray}
		where $ \theta(\overline{\vA},\vA,\overline\vb,\vb,\vC)$ is the Hoffman constant given in~\eqref{eq:defHoffman}.
	\end{enumerate}	
\end{corollary}

	\begin{proof}
		According to Theorem~\ref{thm:allcomple}, Algorithm~\ref{alg:ipganormal} needs at most $K_\varepsilon$ outer iterations to find an $\varepsilon$-KKT point of problem~\eqref{eq:model-spli}, where $K_\varepsilon$ is given in~\eqref{eq:set-K-eps}. Moreover, in each outer iteration of  Algorithm~\ref{alg:ipganormal}, Algorithm~\ref{alg:restartedAPG} needs $i_kj_k$ APG steps to find a point $\vz^{(k+1)}$  satisfying~\eqref{eq:boundz-2} with $i_k$ and $j_k$ specified in  Theorems~\ref{thm:inner-iter-dual2} and~\ref{thm:complesub} for different assumptions.  
		Hence, the total number of inner iterations by Algorithm~\ref{alg:ipganormal} is $\sum_{k=0}^{K_\varepsilon-1}i_kj_k$.
		
		To  prove statements (a) and (b), we only need to derive an upper bound for $\mathcal{D}_k(\vz^{\text{ini}})-\mathcal{D}_k^*$ that appears in the formula of $i_k$ in  Theorems~\ref{thm:inner-iter-dual2} and~\ref{thm:complesub}. First, we observe that, for any $k\leq K_\vareps $, the inequalities \eqref{eq:xkx01}--\eqref{eq:xkx02} hold.
		 Second,  from $\|\bar{\vz}_1^{(k+1)}\|\leq l_g$, \eqref{eq:formula-z-2-k+1},~\eqref{eq:xkx01}, and the $l_f$-Lipschitz continuity of $f_0$, we have
		 \begin{align}\label{eq:bd-bar-z2}
		 	\|\overline{\vz}_2^{(k+1)}\| \le &l_g \left\| \left( \vA \vA\zz\right)^{-1}\vA\overline{\vA}\zz\right\| + \left(l^k_f + \tau \|\vx^{(k)}\| \right)\left\|\left( \vA \vA\zz\right)^{-1} \vA\right\| + \tau \left\|\left( \vA \vA\zz\right)^{-1} \vb\right\| = : B^k_{\vz}.
		 \end{align} 
		 Using the bound of $\|\vx^{(k)}\|$ and $l_f^k$ from \eqref{eq:xkx01}--\eqref{eq:xkx02}, and retaining only the key quantity  $1/\varepsilon$, 
		 we conclude that $B^k_{\vz}=\mathcal{O}(\vareps^{-1})$ 
		 for all $0\le k \le K_{\vareps}$.
		Moreover,  for any $\vz$, we have
		\begin{align}
			\notag
			\|\nabla (\mathcal{D}_k -  g^\star)(\vz)\| = &~\left\|\frac{1}{\tau}\left([\bar \vA; \vA] \big([\bar \vA; \vA] ^\top \vz +\nabla f_0(\vx^{(k)})-\tau \vx^{(k)}\big)\right)- [\bar\vb; \vb]\right\| \\
			\le &~\frac{1}{\tau} \|[\overline \vA; \vA]\|^2 \cdot \|\vz\| + \frac{1}{\tau}\|[\overline \vA; \vA]\|\left(l^k_f + \tau \|\vx^{(k)}\| \right) + \|[\overline\vb; \vb]\| =: \widehat{\mathcal{D}}_k(\vz).
			\label{eq:bd-bar-z2-k+13}
		\end{align}
		Hence, it follows from the convexity of $\mathcal{D}_k$ that for any $\vxi \in \partial  g^\star(\vz^{\text{ini}})$,  
		\begin{align}
			\label{eq:boundforDini}
				\mathcal{D}_k(\vz^{\text{ini}})-\mathcal{D}_k^*\leq \left\langle \nabla (\mathcal{D}_k -  g^\star)(\vz^{\text{ini}}) + \vxi, \vz^{\text{ini}}-\bar\vz^{(k+1)}\right\rangle 
			\leq \big(\widehat{\mathcal{D}}_k(\vz^{\text{ini}}) + \|\vxi\|\big) \left(\|\vz^{\text{ini}}\| + l_g + B_\vz^k\right) {:=B_\cD^k}.
		\end{align}
			Now again using the bounds on $\|\vx^{(k)}\|$ and $l_f^k$ from \eqref{eq:xkx01}--\eqref{eq:xkx02} and keeping only the key quantity $1/\varepsilon$, we have $B^k_{\vz}=\mathcal{O}(\vareps^{-1})$ and $\widehat{\mathcal{D}}_k(\vz^{\text{ini}})=\mathcal{O}(\vareps^{-1})$ for all $0\le k \le K_{\vareps}$. 
		Hence, $	\mathcal{D}_k(\vz^{\text{ini}})-\mathcal{D}_k^* = \mathcal{O}(\vareps^{-2})$.
		Applying the bound of $\mathcal{D}_k(\vz^{\text{ini}})-\mathcal{D}_k^*$ to the $i_k$'s in Theorems~\ref{thm:inner-iter-dual2} and~\ref{thm:complesub}, noticing $\log\frac{1}{\mu_k} = \mathcal{O}(\log\frac{1}{\varepsilon})$ from Remark~\ref{rem:1}, and using \eqref{def:L_D}, we directly obtain the complexity results in~\eqref{eq:complexity_stronglyconvex} and~\eqref{eq:complexity_polyhedral_nob}.
		
		To prove statement (c), we notice that in this case,  $i_k$ and $j_k$ are given in  Theorem~\ref{thm:complesub} and they depend on $\rho_k$ in~\eqref{eq:muk} which further depends on $\|\vnu^{(k)}\|$ and $\mathcal{D}_k(\vz^{\text{ini}})-\mathcal{D}_k^*$.
		 {Summing \eqref{eq:dec-obj} for all $k=1,2,\ldots,K-1$, using $\mu_k= \frac{1}{(k+1)^2}\widetilde{\mu}_k  $ and $\epsilon^2 \leq 3 L_f$,   we obtain the bound 
		\begin{align*}
	 	F(\vx^{(K)}, \vy^{(K)}) \leq &  F(\vx^{(0)}, \vy^{(0)}) + \sum_{k=0}^{K-1}\frac{\varepsilon^2}{12  {(k+1)^2}  L_f}
			\le  F(\vx^{(0)}, \vy^{(0)}) + \frac{2\varepsilon^2}{6 L_f}  
			\le  F(\vx^{(0)}, \vy^{(0)}) + 1, 
		\end{align*}
		for all $ 0< K \le K_{\vareps}. $
		Since $\{(\vx, \vy): F(\vx,\vy) \le F(\vx^{(0)}, \vy^{(0)}) +1\}$ is compact,
		it follows that $\{\vx^{(k)}\}_{k=0}^{K_{\vareps}}$ is bounded, i.e., there is a constant $B_\vx>0$ independent of $\varepsilon$ such that $\|\vx^{(k)}\| \le B_\vx$}. 
		Consequently, we have $l_f^k=\|\nabla f_0(\vx^{(k)})\|=\mathcal{O}(1)$, $\|\overline{\vz}^{(k+1)}\|=\mathcal{O}(1)$, where the latter follows from $\|\bar{\vz}_1^{(k+1)}\|\leq l_g$ and \eqref{eq:formula-z-2-k+1}. Finally, the existence of $B_5, B_6$ is then verified by the definition of $\vnu^{(k)}$ in  \eqref{eq:tk}, and the bound in \eqref{eq:boundforDini}. 
		Using~\eqref{eq:set-K-eps}, \eqref{def:L_D}, \eqref{eq:boundgAx} and the fact that $\log\frac{1}{\mu_k} = \mathcal{O}(\log\frac{1}{\varepsilon})$ from Remark \ref{rem:1}. We can bound $i_k$ and $j_k$ from above and obtain~\eqref{eq:complexity_polyhedral}. The proof is then completed.
	\end{proof}
	
	In the following corollary, we also derive from Theorems~\ref{thm:allcompleforp},~\ref{thm:inner-iter-dual2} and~\ref{thm:complesub} the required total number of inner iterations for obtaining a $(\delta,\varepsilon)$-KKT point of~\eqref{eq:model}. We skip its proof because it is essentially the same as that of Corollary~\ref{cor:totalcom}.
	
	\begin{corollary}
		\label{cor:totalcom2}
		Suppose Assumptions~\ref{assume:problemsetup} holds and Algorithm~\ref{alg:ipganormal} uses Algorithm~\ref{alg:restartedAPG} at Step~4.  Let $\tau=2L_f$.  For any $0<\varepsilon <\min\{\sqrt{3 L_f},1\}$, the following statements hold.
		\begin{enumerate}	
			\item[\textnormal{(a)}] If Assumption~\ref{ass:dual2} holds, let $\mu_k= \overline{\mu}_k$  for all $k\ge 0$ with $\overline{\mu}_k$ given in \eqref{eq:delta2}. Then, Algorithm~\ref{alg:ipganormal} finds a $(\delta, \varepsilon)$-KKT point of problem~\eqref{eq:model} with total number  
			\begin{eqnarray}
				\label{eq:complexity_stronglyconvex2} 
				\mathcal O\left({\kappa([\overline{\vA}; \vA])} \log\left({\textstyle\frac{1}{\delta\varepsilon}}\right) \frac{L_f\Delta_{F_0}}{\varepsilon^2}\right) 
			\end{eqnarray}	
			of inner iterations (i.e, APG steps), where $\Delta_{F_0}=F_0(\vx^{(0)})-\inf_{\vx}F_0(\vx)$.	
			\item[\textnormal{(b)}] If Assumption~\ref{ass:polyhedralg} holds and, in addition, $\overline{\vb}=\mathbf{0}, \vb=\mathbf{0}$,  let $\mu_k= \overline{\mu}_k$  for all $k\ge 0$ with $\overline{\mu}_k$ given in \eqref{eq:delta2}. Then, 
			the same conclusion in statement \textnormal{(a)} holds except that the total number of inner iterations  becomes
			\begin{eqnarray}
				\notag 
				\mathcal O\left(\sqrt{\lambda_{\max}([\overline{\vA}; \vA] [\overline{\vA}; \vA]\zz)\theta^2(\overline{\vA},\vA,\vC)}\log\left({\textstyle\frac{1}{ \delta\varepsilon}}\right) \frac{L_f\Delta_{F_0}}{\varepsilon^2}\right),
				\label{eq:complexity_polyhedral_nob2}
			\end{eqnarray}	
			where $\theta(\overline{\vA},\vA,\vC)$ is the Hoffman constant given in~\eqref{eq:defHoffman_small}.		
			\item[\textnormal{(c)}]
			If Assumption~\ref{ass:polyhedralg} holds and, in addition, 
			 {the level set $\{(\vx, \vy): F(\vx,\vy) \le F(\vx^{(0)}, \vy^{(0)}) +1\}$ is compact},  let $\mu_k= \frac{1}{(k+1)^2}\overline{\mu}_k $  for all $k\ge 0$ with $\overline{\mu}_k$ given in \eqref{eq:delta2}. There exists $B_5$ and $B_6$, which are constants independent of $\varepsilon$, such that 
			 \begin{align} 
			 	\|\vnu^{(k)}\|^2\leq B_5,\quad \mathcal{D}_k(\vz^{\text{ini}})-\mathcal{D}_k^*\leq B_6,~~~\forall\, 0\le k< \overline K_{\epsilon}, 
			 \end{align}
			 with $\vnu^{(k)}$ defined in~\eqref{eq:tk}. Then, the same conclusion in statement \textnormal{(a)} holds except that the total number of inner iterations  becomes
			\begin{eqnarray}
				\notag 
				\mathcal O\bigg(\sqrt{\lambda_{\max}([\overline{\vA}; \vA] [\overline{\vA}; \vA]\zz)\theta^2(\overline{\vA},\vA,\overline\vb,\vb,\vC)(1+B_6/\tau+2 B_5)}\log\left({\textstyle \frac{ B_6}{\delta\varepsilon}}\right) \frac{L_f\Delta_{F_0}}{\varepsilon^2} \bigg),
				\label{eq:complexity_polyhedral2}
			\end{eqnarray}
			where $ \theta(\overline{\vA},\vA,\overline\vb,\vb,\vC)$ is the Hoffman constant given in~\eqref{eq:defHoffman}.
		\end{enumerate}	
	\end{corollary}
	
	\begin{remark}
		Note that the key step in deriving the desired results in Corollary~\ref{cor:totalcom2} relies only on the choice of $\overline{K}_{\varepsilon}$ in \eqref{eq:set-K-eps2}, and
		\( \mathcal{O}(\log(\mu_k^{-1})) = \mathcal{O}(\log(\delta^{-1}\varepsilon^{-1})) \)
		from \eqref{eq:delta2} and Remark \ref{rem:2}. 		
		Building on the complexity results in Corollary~\ref{cor:totalcom2}, by setting
		$\delta=\varepsilon$,
		we can achieve an \((\varepsilon,\varepsilon)\)-KKT point of problem~\eqref{eq:model} using the same oracle complexity. 		
		Furthermore, by choosing a sufficiently small \(\delta\), one can find a point that is arbitrarily close to an \(\varepsilon\)-KKT point of problem~\eqref{eq:model}, still within the same oracle complexity bound given in Corollary~\ref{cor:totalcom2}.
	\end{remark}
	
	\subsection{A computable value of $i_k$ and $j_k$ under Assumption \ref{ass:polyhedralg}}\label{sec:com}
		In this subsection, we give computable values of $i_k$ and $j_k$ that satisfy the requirement in Theorem~\ref{thm:complesub}. 

Define	
\begin{equation}\label{eq:bd-rhok}
		B^k_{\rho}= \left\{
		\begin{aligned}
			  & \tau+B_{\mathcal{D }}^k + 4\tau
				\left(  \left\|[\overline{\vA}; \vA] \right\|^2 \left(l_g + B_{\vz}^k\right)^2 +  (l_f^k + \tau \|\vx^{(k)}\|)^2  \right),&\text{if }\overline{\vb}\neq\vzero\text{ or } \vb\neq\vzero,
			\\
			&\tau, &\text{if }\overline{\vb}=\vzero\text{ and } \vb=\vzero.
		\end{aligned}
		\right.
	\end{equation}	
Then from the definition of $\rho_k$ in  \eqref{eq:muk} and $\vnu^{(k)}$ in \eqref{eq:tk} and using \eqref{eq:bd-bar-z2}--\eqref{eq:boundforDini}, we have \mbox{$\frac{1}{\rho_k\theta^2(\overline{\vA},\vA,\overline\vb,\vb,\vC)} \leq B^k_{\rho}$.}
	Hence, from \eqref{eq:boundforDini} and \eqref{eq:bd-rhok}, to satisfy the condition in Theorem~\ref{thm:complesub}, we can set $j_k$ and $i_k$ to
	\begin{equation}
		\label{eq:ikk2}
		\small
		j_k=\left\lceil 2\sqrt{2B^k_{\rho}\theta^2(\overline{\vA},\vA,\overline\vb,\vb,\vC) L_\mathcal{D}}\right\rceil,\  
		i_k = \left\lceil \log_2\left(\frac{2B^k_{\rho}\theta^2(\overline{\vA},\vA,\overline\vb,\vb,\vC)B_{\cD}^k}{ \mu_k^2}\right)\right\rceil,
	\end{equation} 
	which are explicitly computable, provided the Hoffman parameter $\theta(\overline{\vA},\vA,\overline\vb,\vb,\vC)$. 
	
	Without prior knowledge of \( \theta(\overline{\vA},\vA,\overline\vb,\vb,\vC) \), we can start from 
	a (possibly lower) estimate \( \theta_0>0\) for this parameter and adopt an adaptive estimation strategy to refine the estimate. 
	Such a strategy has been successfully applied in prior works. Notably, Nesterov~\cite{nesterov2013gradient} employs the strategy to estimate an unknown strong convexity parameter, and Davis~$\&$~Grimmer~\cite{davis2019proximally} apply it to estimate an unknown weak convexity parameter.
	
	Specifically, we begin
	by running Algorithm~\ref{alg:ipganormal}, which invokes Algorithm~\ref{alg:restartedAPG} at Step 4, where \( j_k \) and \( i_k \) are set according to~\eqref{eq:ikk2} with \( \theta(\overline{\vA}, \vA, \overline\vb, \vb, \vC) \) replaced by \( \theta_0 \). 
	We terminate Algorithm~\ref{alg:ipganormal}, if for some $K>0$, either \eqref{eq:checkstop} or \eqref{eq:checkstop2} given below is satisfied:
	\begin{equation}
		\label{eq:checkstop}
		{3L_f}{} \left\|\vx^{(K)}-{\vx}^{(K-1)}\right\|^2\le \vareps^2,\,\,  \|\vy^{(K)} - (\overline{\vA}\vx^{(K)} +\overline{\vb})\| \le\vareps\min\big\{\frac{1}{\sigma},1\big\}, \,\, \left\|\mathbf{A} \mathbf{x}^{(K)}+\mathbf{b}\right\| \leq\varepsilon,
	\end{equation}
	\begin{equation}
		\label{eq:checkstop2}
		\frac{6L_f\left(F(\vx^{(0)}, \vy^{(0)})-F(\vx^{(K)}, \vy^{(K)})\right)}{K } \le \frac{\vareps^2}{2}.
	\end{equation}
	Upon termination, if condition~\eqref{eq:checkstop} holds, then by \eqref{eq:xerror-2}, $(\vx^{(K)}, \vy^{(K)})$ is an 
	$\vareps$-KKT point of problem~\eqref{eq:model-spli}. 
	Otherwise, if condition~\eqref{eq:checkstop2} holds while~\eqref{eq:checkstop} does not, 
	we update our estimate for the Hoffman parameter by doubling its value, i.e., let $
	\theta_1 = 2\theta_0.
	$
	We then rerun Algorithm~\ref{alg:ipganormal}, which invokes Algorithm~\ref{alg:restartedAPG} at Step 4, where \( j_k \) and \( i_k \) are set according to~\eqref{eq:ikk2} with \( \theta(\overline{\vA}, \vA, \overline\vb, \vb, \vC) \) replaced by \( \theta_1 \).
	This adaptive estimation strategy continues until \eqref{eq:checkstop} holds. 
	Notice that once 
	{$
		\theta_t > \theta(\overline{\vA}, \vA, \overline\vb, \vb, \vC),
		$}
	and $i_k$ and $j_k$ 
	are set according to~\eqref{eq:ikk2} with \( \theta(\overline{\vA}, \vA, \overline\vb, \vb, \vC) \) replaced by \( \theta_t\), 
	then~\eqref{eq:boundz-2} must hold. Consequently, the condition in~\eqref{eq:checkstop} must eventually be satisfied before~\eqref{eq:checkstop2} holds. This follows because, if~\eqref{eq:checkstop2} is satisfied, then by~\eqref{eq:stationarybound}, together with~\eqref{eq:feasibility},~\eqref{eq:kktvy-2vio} and~\eqref{eq:deltaep}, it necessarily implies that~\eqref{eq:checkstop} must hold for some $0\le k< K$.  
	Notice that {when $K\ge K_\varepsilon$, with $K_\varepsilon$ given in \eqref{eq:set-K-eps}, the condition in \eqref{eq:checkstop2} must hold. Hence,} the total number of inner iterations of $t$-th running {with $\theta_t$ as the estimate of the Hoffman constant} is at most 
	\begin{equation}
		\label{eq:totalguess}
		\sum_{k=0}^{K_\varepsilon-1}i_kj_k = \sum_{k=0}^{K_\varepsilon-1}
		2\theta_t \sqrt{2B^k_{\rho}  L_\mathcal{D}} i_k \sim \theta_t  \sum_{k=0}^{K_\varepsilon-1}
		2\sqrt{2B^k_{\rho}  \lambda_{\max}([\overline{\vA}; \vA] [\overline{\vA}; \vA]\zz)}  \log( \theta(\overline{\vA}, \vA, \overline\vb, \vb, \vC)\vareps^{-1}).
	\end{equation}  
	
	For the case in Corollary \ref{cor:totalcom}(b), it holds $B^k_{\rho}=\tau$ by  the  definition in \eqref{eq:bd-rhok}. Summing \eqref{eq:totalguess} over 
	$t=0,1,\ldots, \lceil \log_2(\theta(\overline{\vA}, \vA, \overline\vb, \vb, \vC)/ \theta_0) \rceil,$ and comparing the resulting term to the complexity bound in Corollary~\ref{cor:totalcom}(b), we deduce that this strategy introduces an additional complexity factor:
	\mbox{$\sum {\theta_t}/{\theta(\overline{\vA}, \vA, \overline\vb, \vb, \vC)} \leq 4 $}. Thus, the overall complexity bound remains in the same order.
	
	For the case in Corollary \ref{cor:totalcom}(c), we have $ {B}_{\rho}^k=\mathcal{O}(1)$ for all $0\le k < K_\varepsilon$. 
	 Similarly,  the overall complexity bound remains in the same order.

	\subsection{Oracle complexity matching the lower bound}
	As we have mentioned in Section \ref{sec:intro}, the lower bound to achieve an ($\varepsilon,\varepsilon$)-KKT point (see Definition \ref{def:eps-pt-P}) of problem~\eqref{eq:model} is  at least
	$\Omega({\kappa([\overline{\vA};\vA]) L_f \Delta_{F_0}} \varepsilon^{-2})$.
	Meanwhile, it is  established in \cite{liu2025lowercomplexityboundsfirstorder} that the lower bound to achieve an $\varepsilon$-KKT point (see Definition \ref{def:eps-pt-P}) of problem~\eqref{eq:model-spli} is at least
	$\Omega({\kappa([\overline{\vA};\vA]) L_f \Delta_{F}} \varepsilon^{-2})$.
	In this subsection, We  provide several remarks to point out that the oracle complexity results in Corollaries~\ref{cor:totalcom} and~\ref{cor:totalcom2} matches the lower bound of oracle complexity,  under the following two cases: (i)  Assumptions~\ref{assume:problemsetup} and~\ref{ass:dual2} hold; (ii) Assumptions~\ref{assume:problemsetup} and~\ref{ass:polyhedralg} hold and in addition, $\overline{\vb}=\mathbf{0}$ and $\vb=\mathbf{0}$.
	
	\begin{remark} \label{rem:matchsp}
		A few remarks follow on the oracle complexity of FOMs for solving problem~\eqref{eq:model-spli}. 
		\begin{enumerate}
			\item[\textnormal{(i)}] 
			If Assumptions~\ref{assume:problemsetup} and~\ref{ass:dual2} hold,
			we assert that the oracle complexity\footnote{The dependency of oracle complexity in~\eqref{eq:complexity_stronglyconvex} on $\Delta_{F}$, $l^k_f$ and $l_g$ is only logarithmic and has been suppressed in $\mathcal{O}(\cdot)$.} in~\eqref{eq:complexity_stronglyconvex} is not improvable (up to a logarithmic factor)  by comparing their dependency on $\varepsilon$, $L_f$ and $\Delta_F$ and $\kappa([\overline{\vA}; \vA])$. 
			
			\item[\textnormal{(ii)}] Suppose Algorithm~\ref{alg:ipganormal} is applied to the reformulation~\eqref{eq:model-spli} of instance~$\mathcal{P}$ given in \cite{liu2025lowercomplexityboundsfirstorder} (see Definition~\ref{def:hardinstance}). In particular, for instance~$\mathcal{P}$,  it holds that  $g(\vy)=\beta\|\vy\|_1=\max_{\|\vz_1\|_\infty\leq\beta}\vz_1^\top\vy$, and $\overline{\vb}=\mathbf{0}$, and $\vb=\mathbf{0}$. Thus,  Algorithm~\ref{alg:ipganormal} finds an $\varepsilon$-KKT point of~\eqref{eq:model-spli} with oracle complexity given in \eqref{eq:complexity_polyhedral_nob}. In this case, it holds that $g^\star(\vz_1)=\iota_{\|\vz_1\|_\infty\leq \beta}$, and  $\theta(\overline{\vA},\vA,\vC)$ becomes the Hoffman constant of the linear system
			$$
			\left\{\vz=(\vz_1\zz,\vz_2\zz)\zz\Bigg| 
			\begin{array}{c}
				[\overline{\vA}; \vA]\zz\vz=\vnu^{(k)}-\nabla f_0(\vx^{(k)})+\tau \vx^{(k)},\\
				-\beta\leq[\vz_{1}]_i\leq \beta, \forall i
			\end{array}  
			\right\}
			$$
			for some $\vnu^{(k)}$. According to Proposition 6 and the discussion on page 12 of~\cite{pena2021new}, we can define a reference polyhedron 
			$$
			\mathcal{R} = \left\{\vz=[\vz_1;\vz_2]:  -\beta\leq[\vz_{1}]_i\leq \beta, \forall i\right\}
			$$
			and characterize $\theta(\overline{\vA},\vA,\vC)$ as
			$$
			\frac{1}{\theta(\overline{\vA},\vA,\vC)}=\min_{\mathcal{K}\in\mathcal{S}}
			\min\limits_{\vu,\vv}
			\left\{\|\vu\|~\Big|~\|\vH^\top \vv\|= 1,~\vv\in \mathcal{K},~\vH\vH^\top \vv-\vu\in \mathcal{K}^*
			\right\},
			$$
			where $\vH$ is defined in~\eqref{eq:matrixAstar}, $\mathcal{K}^*$ is the dual cone of $\mathcal{K}$, and
			$$
			\mathcal{S}=\{\mathcal{K}: \mathcal{K} \text{ is a tangent cone of }\mathcal{R}\text{ at some point of }\mathcal{R}\text{, and }  \vH^\top \mathcal{K} \text{ is a linear space}\}.
			$$
			
			Next we show $\mathcal{S}=\{\mathbb{R}^{d}\}$. By the definition of $\mathcal{R}$, any tangent cone of $\mathcal{R}$ must look like 
			$$
			\mathcal{K}=\left\{\vz=[\vz_1;\vz_2]~\big|~[\vz_{1}]_i\geq 0, \forall i\in \mathcal{J}_1, [\vz_{1}]_j\leq 0, \forall j\in \mathcal{J}_2\right\}
			$$
			for some {disjoint} index sets $\mathcal{J}_1$ and $\mathcal{J}_2$. Suppose  $\mathcal{K}\in\cS$. 
			Then, for any $\vv\in \mathcal{K}$, we must have $-\vH^\top\vv\in \vH^\top \mathcal{K}$, namely, there exists $\vv'\in\mathcal{K}$ such that $-\vH^\top\vv=\vH^\top\vv'$. Since $\vH^\top$ has a full-column rank, $\vv+\vv'=\vzero$. This means for any $\vv\in \mathcal{K}$,  we must have $-\vv\in \mathcal{K}$. {Since $\mathcal{J}_1\cap\mathcal{J}_2=\emptyset$, this happens only if} $\mathcal{J}_1=\mathcal{J}_2=\emptyset$. Thus $\mathcal{K}=\mathbb{R}^{d}$ and $\mathcal{K}^*=\{\vzero\}$. 
			
			With this fact, we have
			$$
			\theta(\overline{\vA},\vA,\vC)=\frac{1}{\min\limits_{\vv: \|\vH^\top \vv\|= 1}\|\vH\vH^\top \vv\|}=\frac{1}{\sqrt{\lambda^+_{\min}(\vH\vH^\top)}}.
			$$ 
			Hence by Corollary~\ref{cor:totalcom}(b), the oracle complexity of Algorithm~\ref{alg:ipganormal} becomes
			$$
			\mathcal{O}\left(\sqrt{\frac{\lambda_{\max}(\vH \vH\zz)}{\lambda^+_{\min}(\vH \vH\zz)}}\log({\textstyle \frac{1}{\varepsilon}})\frac{L_f\Delta_F}{\varepsilon^2}\right)=
			\mathcal{O}\left(\kappa([\overline{\vA};\vA])\log({\textstyle \frac{1}{\varepsilon} })\frac{L_f\Delta_F}{\varepsilon^2}\right).
			$$
			Again, we conclude that the oracle complexity above are neither improvable (up to a logarithmic factor) under  Assumptions~\ref{assume:problemsetup} and~\ref{ass:polyhedralg} when $\overline{\vb}=\mathbf{0}$ and $\vb=\mathbf{0}$.
		\end{enumerate}
	\end{remark}			
	
	\begin{remark} 
		\label{rem:matchp}
		A few remarks follow on the oracle complexity of FOMs for solving problem~\eqref{eq:model}. 
		\begin{enumerate}
			\item[\textnormal{(i)}] 
			Again, we assert that  the oracle complexity in~\eqref{eq:complexity_stronglyconvex2}  is not improvable, up to a logarithmic factor, under Assumptions~\ref{assume:problemsetup} and~\ref{ass:dual2} by comparing their dependency on $\varepsilon$, $L_f$ and $\Delta_{F_0}$ and $\kappa([\overline{\vA}; \vA])$. 
			
			\item[\textnormal{(ii)}] Suppose Algorithm~\ref{alg:ipganormal} is applied to instance~$\mathcal{P}$ given in Definition~\ref{def:hardinstance}. We obtain from Remark~\ref{rem:matchsp}(ii) that if $\overline{\vb}=\mathbf{0}$ and $\vb=\mathbf{0}$, then Algorithm~\ref{alg:ipganormal} finds a near $\varepsilon$-KKT point of~\eqref{eq:model} with oracle complexity~\eqref{eq:complexity_polyhedral_nob2}, which equals $\mathcal{O}\left(\kappa([\overline{\vA};\vA])\log({\textstyle \frac{1}{\varepsilon} })\frac{L_f\Delta_{F_0}}{\varepsilon^2}\right)$ under  Assumptions~\ref{assume:problemsetup} and~\ref{ass:polyhedralg}. Hence,  the oracle complexity above are neither improvable (up to a logarithmic factor) under  Assumptions~\ref{assume:problemsetup} and~\ref{ass:polyhedralg} when $\overline{\vb}=\mathbf{0}$ and $\vb=\mathbf{0}$.
		\end{enumerate}
	\end{remark}

\section{Numerical Experiments}\label{sec:numer}
In this section, we evaluate the empirical performance of the proposed PG-RPD on a weakly convex problem.  Specifically, we compare its performance with the (linearized) ADMM from \cite{melo2017iteration3} and the (proximal) ALM (PALM) from \cite{rockafellar1976augmented}, both of which are applied to solving the splitting problem \eqref{eq:model-spli}.
We show that as the condition number \( \kappa([\overline{\mathbf{A}}; \mathbf{A}]) \) increases, PG-RPD gradually outperforms ADMM, and, on average, surpasses PALM in efficiency.
All the experiments are performed on iMac with Apple M1, a 3.2 GHz 8-core processor, and 16GB of RAM running MATLAB R2024a.

\subsection{Implementation details}
We conduct the comparisons among PG-RPD, ADMM, and PALM on solving a linearly constrained nonconvex quadratic program with $l_1$ reglarizer in the form of 
\begin{equation}
	\label{eq:test}
	\min _{\mathbf{x} \in \mathbb{R}^d} \frac{1}{2} \mathbf{x}^{\top} \mathbf{Q}_0 \mathbf{x} + \|\overline \vA\vx + \overline\vb\|_1 \text {, s.t. } \mathbf{A} \mathbf{x}+\mathbf{b}=\vzero.
\end{equation}
Suppose that the condition number of $[\overline{\vA};\vA]$ is $\kappa$.
We then generate a full row-rank matrix $[\overline{\vA};\vA]$ by the MATLAB command 
\begin{center}
\verb|[U,~]=qr(randn(n0+n1));|\verb|V=randn(d,n0+n1);|\verb|eigens=linspace(1,1/kappa,n0+n1);|\\ \verb|AAbar=U(:,1:n0+n1)*diag(eigens)*V';| \verb|Abar=AAbar(1:n0,:);| \verb|A=AAbar((n0+1):end,:);|.
\end{center} Here, \verb|n0| equals $\overline{n} = \frac{ d}{2}$, \verb|n1| refers to $n=\frac{ 2d}{5}$, and \verb|kappa| represents~\( \kappa \).  The parameters $d$ and $\kappa$ take values from 
$\{100,1000,2000\}$ and $\{2,10,100,10000\}$, respectively.
The vectors $\overline{\vb},\vb$ are generated using the MATLAB function \verb|randn|, and 
$\mathbf{Q}_0$ is generated by the command \verb|R*diag((Lf-rho)*rand(1,n))*R'-rho*eye(d)|, where \verb|R| is generated as an orthonormal matrix, the weakly convex modulus \verb|rho|$\in\{0.1,1,5\}$, and \verb|Lf|=10$\times$\verb|rho|. 
For the test problem, we know that $L_f =$\verb|Lf| and $\rho$=\verb|rho|.

For PG-RPD, we set $\sigma = 1$, $\tau = \frac{11}{10} L_f$, and $\varepsilon = 10^{-3}$. In each inner iteration, we choose \mbox{$i_k=20$} and \mbox{$j_k=\lceil 2\sqrt{2}\kappa ([\overline{\vA}; \vA])\rceil$.} Additionally, we terminate the inner iterations once the stationary violation of the problem $\min_{\vz}\mathcal{D}_k(\vz)$ falls below $10^{-4}$.
For ADMM, following the notations in \cite{melo2017iteration3}, we set  $\theta = \beta = 1$ and \mbox{$\tau = \frac{11}{10} L_f$.}  
For PALM, following the notations in \cite{rockafellar1976augmented}, we set $c_k=\frac{1}{\rho}$ for all $k\ge 0$.
Both ADMM and PALM employ Nesterov’s accelerated proximal gradient (APG) method \cite{nesterov2013gradient} to solve their core strongly convex subproblems.
In addition, all compared methods are initialized using the MATLAB command \verb|x=A\b; y=Abar*x+bbar;|, and terminated when the KKT violation of problem \eqref{eq:model} falls below $\varepsilon=10^{-3}$.

\subsection{Comparisons among PG-RPD, ADMM, and PALM}
Figures \ref{fig:all}--\ref{fig:all3} present numerical comparisons between {PG-RPD}, {ADMM}, and PALM across varying problem size \( d \), weakly convex modulus $\rho$, and condition numbers \( \kappa([\overline{\mathbf{A}}; \mathbf{A}]) \). The horizontal axis denotes the number of gradient evaluations (Grad), while the vertical axis displays the KKT violation on a logarithmic scale, providing a detailed view of convergence behavior. 
Overall, the results demonstrate that {PG-RPD} performs well, particularly for large-scale and moderately ill-conditioned problems, where it achieves faster convergence and greater robustness compared to {ADMM} and PALM. In comparison, ADMM sometimes diverges when 
$\rho$ and $\kappa$ are large, due to the unbounded feasible region of problem \eqref{eq:test}. 
For small-scale problems, such as \( d = 100 \), all methods demonstrate rapid reductions in KKT violations across all condition numbers. 
As the problem size \( d \) or the weakly convex modulus $\rho$ increases, all methods require more gradient evaluations. 
Moreover, Figures~\ref{fig:all}--\ref{fig:all3} reveal that {ADMM} and PALM are more sensitive to the condition number, with a noticeable slowdown in convergence when \(\kappa = 10000\). In contrast, {PG-RPD} gradually outperforms {ADMM} and PALM, with the performance gap widening as \(\kappa\) increases. 
This phenomenon can be explained by the fact that the upper complexity of PG-RPD is lower than that of ADMM in \cite{melo2017iteration3} by at least a factor of
$
\left( \frac{\|\overline{\mathbf{A}}\|}{\lambda^+_{\min}\bigl([\overline{\mathbf{A}}; \mathbf{A}][\overline{\mathbf{A}}; \mathbf{A}]^\top\bigr)} \right)^2
$
 (see Appendix~\ref{sec:admm}),
which  corresponds to  $\kappa^2$ in our tests.

\begin{figure}[h] 
	\begin{center}
		\subfloat[$d=$100, $\kappa=$2]{\includegraphics[width=39mm]{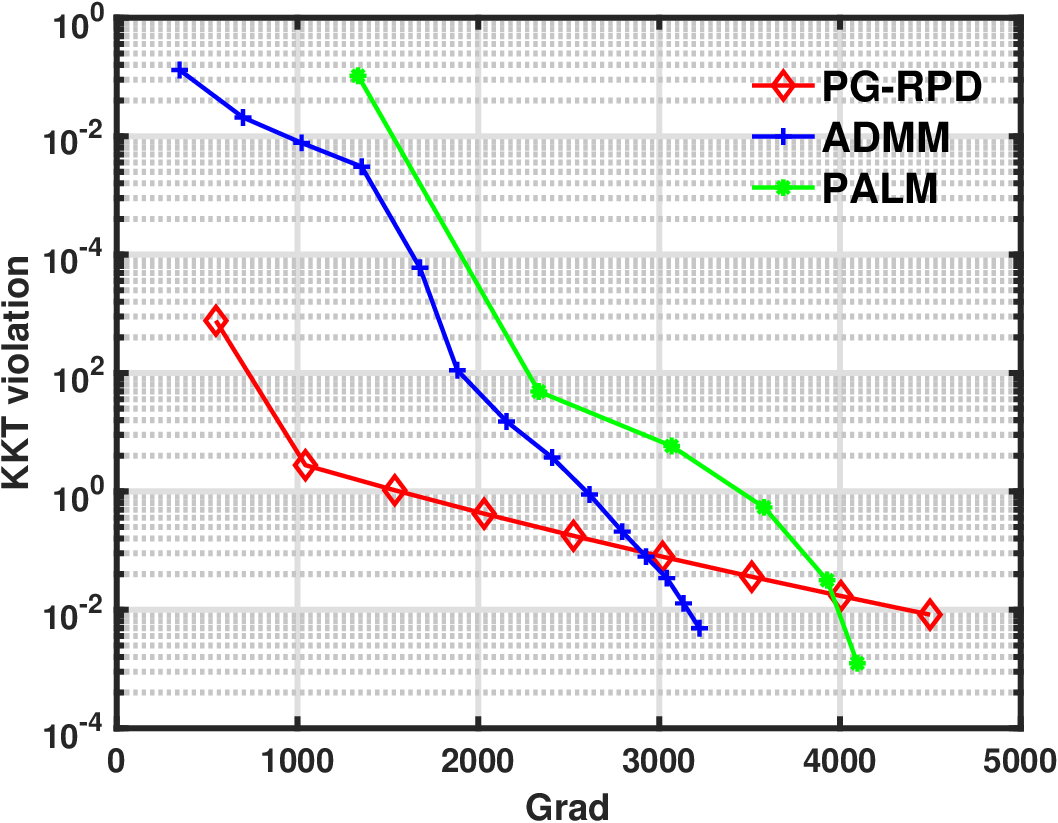}} 
		\subfloat[$d=$100, $\kappa=$10]{\includegraphics[width=39mm]{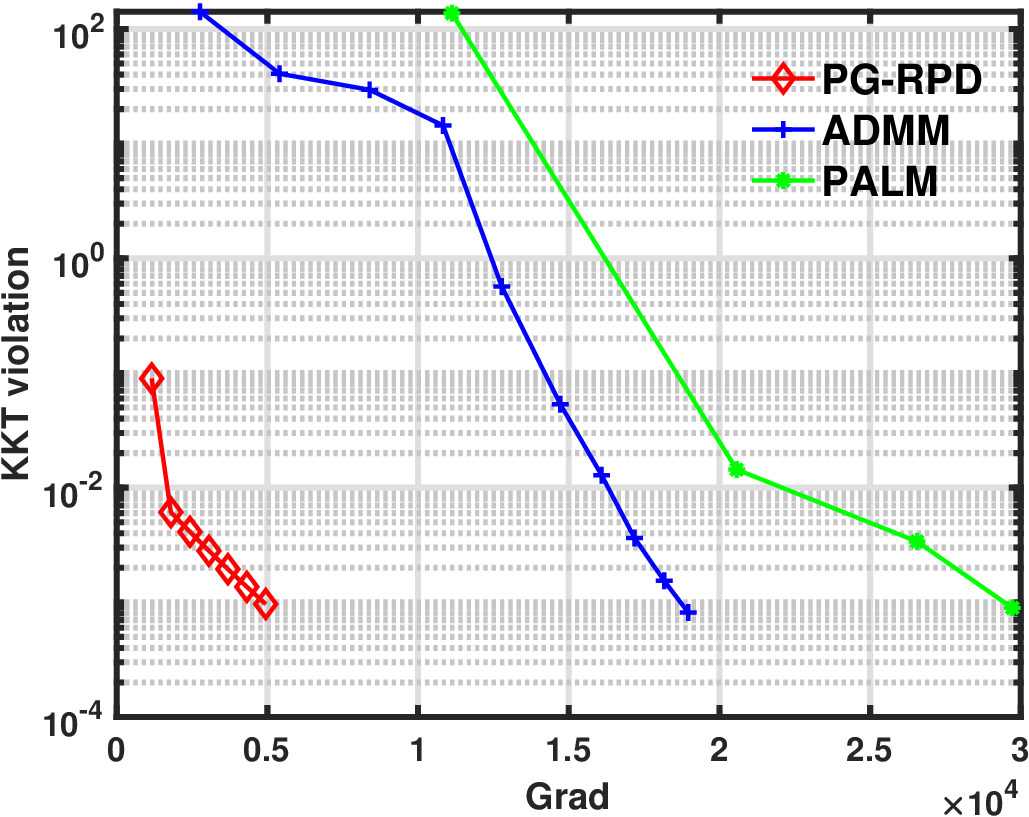}}	 
		\subfloat[$d=$100, $\kappa=$100]{\includegraphics[width=39mm]{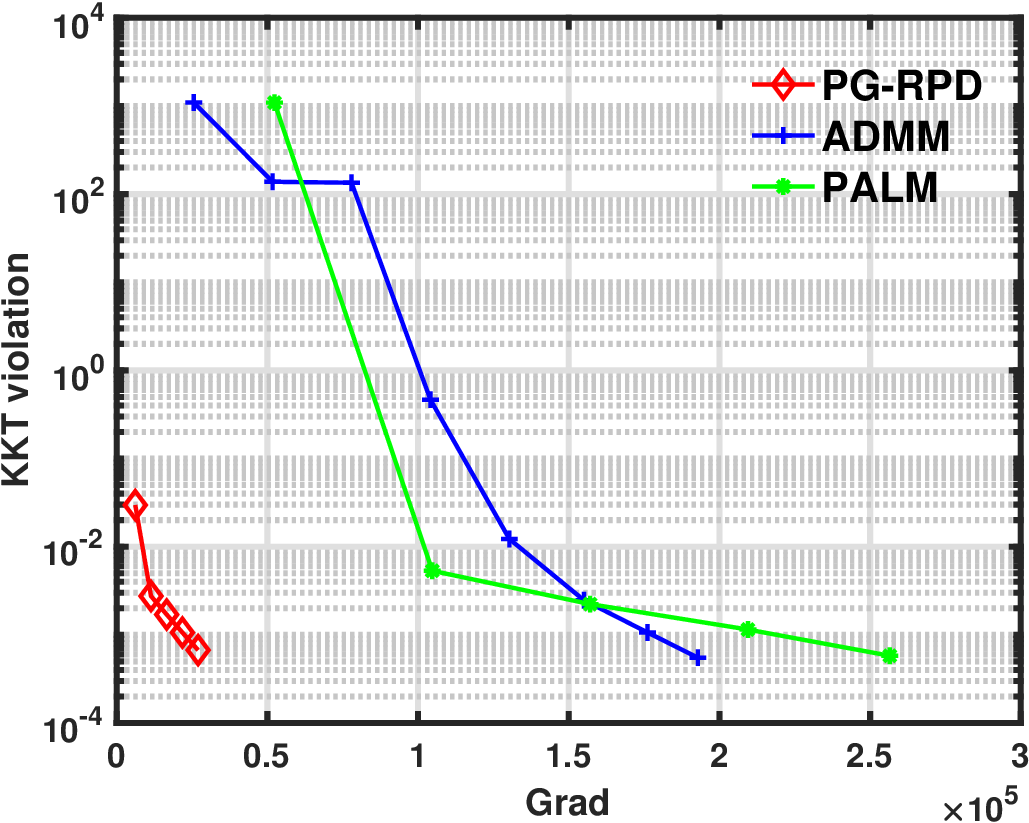}}	 
		\subfloat[$d=$100, $\kappa=$10000]{\includegraphics[width=39mm]{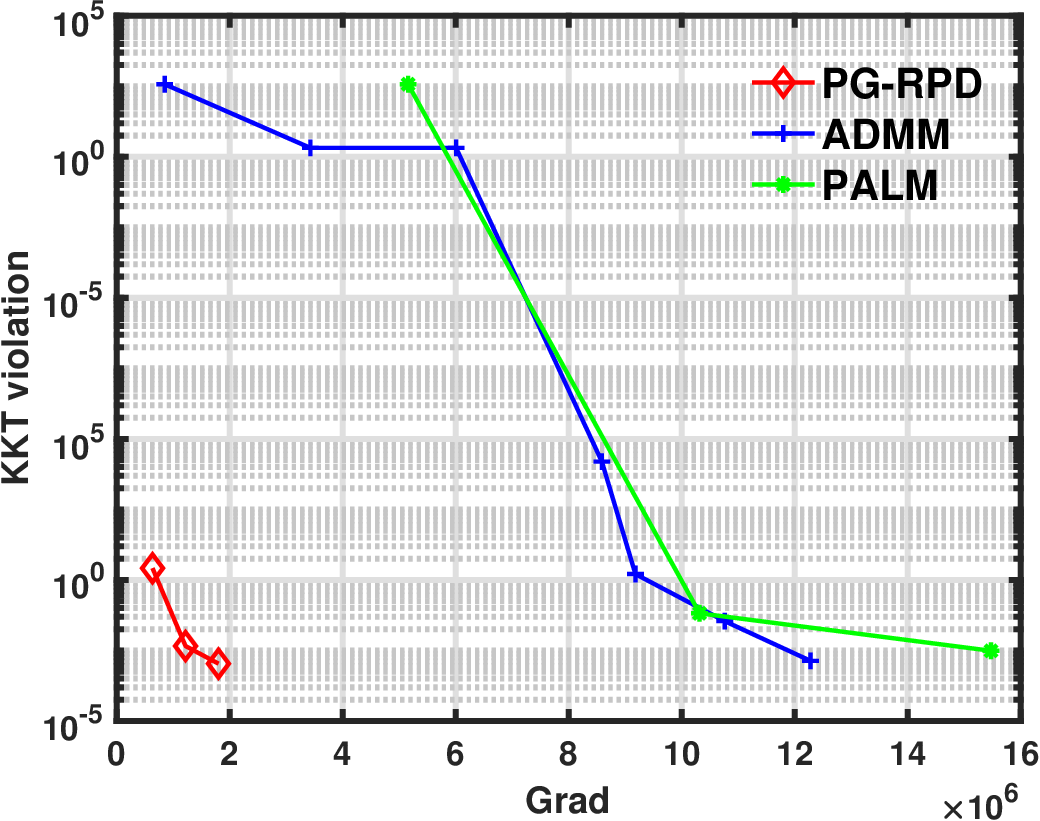}}
		\\
		\subfloat[$d=$1000, $\kappa=$2]{\includegraphics[width=39mm]{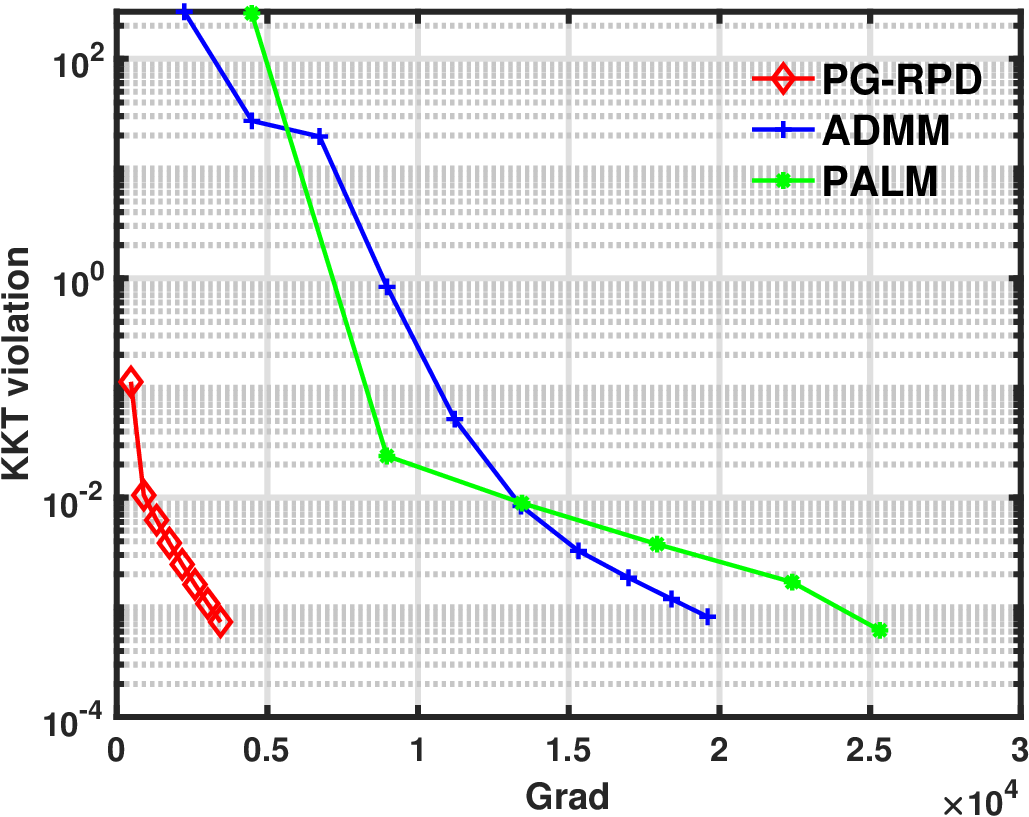}}
		\subfloat[$d=$1000, $\kappa=$10]{\includegraphics[width=39mm]{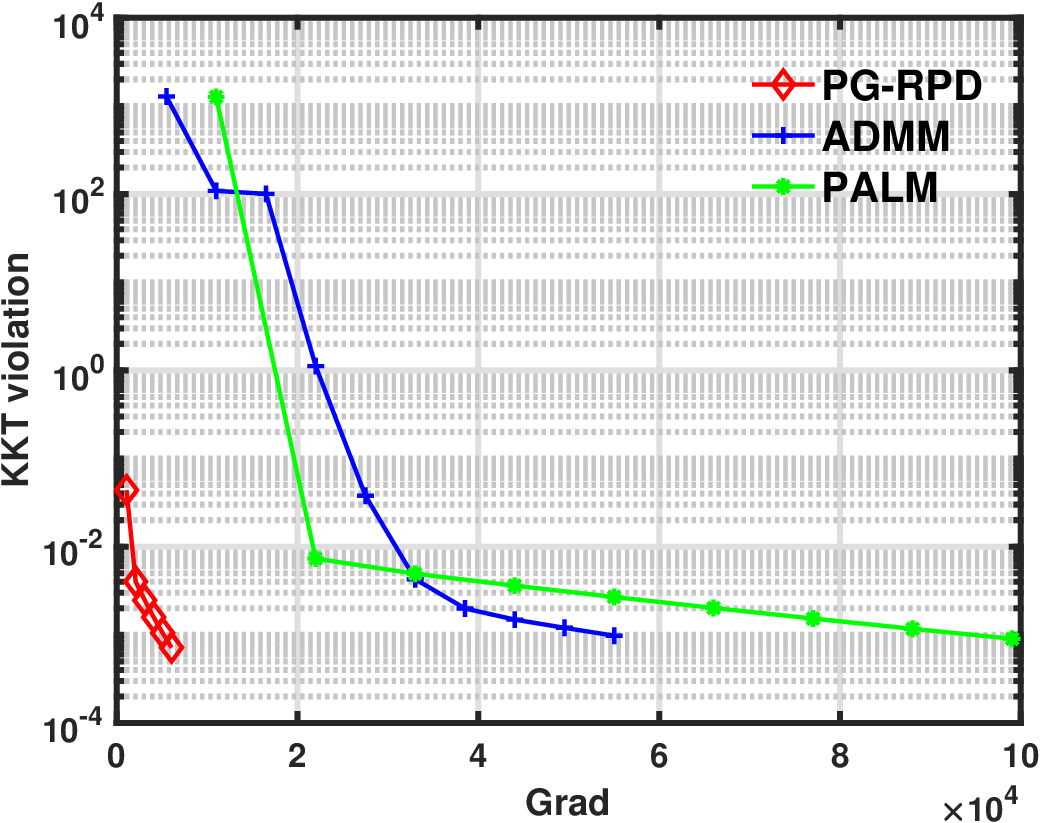}}
		\subfloat[$d=$1000, $\kappa=$100]{\includegraphics[width=39mm]{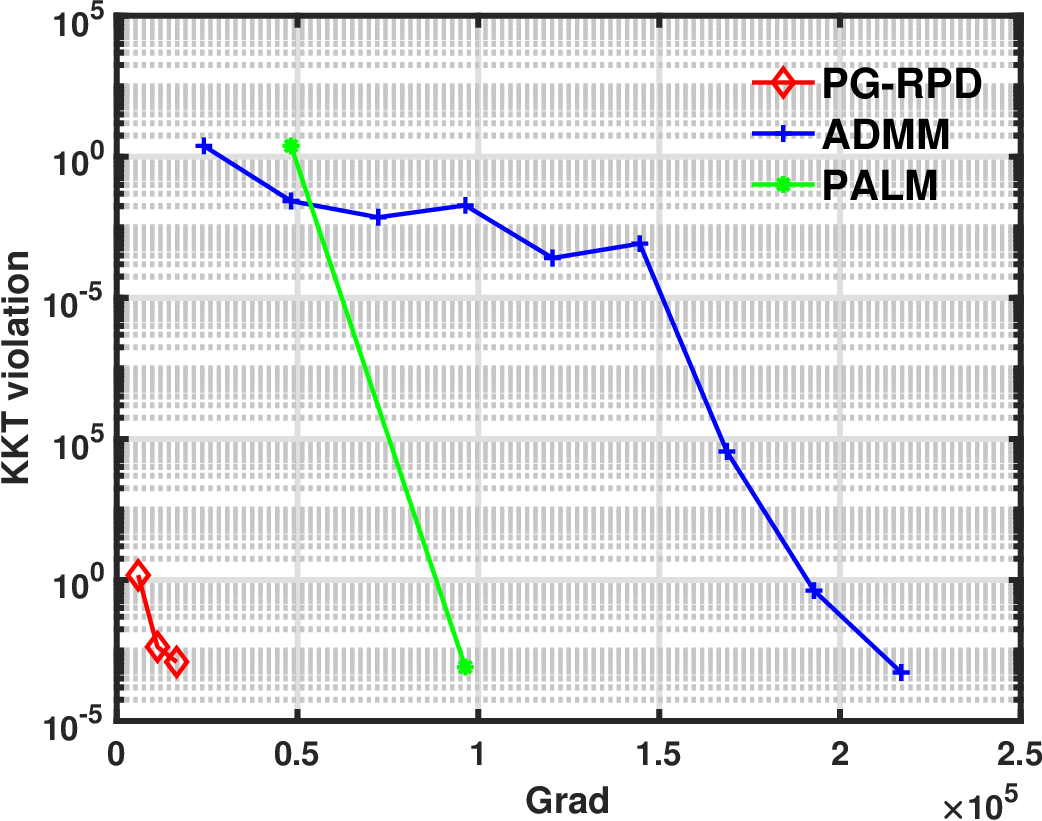}}	
		\subfloat[$d=$1000, $\kappa=$10000]{\includegraphics[width=39mm]{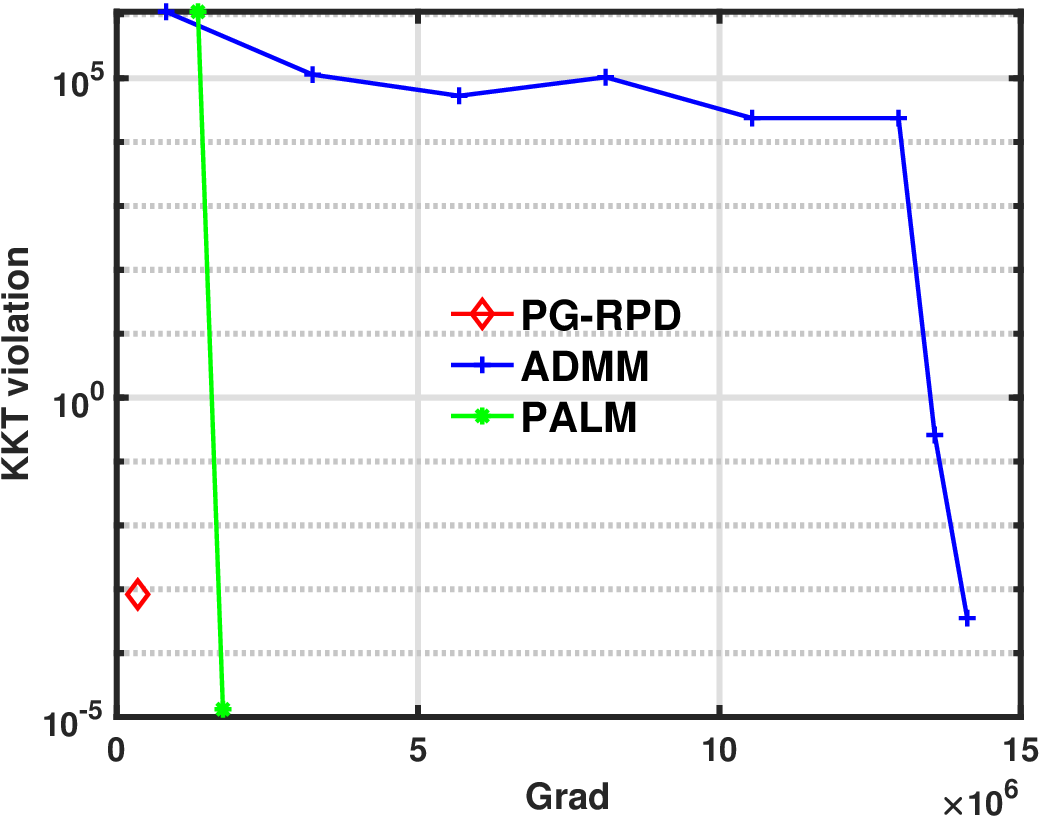}}
		\\
		\subfloat[$d=$2000, $\kappa=$2]{\includegraphics[width=39mm]{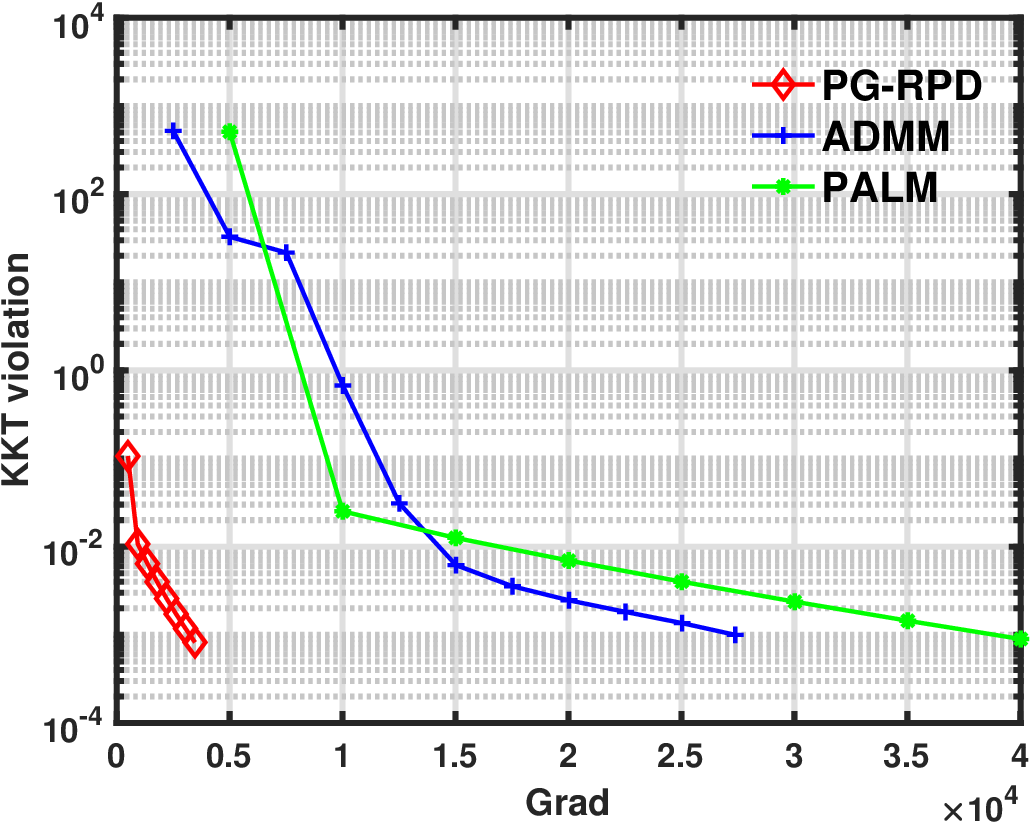}}
		\subfloat[$d=$2000, $\kappa=$10]{\includegraphics[width=39mm]{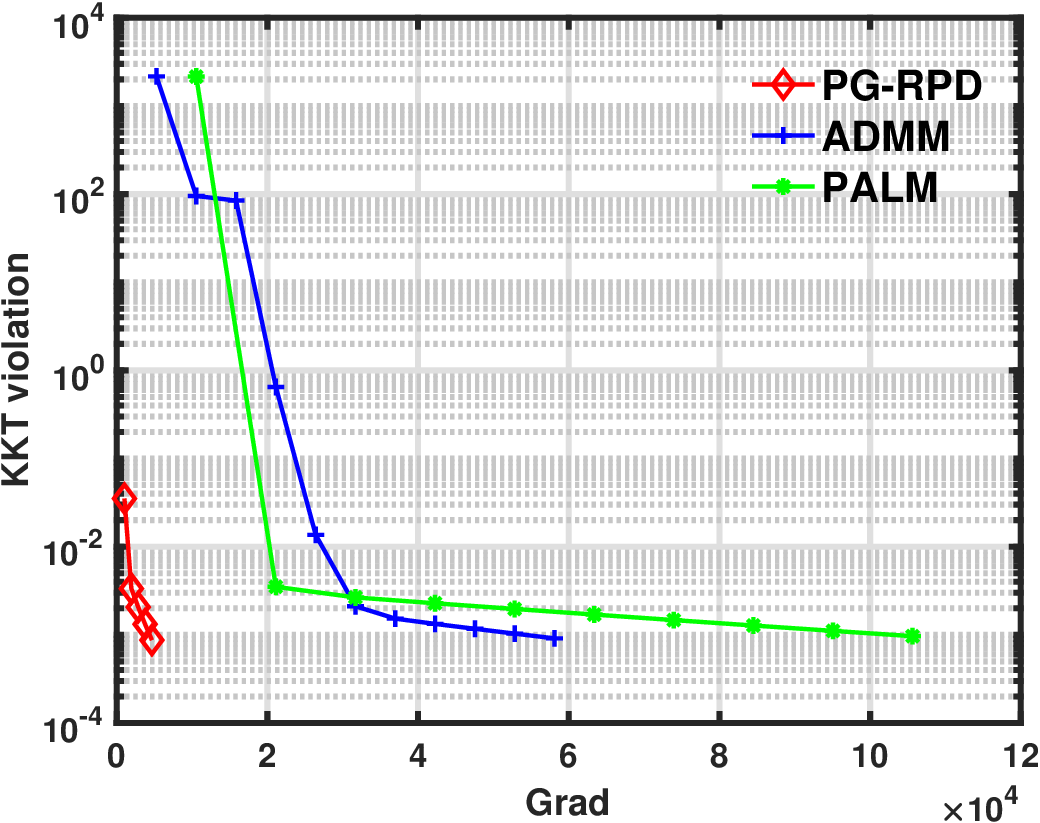}}
		\subfloat[$d=$2000, $\kappa=$100]{\includegraphics[width=39mm]{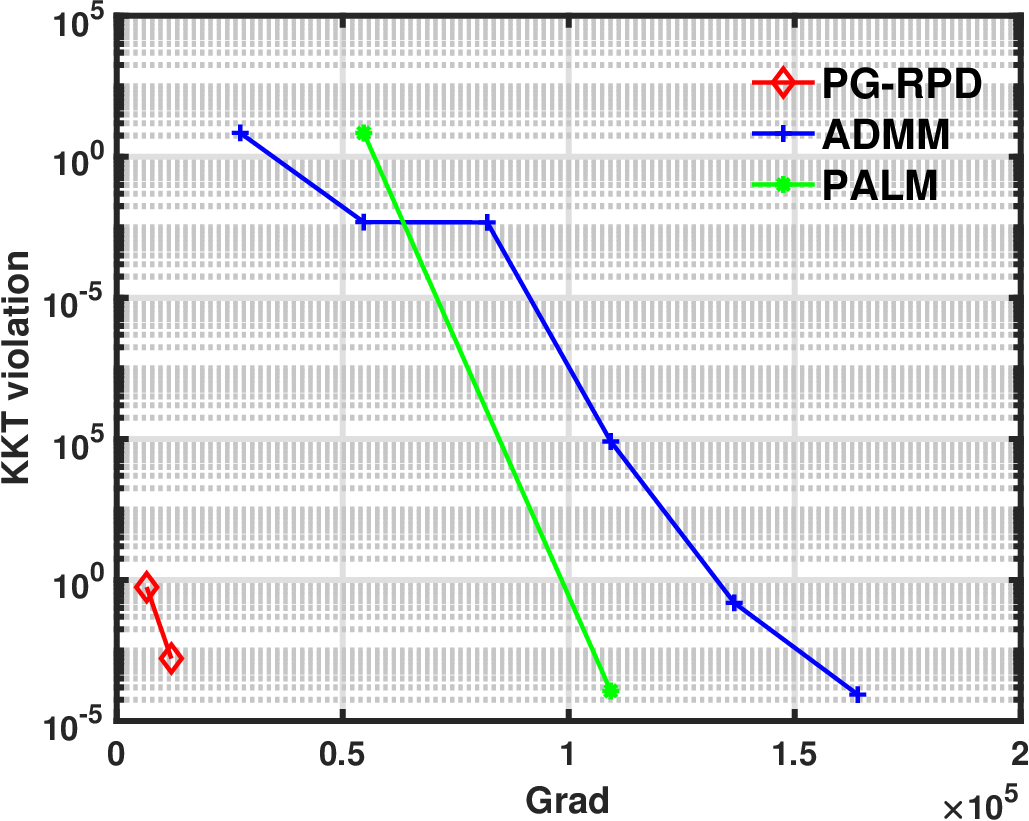}}
		\subfloat[$d=$2000, $\kappa=$10000]{\includegraphics[width=39mm]{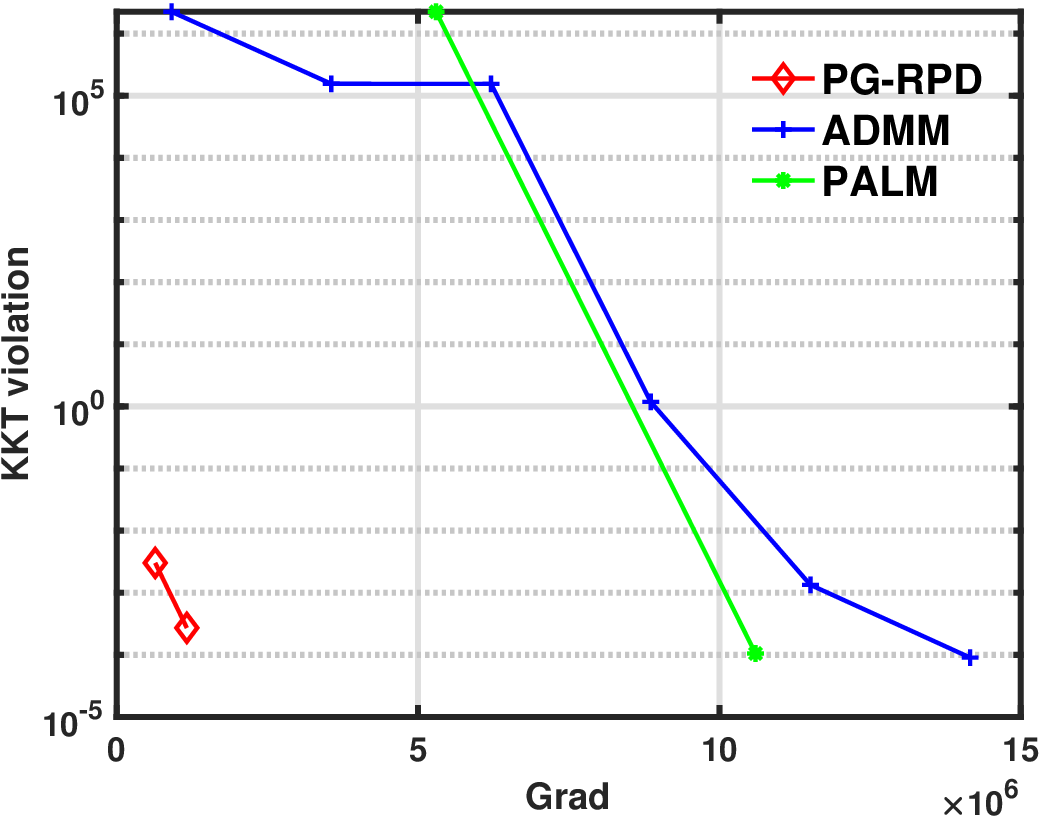}}
	\end{center} 
	\caption{Comparisons among PG-RPD, ADMM in \cite{melo2017iteration3}, and PALM in \cite{rockafellar1976augmented} on solving instances of \eqref{eq:test} with weak convexity modulus $\rho=0.1$}\label{fig:all}
\end{figure}

\begin{figure}[h] 
	\begin{center}
		\subfloat[$d=$100, $\kappa=$2]{\includegraphics[width=39mm]{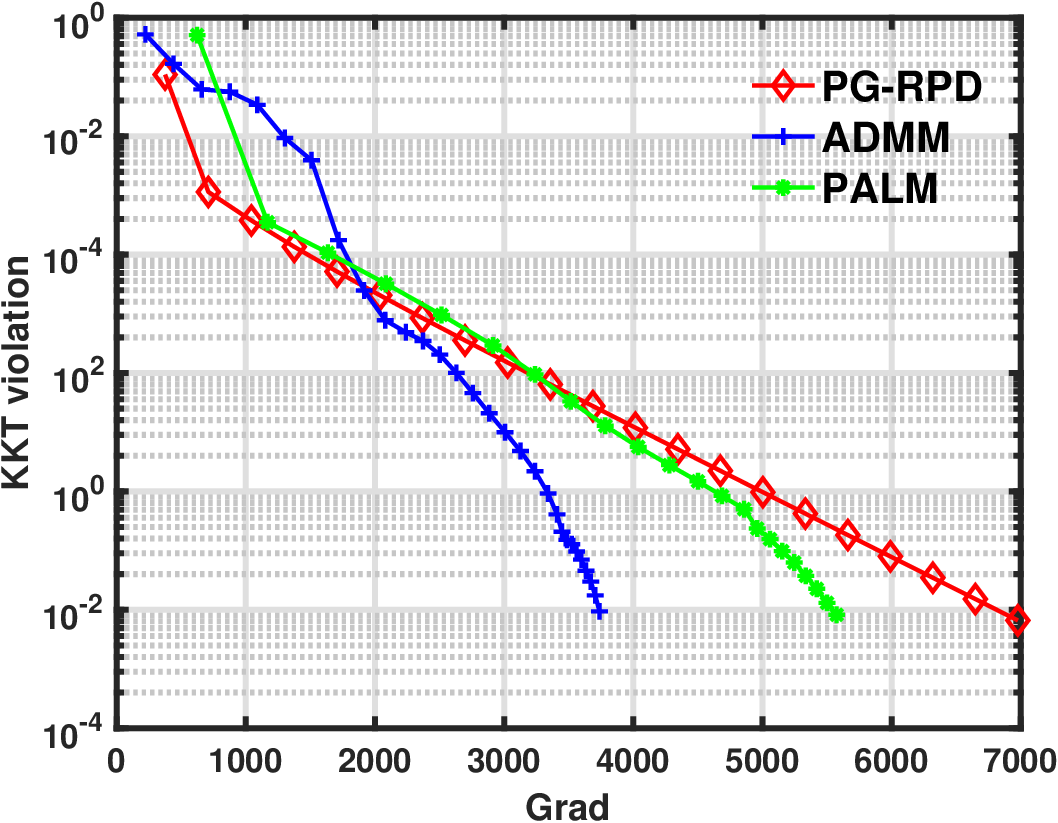}} 
		\subfloat[$d=$100, $\kappa=$10]{\includegraphics[width=39mm]{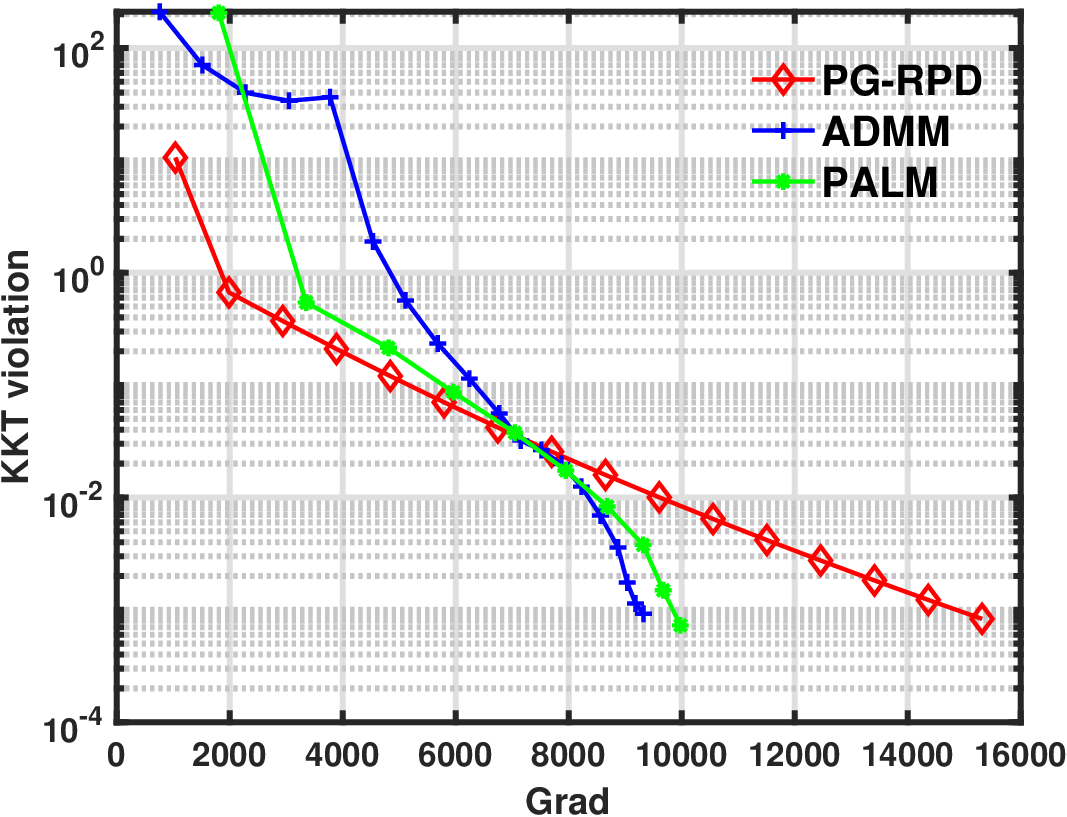}} 
		\subfloat[$d=$100, $\kappa=$100]{\includegraphics[width=39mm]{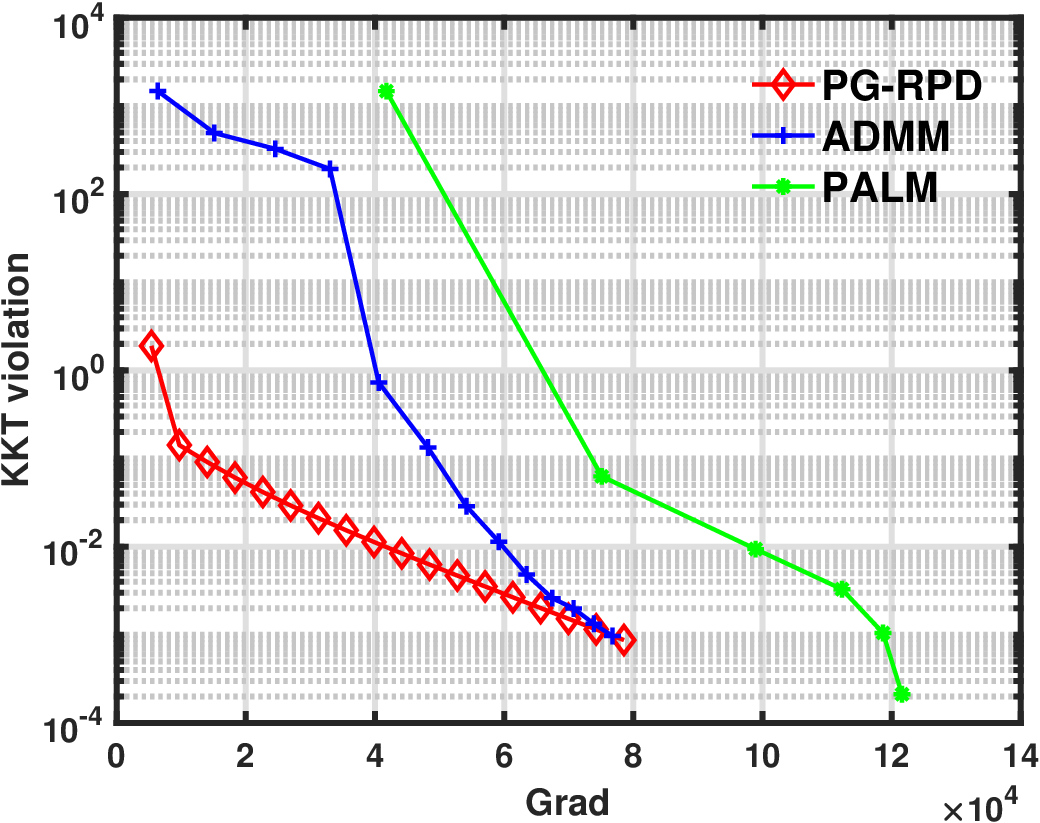}}
		\subfloat[$d=$100, $\kappa=$10000]{\includegraphics[width=39mm]{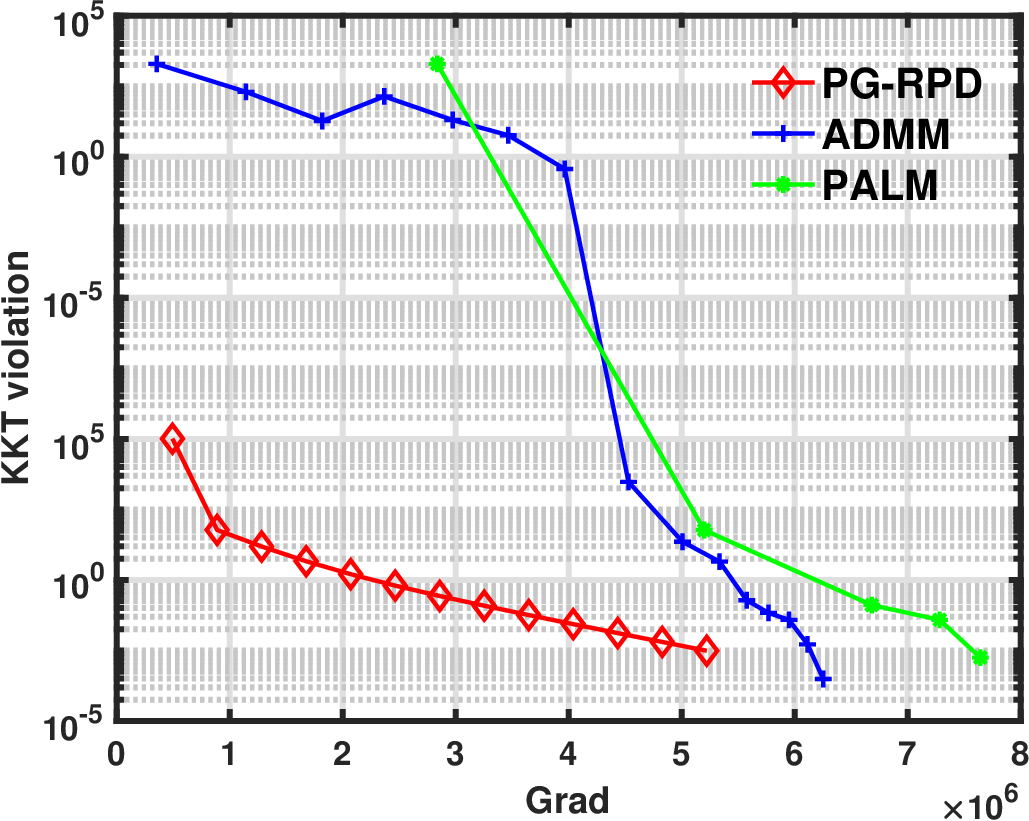}}
		\\
		\subfloat[$d=$1000, $\kappa=$2]{\includegraphics[width=39mm]{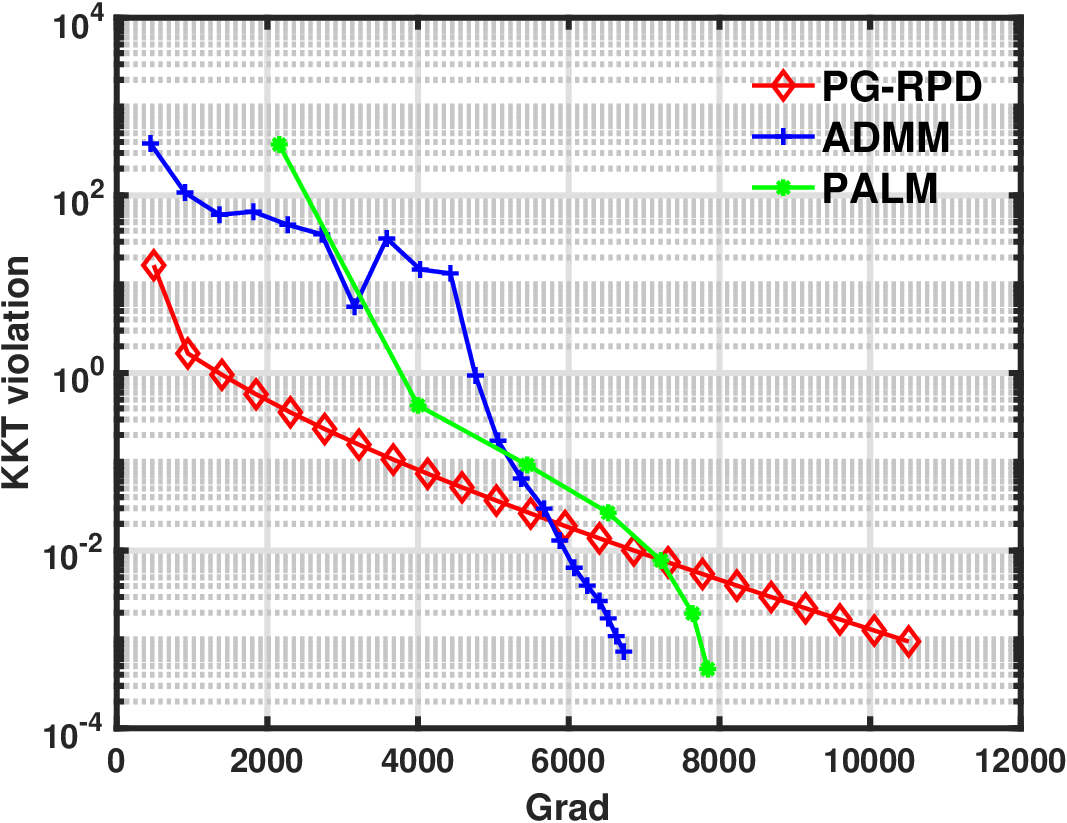}}
\subfloat[$d=$1000, $\kappa=$10]{\includegraphics[width=39mm]{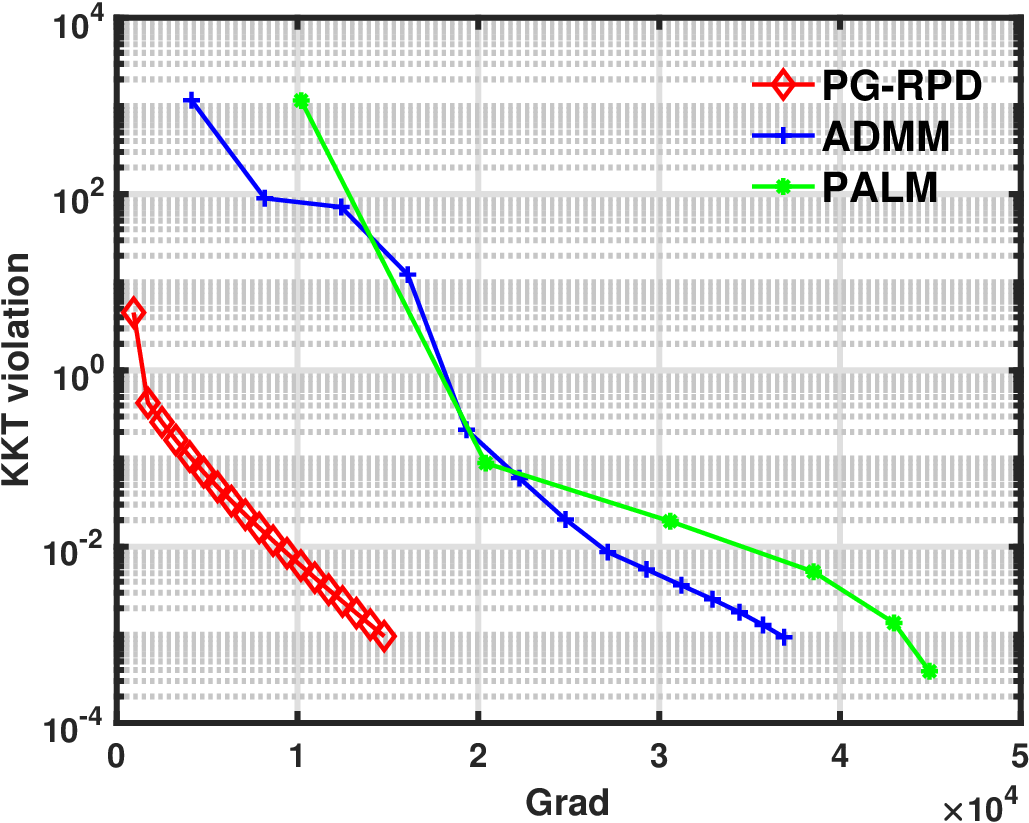}}
\subfloat[$d=$1000, $\kappa=$100]{\includegraphics[width=39mm]{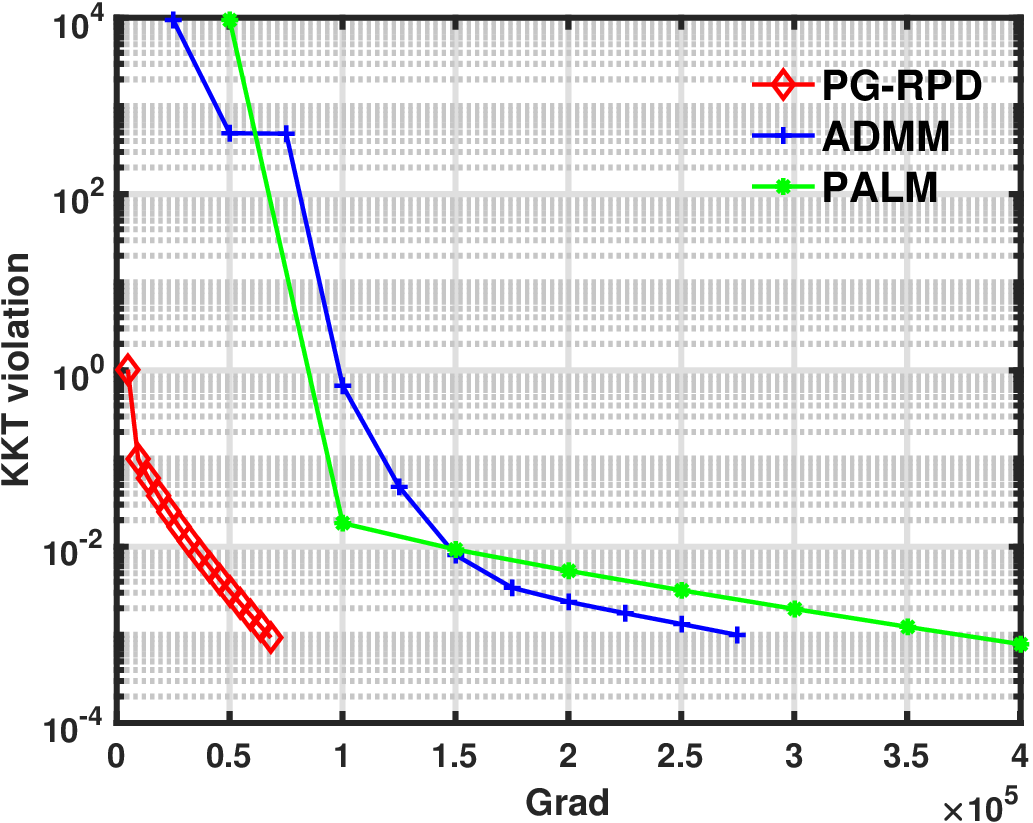}}
\subfloat[$d=$1000, $\kappa=$10000]{\includegraphics[width=39mm]{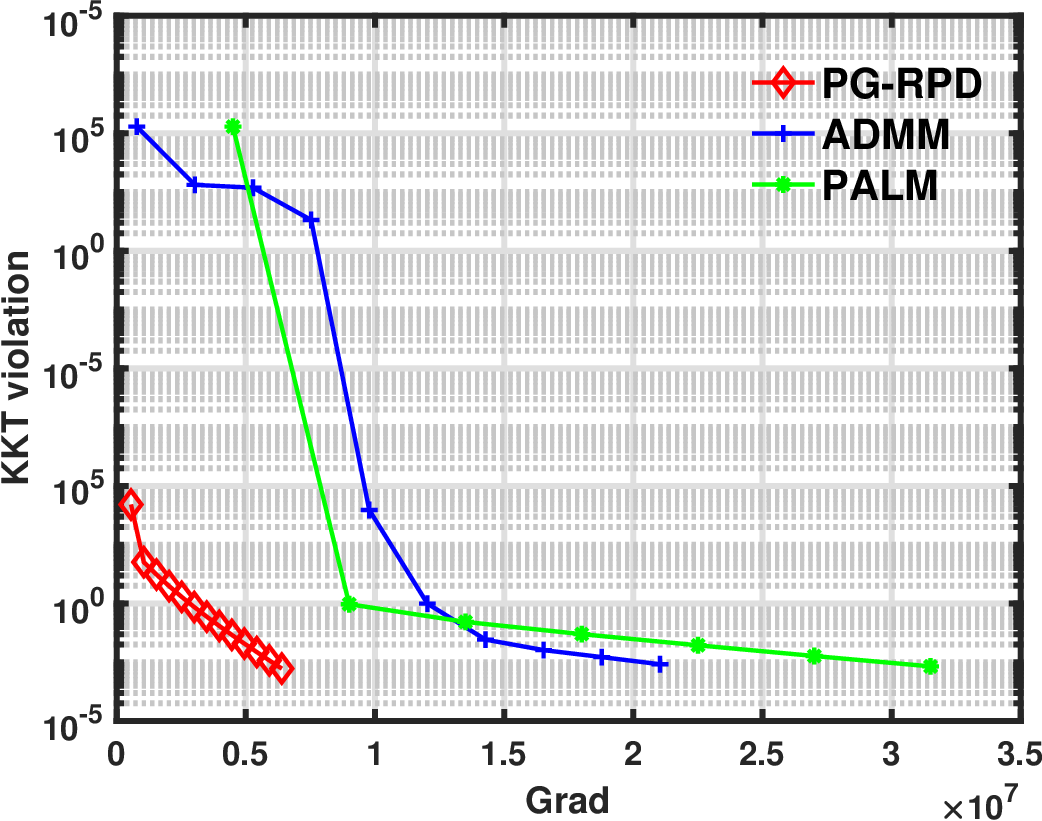}}
		\\
		\subfloat[$d=$2000, $\kappa=$2]{\includegraphics[width=39mm]{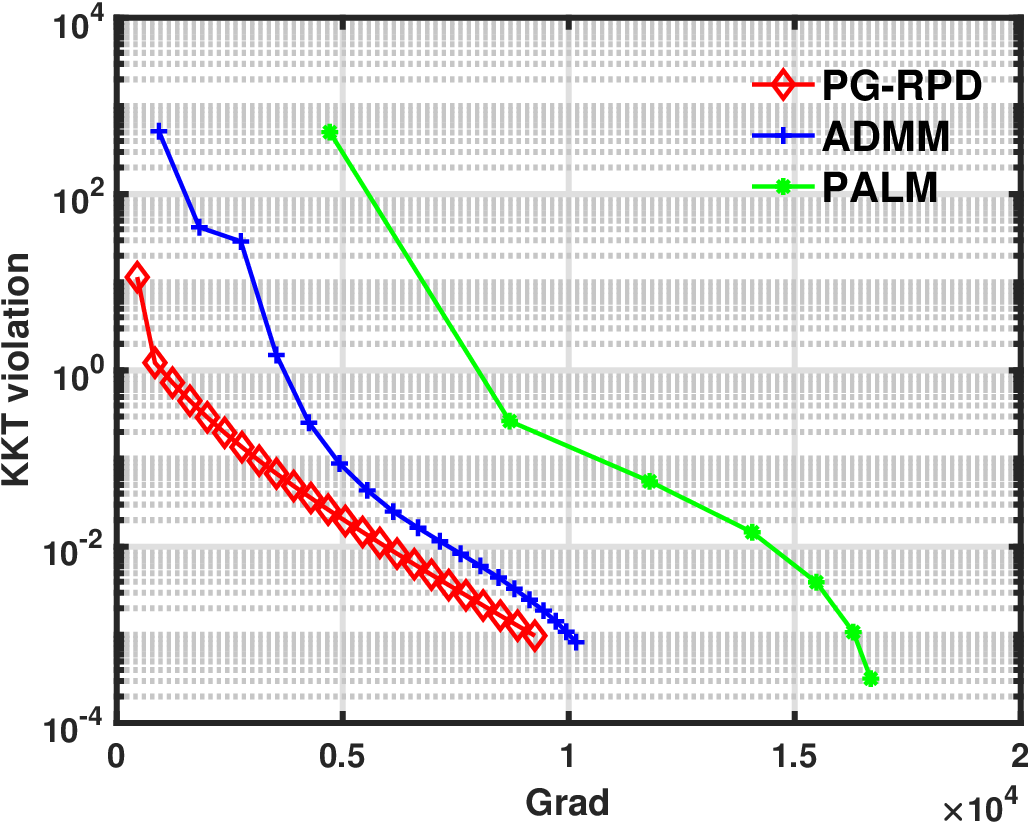}}
\subfloat[$d=$2000, $\kappa=$10]{\includegraphics[width=39mm]{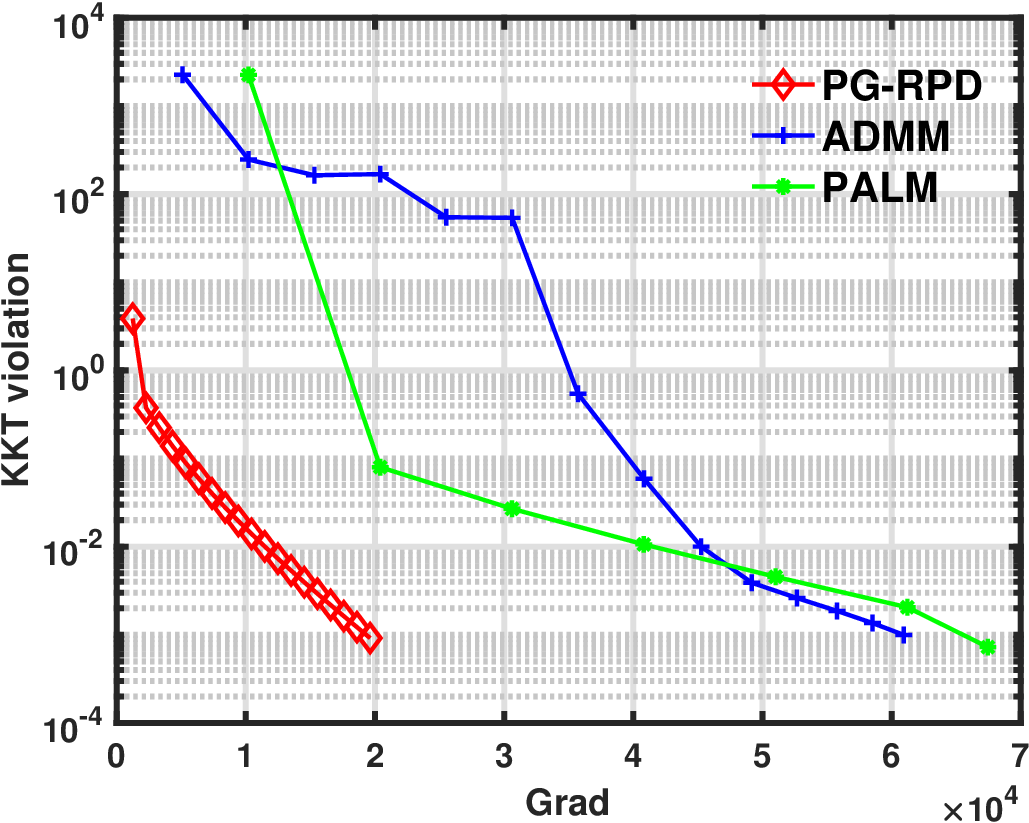}}
\subfloat[$d=$2000, $\kappa=$100]{\includegraphics[width=39mm]{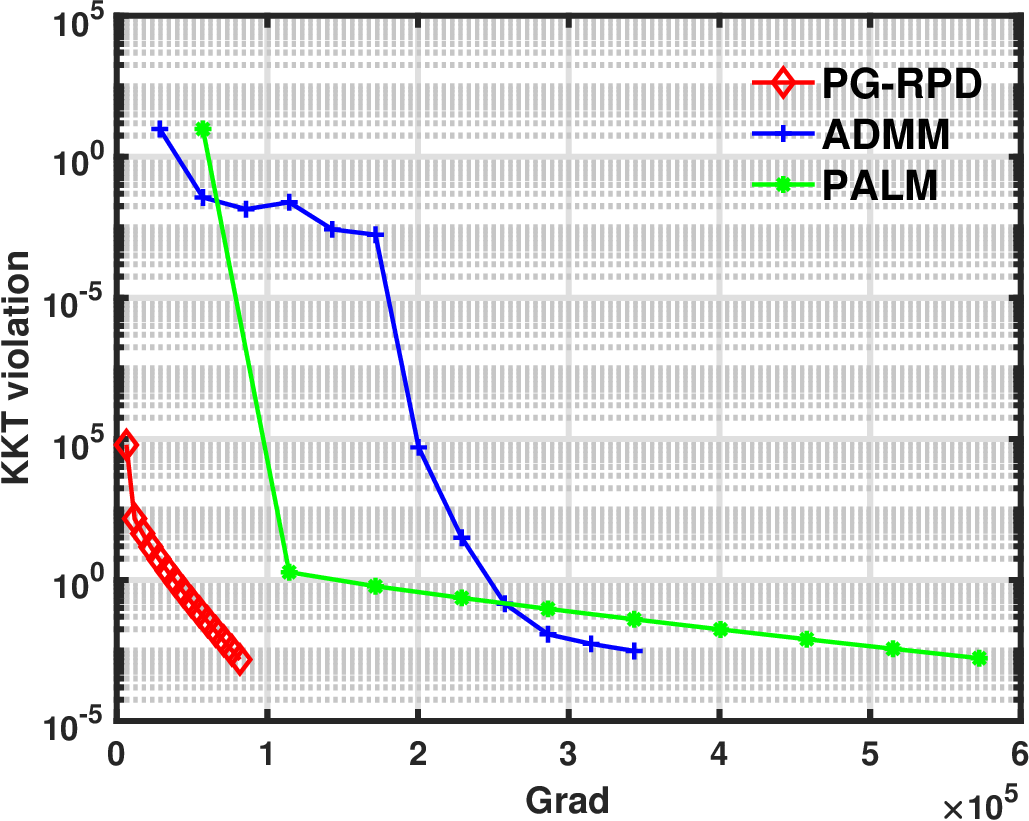}}
\subfloat[$d=$2000, $\kappa=$10000]{\includegraphics[width=39mm]{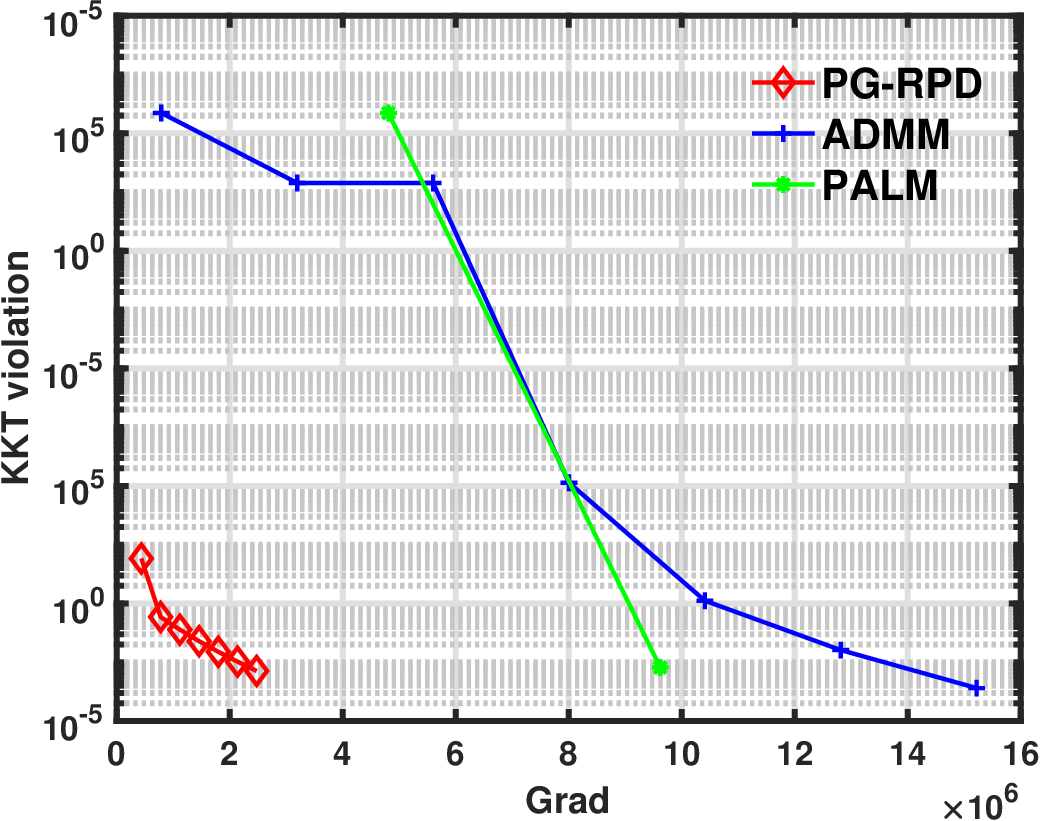}}
	\end{center} 
	\caption{Comparisons among PG-RPD, ADMM in \cite{melo2017iteration3}, and PALM in \cite{rockafellar1976augmented} on solving instances of \eqref{eq:test} with weak convexity modulus $\rho=1$}\label{fig:all2}
\end{figure}

\begin{figure}[h] 
	\begin{center}
		\subfloat[$d=$100, $\kappa=$2]{\includegraphics[width=39mm]{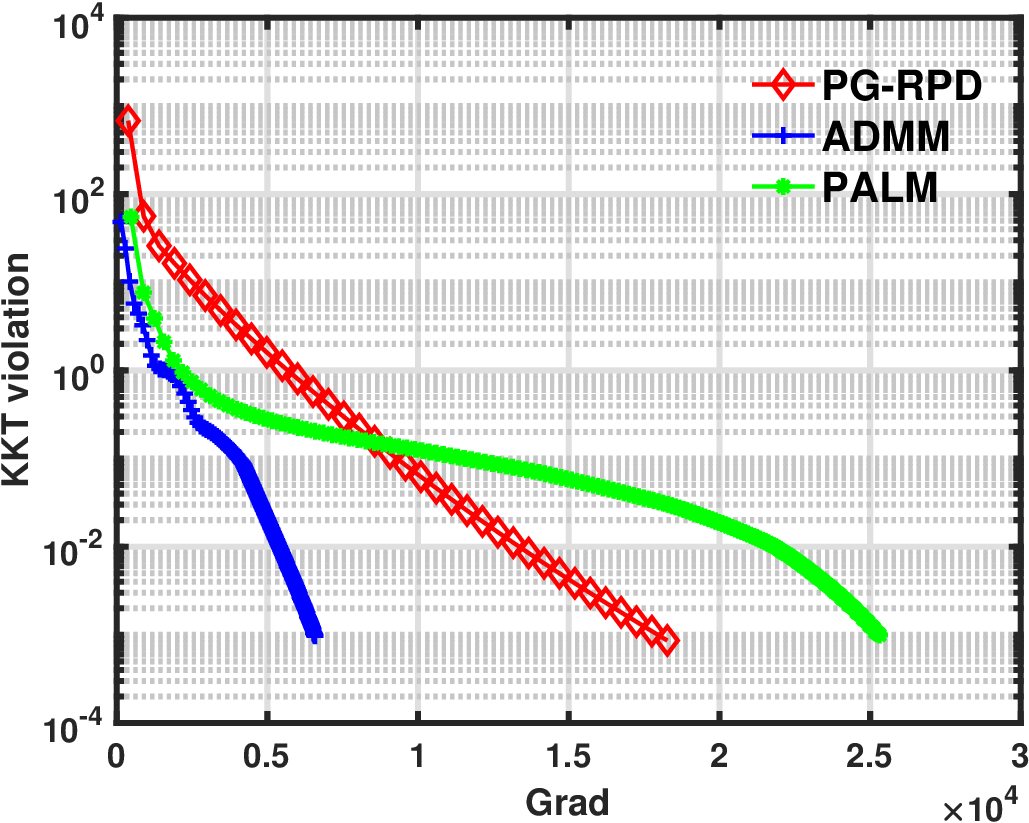}}
		\subfloat[$d=$100, $\kappa=$10]{\includegraphics[width=39mm]{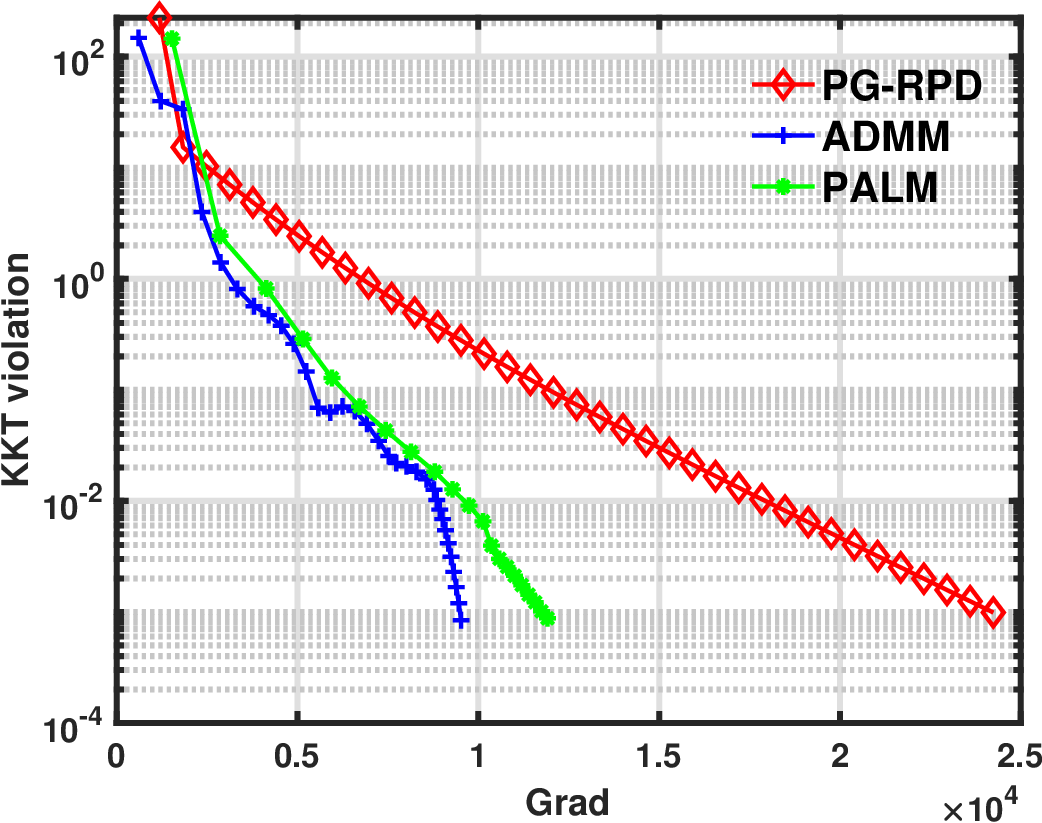}}
		\subfloat[$d=$100, $\kappa=$100]{\includegraphics[width=39mm]{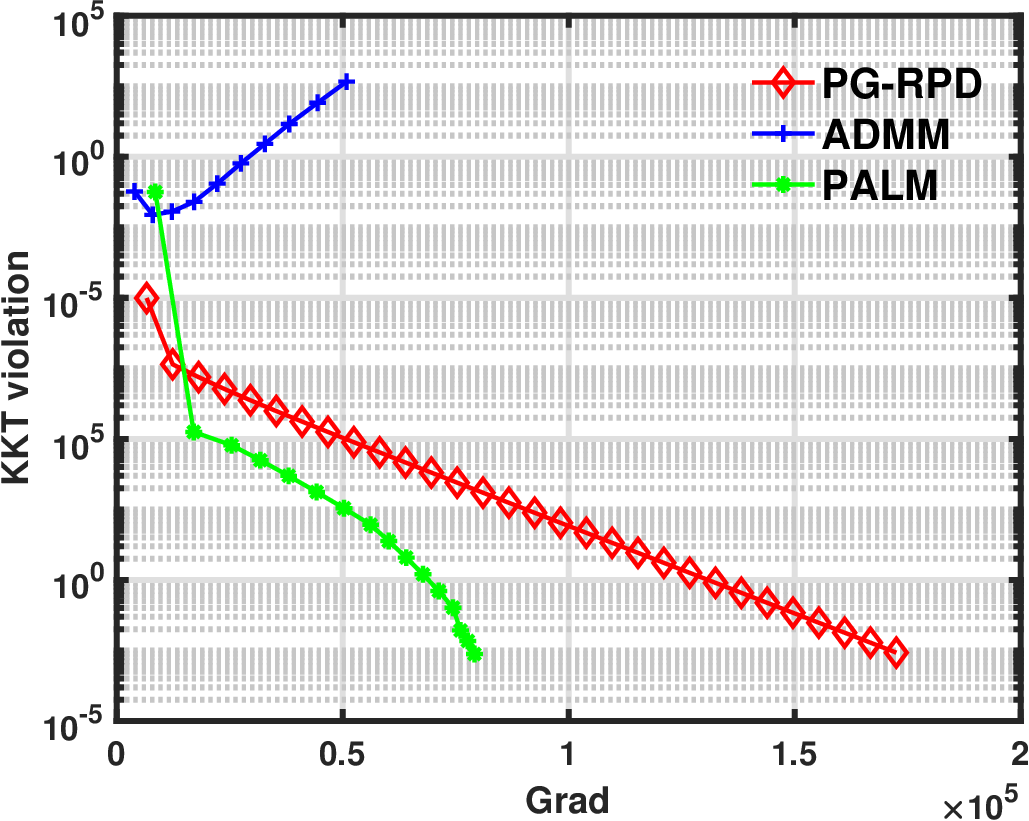}}
		\subfloat[$d=$100, $\kappa=$10000]{\includegraphics[width=39mm]{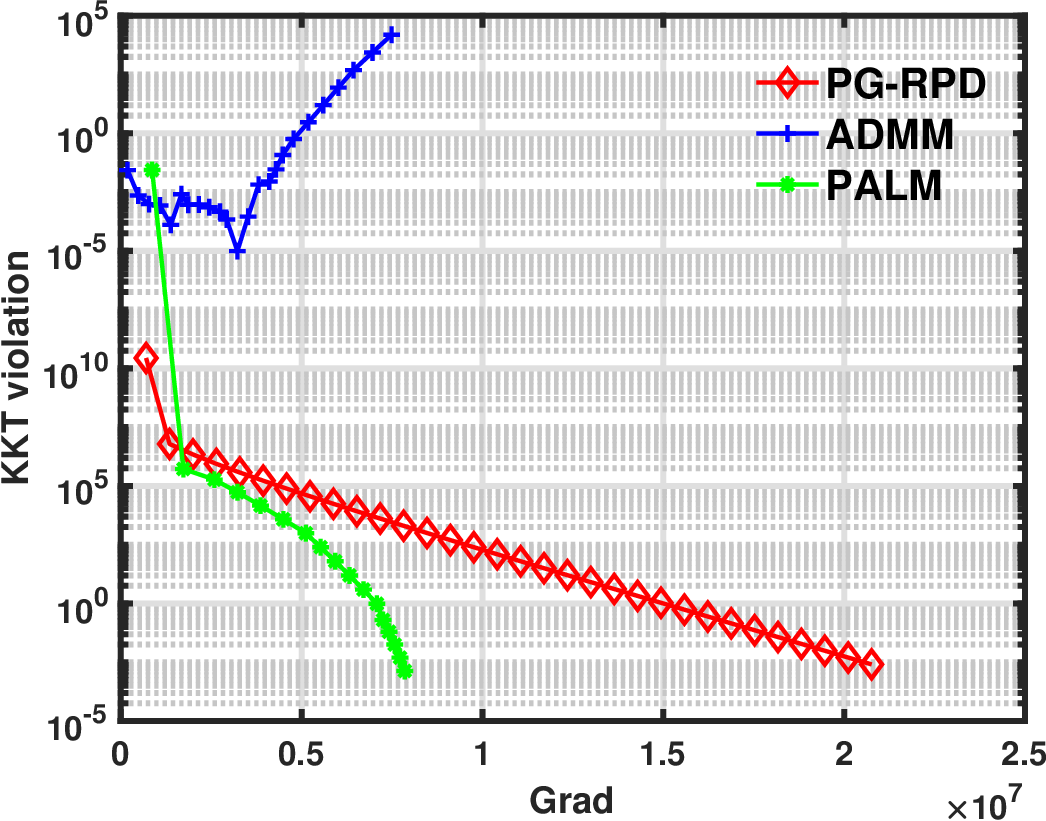}}
		\\
		\subfloat[$d=$1000, $\kappa=$2]{\includegraphics[width=39mm]{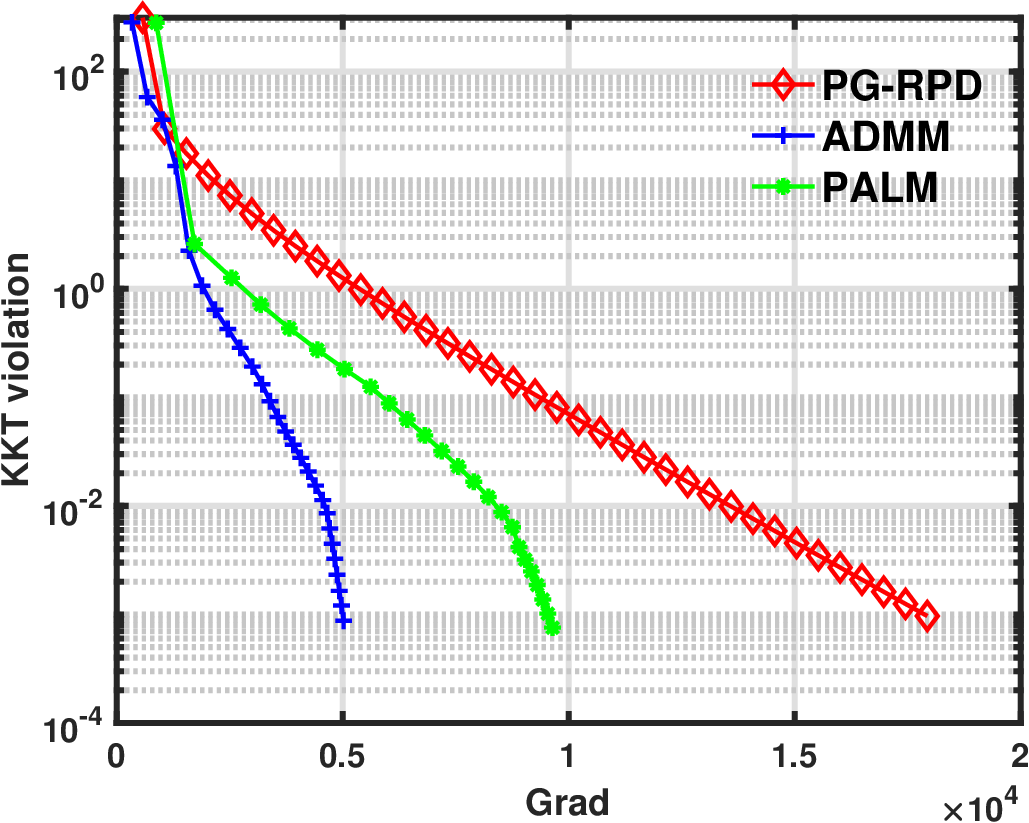}}
\subfloat[$d=$1000, $\kappa=$10]{\includegraphics[width=39mm]{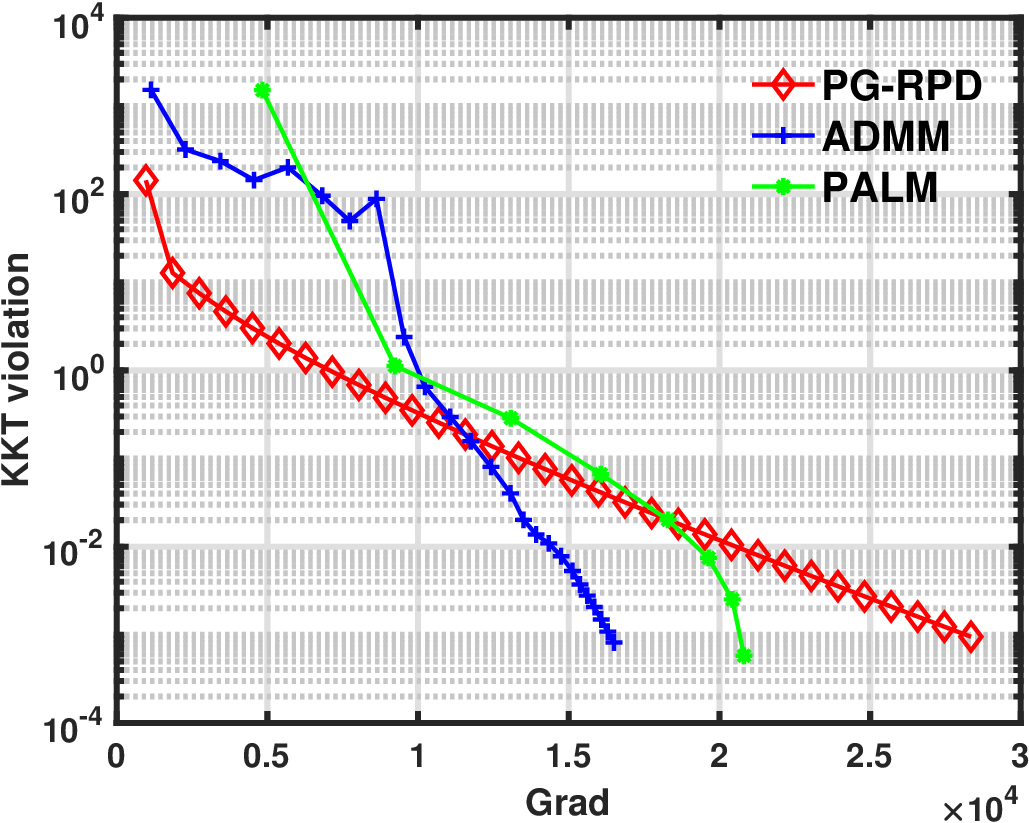}}
\subfloat[$d=$1000, $\kappa=$100]{\includegraphics[width=39mm]{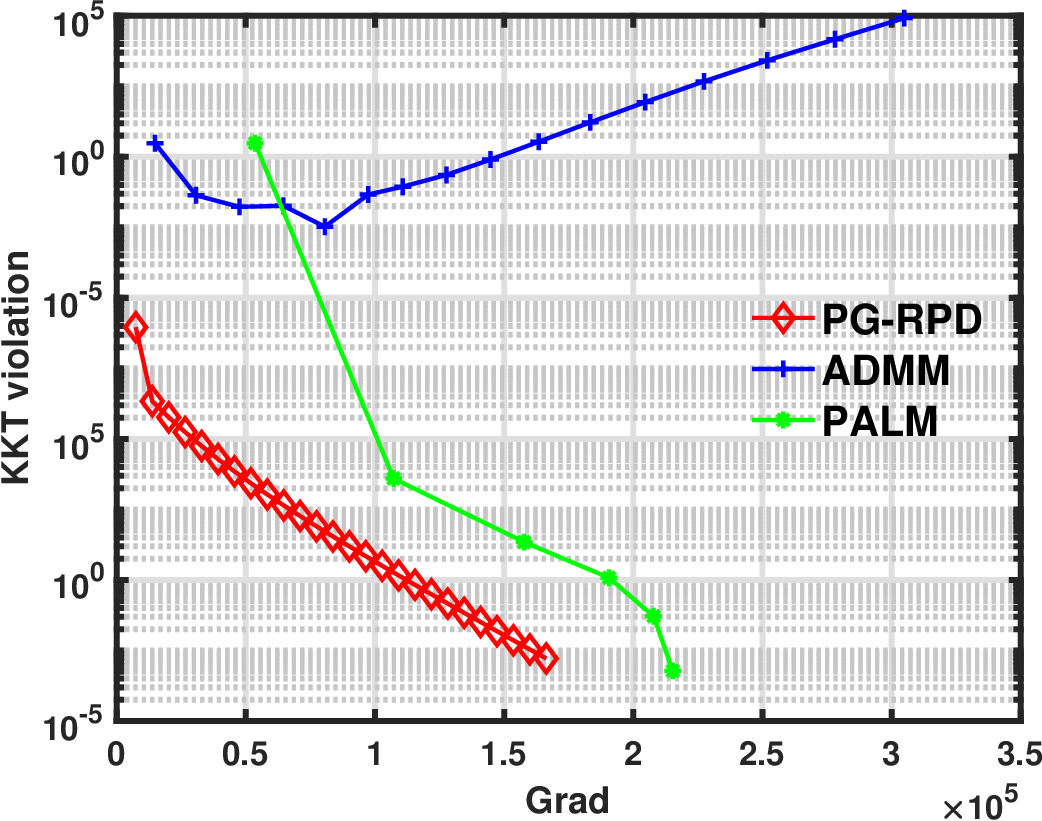}}
\subfloat[$d=$1000, $\kappa=$10000]{\includegraphics[width=39mm]{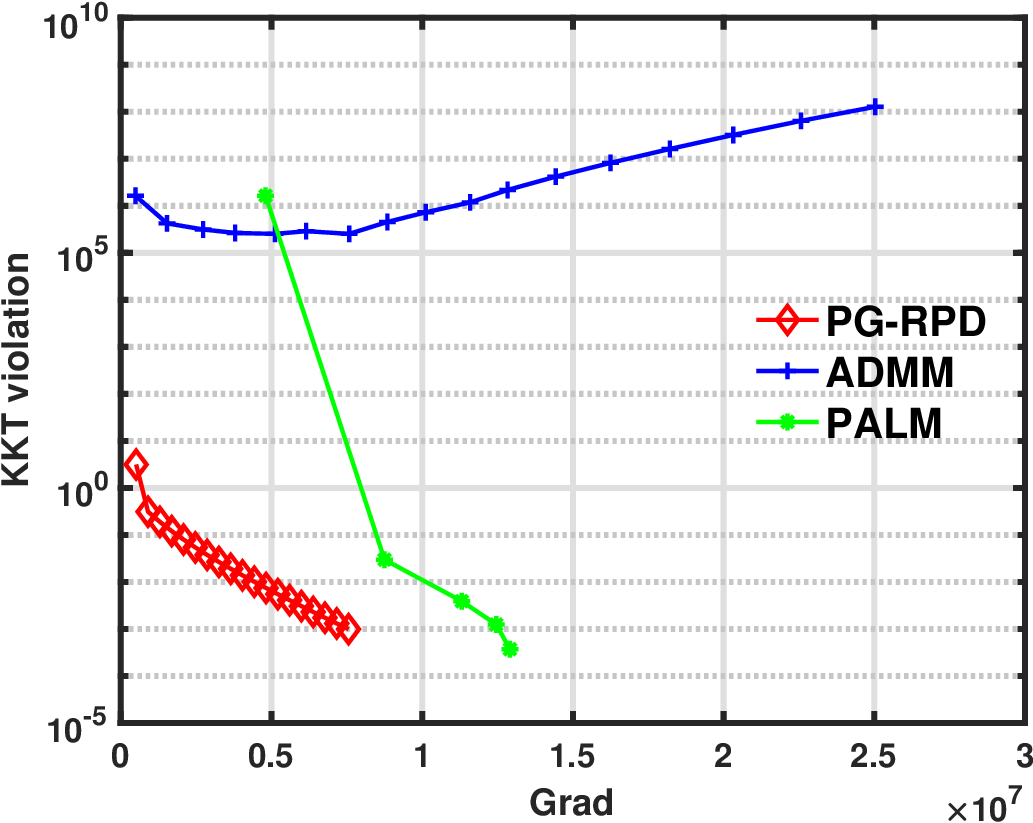}}
		\\
		\subfloat[$d=$2000, $\kappa=$2]{\includegraphics[width=39mm]{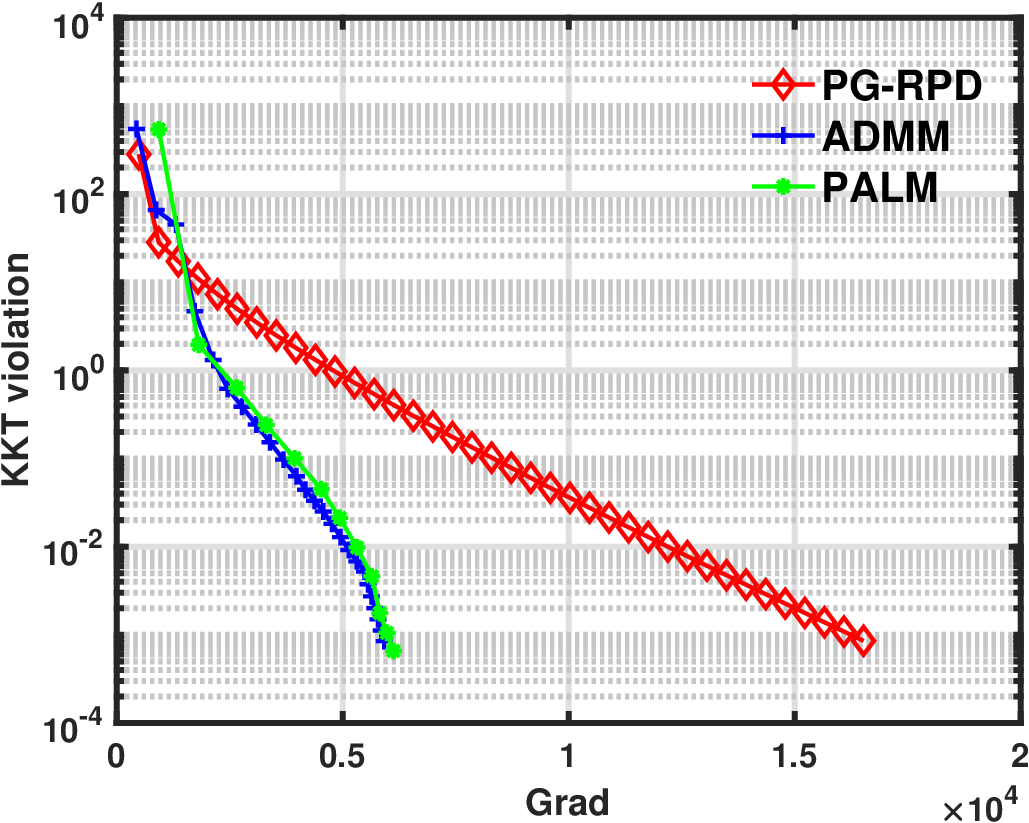}}
\subfloat[$d=$2000, $\kappa=$10]{\includegraphics[width=39mm]{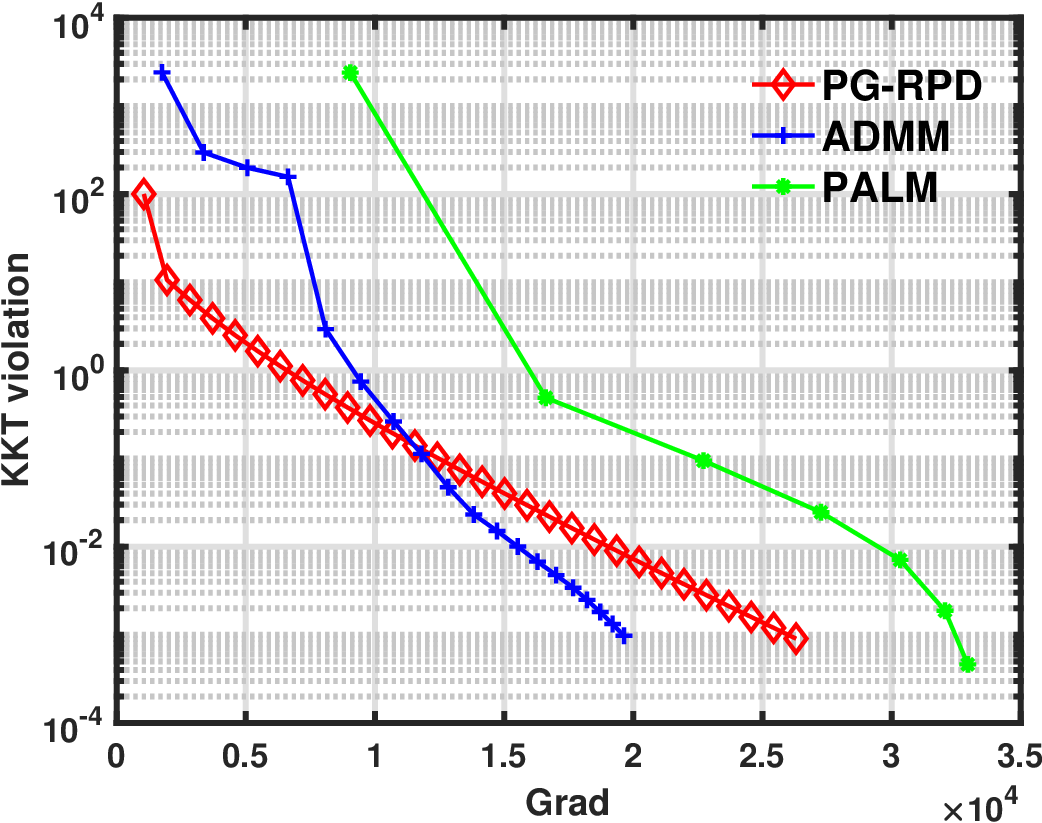}}
\subfloat[$d=$2000, $\kappa=$100]{\includegraphics[width=39mm]{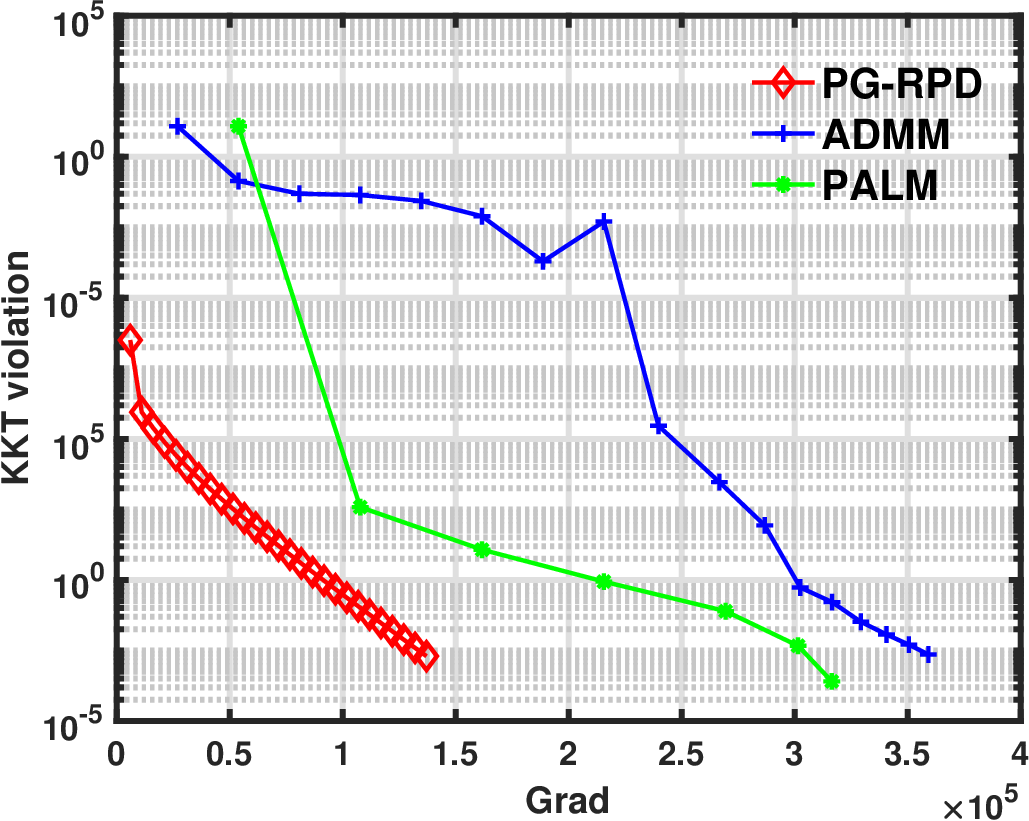}}
\subfloat[$d=$2000, $\kappa=$10000]{\includegraphics[width=39mm]{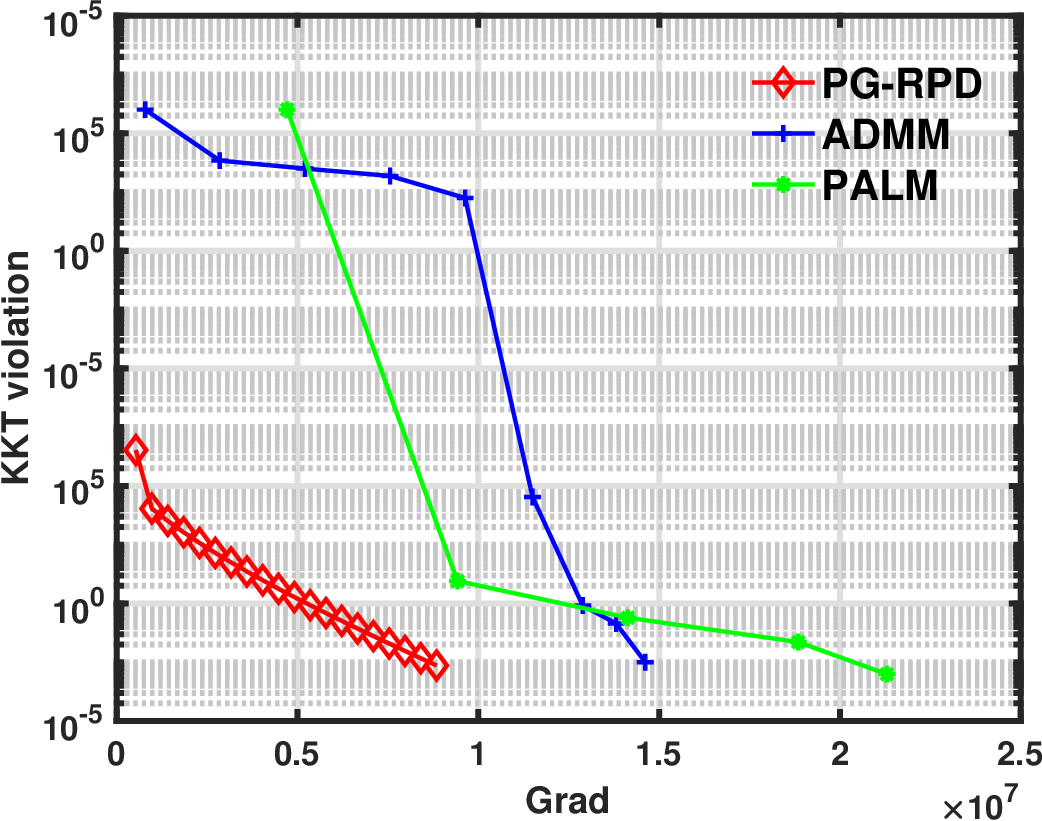}}
	\end{center} 
	\caption{Comparisons among PG-RPD, ADMM in \cite{melo2017iteration3}, and PALM in \cite{rockafellar1976augmented} on solving instances of \eqref{eq:test} with weak convexity modulus $\rho=5$}\label{fig:all3}
\end{figure}

\section{Conclusion}\label{sec:conclusion}
In this paper, we propose an inexact proximal gradient method for solving problem~\eqref{eq:model}, where subproblems are solved by a primal-dual recover method. 
Our method achieves a (nearly) optimal complexity for finding a $(\delta,\varepsilon)$-KKT point of problem \eqref{eq:model}, improving upon existing approaches in terms of the gradient Lipschitz modulus  $L_f$ and the condition number $\kappa([\overline{\vA};\vA])$. Numerical experiments further demonstrate the efficiency of our method, particularly when   $\kappa([\overline{\vA};\vA])$ is large.


%


			\bibliographystyle{plain}
			\bibliography{optim}
			
	\appendix
	
		\section{Comparison of PG-RPD with an ADMM and an ALM}
	\label{sec:matching}
	In this section, we first introduce the hard instance used to establish the lower complexity results in \cite{liu2025lowercomplexityboundsfirstorder}. We  then compare our method with two other algorithms in the subsequent subsections.
	
	\subsection{A hard instance}
	In this subsection, we introduce the hard instance $\cal{P}$ given in \cite{liu2025lowercomplexityboundsfirstorder}. 	
	An algorithm is considered to have a tight complexity result (at least for this specialized problem used to establish the lower complexity bound) if it can find an $(\vareps, \vareps)$-KKT point of this specialized problem within $
	\mathcal{O}({\kappa([\overline{\mathbf{A}}; \mathbf{A}]) L_f \Delta_{F_0}} \vareps^{-2})
	$ oracle calls up to a logarithmic term.
	
	Let $m_1$, $m_2 $, $\overline d \ge 5$ be positive integers, $m=m_1m_2$, $d=m{\overline{d}}$, and   
	\begin{eqnarray} 
		\label{eq:xblock}
		\vx=\left(\vx_1\zz, \ldots,\vx_m\zz\right)\zz\in\mathbb{R}^{d}, \text{ with }\vx_i\in\mathbb{R}^{\overline d}, \, i=1,\dots,m.
	\end{eqnarray} 
	We define a matrix  $\vH\in\mathbb{R}^{(m-1)\overline{d}\times m\overline{d}}$ by
	\begin{equation}
		\label{eq:matrixAstar}
		\newcommand{\zm}{
			\left[\begin{array}{ccccc}
				-\vI_{\overline{d}} & \vI_{\overline{d}} &&&\\
				& -\vI_{\overline{d}} &\vI_{\overline{d}} &&\\
				&& \ddots & \ddots &  \\
				&&  &-\vI_{\overline{d}} & \vI_{\overline{d}} \\
			\end{array}\right]
		}		
		\vH:=mL_f\cdot \vJ_{m}\otimes \vI_{\overline{d}}=
		mL_f\cdot
		\left.
		\,\smash[b]{\underbrace{\!\zm\!}_{\textstyle\text{$m$ blocks}}}\,
		\right\}\text{$m-1$ blocks,}
		\vphantom{\underbrace{\zm}_{\textstyle\text{$m$ blocks}}}		
	\end{equation}
	where \begin{equation}
		\label{eq:matrixJ}
		\vJ_{p}:=\left[\begin{array}{ccccc}
			-1 & 1 &&&\\
			& -1 &1 &&\\
			&& \ddots & \ddots &  \\
			&&  &-1 & 1 \\
		\end{array}\right] \in \RR^{(p-1)\times p}.
	\end{equation}
	In addition, define 
	\begin{equation}
		\label{eq:indexsetM}
		\begin{aligned}
			\mathcal{M}:=\{im_1| i=1,2,\dots, m_2-1\},& \quad \mathcal{M}^C:=\{1,2,\dots,m-1\}\backslash  \mathcal{M}, \\ n=(m- m_2){\overline{d}},& \quad \overline n=( m_2-1){\overline{d}},
		\end{aligned}
	\end{equation}
	and let  
	\begin{eqnarray}
		\label{eq:AandAbar}
		\overline{\vA}:=mL_f\cdot \vJ_{\mathcal{M}}\otimes \vI_{\overline{d}},\quad 
		\vA:=mL_f\cdot \vJ_{\mathcal{M}^C}\otimes \vI_{\overline{d}}, \quad \overline\vb = \vzero \in \RR^{\overline n}, \quad \vb = \vzero \in \RR^n,
	\end{eqnarray}
	where $\vJ_{\mathcal{M}}$ and $\vJ_{\mathcal{M}^C}$ are the rows of $\vJ_{m}$ indexed by $\mathcal{M}$ and $\mathcal{M}^C$, respectively. 
	Define $ g:\mathbb{R}^{\overline{n}}\rightarrow \mathbb{R}$ as
	\begin{equation}
		\label{eq:gbar}
		g(\vy):=\beta\|\vy\|_1=\max\left\{\vu^\top\vy~\Big|~\|\vu\|_\infty\leq \beta\right\}
	\end{equation}
	with  $\beta>0$.
	Let $f_0$ be a $\rho_f$-weakly convex function with $\rho_f>0$.
	Putting all the components given above, we obtain a specific instance of~\eqref{eq:model}, formalized as follows.
	\begin{definition}[instance $\cP$]
		\label{def:hardinstance}
		Given $\vareps\in(0,1)$ and $L_f > 0$, let $m_1, m_2$ and $\overline d$ be integers such that $m_1, m_2, \overline d \ge2$. We refer to as \emph{instance $\cP$} the instance of problem~\eqref{eq:model}, where $f_0$ is a $\rho_f$-weakly convex function with  $\rho_f>0$, $g$ is given in~\eqref{eq:gbar}, and $(\vA, \overline\vA, \vb, \overline\vb)$ is given in~\eqref{eq:AandAbar}. 
	\end{definition}
	
	The following lemma characterizes the joint condition number of  $\overline{\vA}$ and $\vA$ defined in \eqref{eq:AandAbar}.
	\begin{lemma}
		\cite[Lemma 3]{liu2025lowercomplexityboundsfirstorder}
		\label{lem:condH}
		Let $\vA$ and $\overline\vA$ be given in~\eqref{eq:AandAbar}. Then $\frac{m}{4} \le \kappa([\overline{\vA}; \vA])=\kappa(\vH)<m$.
	\end{lemma}

	\subsection{Comparisons with an ADMM}
	\label{sec:admm}
	In this subsection, we compare our results with the complexity results of the (linearized) ADMM established by Melo $\&$ Monteiro in~\cite{melo2017iteration3}. 
	Specifically, they aim to find an $\varepsilon$-KKT point of problem~\eqref{eq:model-spli}. If this point is a feasible point of problem~\eqref{eq:model}, it naturally qualifies as an $\varepsilon$-KKT point of problem~\eqref{eq:model}.
	In addition,
	they assume the boundedness of the feasible set and Assumption~\ref{ass:dual2}.

	To obtain the upper bound of the ADMM to produce an $\vareps$-KKT point of problem~\eqref{eq:model-spli}, we shall let the RHS of the four inequalities in \cite[Theorem 2.3]{melo2017iteration3} be no larger than $\varepsilon$. 
	We adopt the notations from \cite{melo2017iteration3}.
	With the substitutions 
	$f_1 \equiv {g}$, $f_2 \equiv 0, g=f_0, A_1=[-\vI ; 0], A_2=0, B=[\overline{\vA}; \vA], b=[\overline{\vb}; \vb]$, $\beta = B_2$, and $\mu_g= -L_f$
	into \cite[Problem (7)]{melo2017iteration3}, the worst-case complexity result of their ADMM is
	\begin{equation}\label{eq:bd-k}
		\max\left\{4 B_2^2\|\overline{\vA}\|^2 \varepsilon^{-2} \Delta\overline{\mathcal{L}}_0/\overline{\delta}_1,\ 2(\tau+\overline{\tau})^2\varepsilon^{-2}  \Delta\overline{\mathcal{L}}_0/\overline{\delta}_1 \right\}
	\end{equation}
	oracles, where $B_2$, $\Delta\overline{\mathcal{L}}_0$, $\overline{\delta}_1$, $\tau$, and $\overline{\tau}$ are some quantities in \cite{melo2017iteration3}.
	Specifically, $\Delta\overline{\mathcal{L}}_0$ is something like~$\Delta_F$ in our setting, and
	\begin{equation}\label{eq:beta-lb}
		B_2 > \frac{12\gamma_\theta(\tau+\overline{\tau})^2 + L_f^2}{{\lambda^+_{\min }\left([\overline{\boldsymbol{A}} ; \boldsymbol{A}][\overline{\boldsymbol{A}} ; \boldsymbol{A}]^\top\right)}(\tau-\overline{\tau}-L_f)},
	\end{equation}
	where $\gamma_\theta = \frac{\theta}{(1-|\theta-1|)^2} \ge 1$ with $\theta\in (0,2)$.
	In addition, they consider $\min_{\vx}\overline{f}(\vx)+g(\vy)+\frac{\beta}{2}\|[-\vI ; 0]\vy + [\overline{\vA}; \vA] \vx\|^2$ as the inner subproblem at each iteration. By applying an accelerated proximal gradient method to achieve an accuracy of $\vareps' >0$, the number of oracle calls required for the subroutine is $\Theta\left(\frac{\lambda_{\max }\left([\overline{\boldsymbol{A}} ; \boldsymbol{A}][\overline{\boldsymbol{A}} ; \boldsymbol{A}]^\top\right)}{\lambda^+_{\min }\left([\overline{\boldsymbol{A}} ; \boldsymbol{A}][\overline{\boldsymbol{A}} ; \boldsymbol{A}]^\top\right)} \log\frac{1}{\vareps'}\right)$.
	Pick $\overline\tau = L_f$, $\tau = 2(\overline\tau + L_f)$ and make $\Delta_1 = \overline\tau + L_f$ in \eqref{eq:bd-k}. This way, we have from \eqref{eq:bd-k} and \eqref{eq:beta-lb} that the worst-case complexity result of their ADMM becomes
	\begin{equation}\label{eq:bd-k-2}
		  \mathcal{O}\left( \kappa([\overline{\boldsymbol{A}} ; \boldsymbol{A}])\max\left\{L_f \left(\frac{\|\overline \vA\|}{\lambda^+_{\min }\left([\overline{\boldsymbol{A}} ; \boldsymbol{A}][\overline{\boldsymbol{A}} ; \boldsymbol{A}]^\top\right)} \right)^2 \varepsilon^{-2} \Delta_F,\ L_f\varepsilon^{-2}  \Delta_F \right\} \right).
	\end{equation}
	
	In contrast, under Assumption~\ref{ass:dual2}, we can directly apply the results from Remarks~\ref{rem:matchsp}(i) and~\ref{rem:matchp}(i). This establishes that the upper complexity bound of PG-RPD improves upon that of \cite{melo2017iteration3} by at least a factor of
	$
	\left( \frac{\|\overline{\mathbf{A}}\|}{\lambda^+_{\min}\bigl([\overline{\mathbf{A}}; \mathbf{A}][\overline{\mathbf{A}}; \mathbf{A}]^\top\bigr)} \right)^2.
	$
	
	\subsection{Comparisons with a linearized ALM}
	\label{sec:alm}
	In this subsection, we take the complexity results from \cite{zhang2022global} by Zhang $\&$ Luo as a benchmark for a linearized ALM. We then compare our results with theirs on the specific problem instance $\cal{P}$ given in Definition \ref{def:hardinstance}. Their analysis relies on the following assumption.
	\begin{assumption}
		\label{ass:polyhedral2}
		$g$ is an indicator set of a polyhedral set $P=\{\vx : \vG \vx \preceq \vh\}\subset \mathbb{R}^n$, where $\vG \in \mathbb{R}^{\ell \times n}$, and $\vh \in \mathbb{R}^{\ell}$. Here $\preceq$ means the coordinatewise $\leq$.  
	\end{assumption}
	
		We adopt the notations from \cite{zhang2022global}.
	Under Assumption \ref{ass:polyhedral2}, they present in \cite[Theorem 2.4]{zhang2022global} that their linearized ALM find a  $\varepsilon$-KKT point with
	$$
	k\leq \vareps^{-2} C \max\left\{  B_1, B_3
	\right\},
	$$
	outer iterations to achieve an $\vareps$-KKT point of problem \eqref{eq:model}, 
	where $C$, $B_1$ and $B_3$ are {quantities} defined in their paper. 
	Specifically, they let
	$$
	C=\left(\phi^0-\underline{f}\right) \cdot \max \{4 c, 2 / \alpha,  \beta / L_f\},\,\,
	B_1= \left(1+\sigma_{\max }(\vA) \frac{1+2c L_f}{2 c L_f}\right)^2, \text{ and }
	$$
	$$
	B_3:= \left(\left(3L_f+L_f \sigma_{\max }(\vA)^2+2 / c\right)+L_f \sigma_{\max }(\vA) \sqrt{B_1}+3L_f\right)^2,
	$$
	where  $c>0$, $\alpha>0$,   $\beta>0$,
	$\phi^0-\underline{f}$
	is something like 
	$\Delta_F$ in our setting, and
	$\sigma_{\max }(\vA)$ denotes the largest singular value of $\vA$.
	In Section 4.3 of their paper, their complexity analysis requires that 
	$$
	c<1 /\left(4 L_f+L_f \sigma_{\max }^2(A)\right), \,\, \alpha<\frac{c L_f^2}{\sigma_{\max }(A)^2},\,\,\beta<\min \left\{\frac{1}{30}, \frac{\alpha}{12 p  {\sigma}^2_5}\right\},
	$$
	and $\sigma_5>0$ is an error bound constant satisfying
	$$\sigma_5=\frac{\sqrt{2}\left(\overline{\theta}\left(L_f+L_f \sigma_{\max }(A)+3L_f\right)^2+1\right)}{2L_f},$$
	where
	$
	\overline{\theta}=\max _{\overline{\mathbf{M}} \in \mathcal{B}(\mathbf{M})} \sigma_{\max }^2(\overline{\vM}) / \sigma_{\min }^4(\overline{\vM})
	$
	with $\mathcal{B}(\mathbf{M})$ being the set of all submatrices of $\mathbf{M}$ with full row rank, and
	$
	\mathbf{M}=\left[\begin{array}{cc}
		\vA^\top & \vG^\top \\
		0 & \vI
	\end{array}\right].
	$
	
	By  $\max\{B_1, B_3\}\ge B_3$,
	$B_3 \ge L_f^2 \sigma_{\max }(\vA)^4$, and $C\ge\left(\phi^0-\underline{f}\right) \cdot 2 / \alpha$, the number of iterations given below is a worst-case complexity results of their methods:
	\begin{equation}
		\label{eq:mink}
		 \mathcal{O}(\Delta_F L_f  \sigma^6_{\max }(\vA) \vareps^{-2}).
	\end{equation}

	Now, we apply \eqref{eq:mink} to the hard instance $\cal{P}$ given in Definition \ref{def:hardinstance}.
	By the definition of $\vA$ in \eqref{eq:AandAbar}, the smallest positive eigenvalue $\lambda_{\text {min }}^+$ of $\vA\vA\zz$ is  $4m^2 L_f^2 \sin^2(\frac{\pi}{2m_1})$, which is roughly $m_2^2 L_f^2$, and the norm of $\|\vA\|$ equals $ 2m L_f \sin(\frac{(3m_1-1)\pi}{6m_1}) \approx 2m L_f$. From Lemma \ref{lem:condH}, the condition number of $[\overline{\vA}; \vA]$ is approximately~$m$.
	Hence it holds 
	\begin{equation}
		\label{eq:normA}
		\sigma_{\max }(\vA)= \|\vA\|  \approx  2\kappa([\overline{\vA};\vA]) L_f .
	\end{equation}
	Substituting this into \eqref{eq:mink}, the upper complexity bound in \cite{zhang2022global} can be higher than our results in Remarks~\ref{rem:matchsp}(ii) and~\ref{rem:matchp}(ii) by at least  a factor of $\kappa^5([\overline{\vA};\vA]) L_f^6$.

			\end{document}